\documentclass[10pt]{amsart}
\usepackage[hidelinks]{hyperref}
\usepackage{doi,url}
\usepackage[utf8]{inputenc}
\usepackage{amsmath}
\usepackage{amssymb} % serve
\usepackage[numbers]{natbib}
\DeclareRobustCommand{\noopsort}[1]{}%to handle reference sorting of bibliography entries with nobility particles (e.g van)

\usepackage{prettyref}
\usepackage{graphicx}
\usepackage{enumitem}
\usepackage{mathtools}
\usepackage{comment}

\numberwithin{equation}{section}

\newtheorem{theorem}{Theorem}[section]

\newtheorem{proposition}[theorem]{Proposition}
\newtheorem{corollary}[theorem]{Corollary}

\newtheorem{lemma}[theorem]{Lemma}
\newtheorem{fact}[theorem]{Fact}
\newtheorem{question}[theorem]{Question}

\newcounter{introthmcounter}
\newtheorem{thmA}{Theorem}[introthmcounter]

\theoremstyle{definition}

\newtheorem{definition}[theorem]{Definition}
\newtheorem{notation}[theorem]{Notation}

\newtheorem{example}[theorem]{Example}

\newtheorem{context}[theorem]{Context}

\theoremstyle{remark}
\newtheorem{remark}[theorem]{Remark}

\newcommand{\cA}{\mathcal{A}}

\newcommand{\bE}{\mathbb E}
\newcommand{\bQ}{\mathbb Q}
\newcommand{\bR}{\mathbb R}
\newcommand{\bN}{\mathbb N}
\newcommand{\bK}{\mathbb K}

\newcommand{\bZ}{\mathbb Z}

\newcommand{\Var}{\mathrm{Var}}
\newcommand{\RCF}{\mathrm{RCF}}

\DeclareMathOperator{\Supp}{\mathrm{Supp}}

\DeclareMathOperator{\Hom}{\mathrm{Hom}}

\DeclareMathOperator{\sign}{\mathrm{sgn}} %sign

\DeclareMathOperator{\std}{\mathrm{st}}
\DeclareMathOperator{\Th}{\mathrm{Th}}
\DeclareMathOperator{\End}{\mathrm{End}}
\DeclareMathOperator{\id}{\mathrm{id}}

\DeclareMathOperator{\RExp}{\mathrm{RExp}}

\newcommand{\bU}{\mathbb{U}}

\newcommand{\cQ}{\mathcal{Q}}
\newcommand{\cS}{\mathcal{S}}

\def\lc{\mathrm{lc}} %\leading coefficient
\def\Lhp{(\!(} %tighter \Lhp
\def\Rhp{)\!)} %tighter \Rhp

\def\op{\mathrm{op}}

\def\cB{\mathcal{B}}
\def\cP{\mathcal{P}}
\def\cF{\mathcal{F}}
\def\cG{\mathcal{G}}
\def\cC{\mathcal{C}}

\def\cH{\mathcal{H}}
\def\cR{\mathcal{R}}

\def\cL{\mathcal{L}}

\newcommand{\cT}{\mathcal{T}}
\def\-{\text{-}}

\newcommand{\No}{\mathbf{No}}

\newcommand{\bT}{\mathbb{T}}

\newcommand{\tp}{\mathrm{tp}}

\newcommand{\dcl}{\mathrm{dcl}}
\newcommand{\convex}{\mathrm{convex}}

\newcommand{\tame}{\mathrm{tame}}

\newcommand{\var}[1]{\mathrm{#1}}
\newcommand{\val}{\mathbf{v}}
\newcommand{\res}{\mathbf{r}}
\newcommand{\rv}{\mathbf{rv}}
\newcommand{\CH}{\mathrm{CH}}
\newcommand{\cO}{\mathcal{O}}
\newcommand{\cE}{\mathcal{E}}
\newcommand{\co}{\mathfrak{o}}
\newcommand{\fM}{\mathfrak{M}}
\newcommand{\fN}{\mathfrak{N}}
\newcommand{\fm}{\mathfrak{m}}
\newcommand{\fp}{\mathfrak{p}}
\newcommand{\fn}{\mathfrak{n}}

\newcommand{\Exponents}{\mathrm{Exponents}}

\DeclareMathOperator{\Span}{\mathrm{Span}}
\DeclareMathOperator{\Pos}{\mathrm{Pos}}

\DeclareMathOperator{\HSder}{\mathrm{H\!S}}

\title[Truncations and the Mourgues-Ressayre construction]{Truncations in languages of generalized power series and the structure of $T$-$\lambda$-spherical completions of o-minimal fields.}
\author{Pietro Freni}
\address{Institute of Mathematics, Czech Academy of Sciences, Žitná 25, 110 00 Praha 1, Czech Republic}


\thanks{This work has been partially supported by the Czech Science Foundation grant 26-22545I, by the Czech Academy of Sciences CAS (RVO 67985840), and by a Scholarship of the School of Mathematics of the University of Leeds.}

\begin{document}
	
\begin{abstract}
	Let $T$ be the theory of an o-minimal field and $T_0$ a common reduct of $T$ and $T_{an}$.
		
	I adapt Mourgues' and Ressayre's constructions to deduce structure results for $T_0$-reducts of $T$-$\lambda$-spherical completion of models of $T_\convex$.
		
	These in particular entail that whenever $T$ is the theory of a reduct of $\bR_{an,\exp}$ defining the exponentiation (e.g.\ $T=T_{\exp}$, the theory of the field of reals expanded by the exponential function), every model of $T$ has an initial elementary embedding in $\No$. This answers positively an open question in \cite{ehrlich2021surreal}.
		
	The main technical result is that expanding an integral domain of generalized series in the sense of Hahn-Higman-Ribenboim (such as a Hahn field) by a family of generalized power series interpreted as functions defined on certain infinitesimal elements, has the property that truncation closed subsets generate truncation closed substructures, provided that the family of generalized power series is itself closed under truncations and partial derivatives. It is also shown that the further closure of the generated set under solutions to certain equations is as well closed under truncations.
    
    The formal results on power series leave room for possible generalizations to the case in which $T_0$ is power bounded but not necessarily a reduct of $T_{an}$.
\end{abstract}
	
\maketitle
\tableofcontents
	
\section{Introduction}
		
\subsection{Motivation and background}

Recall that given a ring $\bK$ and a partially ordered monoid $(\fM, \cdot, <)$ the ring of generalized series à la Hahn-Higman-Ribenboim, $\bK\Lhp \fM\Rhp$ consists of the formal series $f:=\sum_{\fm} k_\fm \fm$ whose support $\Supp(f):=\{\fm \in \fM : k_\fm \neq 0\}$ is Noetherian (i.e.\ contains no infinite antichain and no infinite ascending chain). The field operations are given by the term-wise sum and Cauchy's product formula.
		
A distinctive feature of these objects is that they have an extra notion of \emph{infinite sum} that is useful in applications, for example it allows for a natural interpretation of power series at infinitesimal elements (cf \cite{neumann1949ordered}). 

Subrings of $\bK \Lhp \fM \Rhp$ are usually referred to as \emph{rings of generalized series} and can inherit the extra \emph{summability structure} of the ambient $\bK \Lhp \fM \Rhp$.

When $\bK$ is an ordered field and $\fM$ is a totally ordered group, $\bK\Lhp \fM \Rhp$ is a field, called a Hahn field and is itself naturally ordered stipulating that a series is positive if its leading coefficient is positive.

Fields of generalized series played an important role in the study of o-minimal fields (i.e.\ o-minimal expansions of ordered fields), see for example \cite{Ressayre1993Integer}, \cite{dries1994elementary} or, more recently, \cite{rolin2024transasymptotic}.
		
The study of their relation with o-minimal fields follows several lines of investigation. The ones that will be addressed in this paper can be broadly synthesized in the following two loose questions:
\begin{enumerate}[label=(Q\arabic*), ref=(Q\arabic*)]
	\item\label{loosequestion1} Let $\bK$ be an o-minimal field, and $\bE \subseteq \bK \Lhp \fM \Rhp$ some field of generalized series containing $\bK$, when does $\bE$ admit an expansion to an elementary extension of $\bK$?
	\item\label{loosequestion2} Let $\bK$ be an o-minimal structure, and let $\bE \subseteq \bK \Lhp \fM \Rhp$ be given a structure of elementary extension of $\bK$, what elementary extensions of $\bK$ admit \emph{truncation closed} elementary embeddings into $\bE$?
\end{enumerate}
		
By a truncation of $f \in \bK \Lhp \fM \Rhp$ we mean an element of the form $f|\fm:= \sum_{\fn>\fm} k_\fn \fn$ and \emph{truncation closed} means closed under taking truncations of the elements.
		
The fields of generalized series $\bE \subseteq \bK \Lhp \fM \Rhp$ for which the inherited summability structure is more relevant are obtained by restricting the family of allowed supports in the definition of Hahn field from all well ordered sets to some suitable ideal $B$ of subsets of $\fM$ (cf \cite{berarducci2021value}, \cite{krapp2022rayner}), such subfields will be denoted by $\bK \Lhp \fM\Rhp_B$. These are the fields for which also the two questions above seem more relevant.
		
If $B$ consists of the well ordered subsets with cardinality strictly less then some uncountable cardinal $\lambda$, $\bK\Lhp \fM \Rhp_B$ is denoted by $\bK \Lhp \fM \Rhp_\lambda$ and is called a \emph{$\lambda$-bounded} Hahn field.
		
\smallskip
		
Question \ref{loosequestion1} is known to have positive answer in the case $\bE = \bK \Lhp \fM \Rhp$ and $\bK$ is power-bounded whenever $\fM$ is a vector space over the exponents of $\bK$ (this is a consequence of the so called residue-valuation property, see \cite[Sec.s~9 and 10]{dries2000field}, \cite[Ch.~12 and 13]{tyne2003t} and \cite[]{kaplan2023t}, or \cite[Sec.~4]{freni2024t} for a self-contained treatment).
		
On the contrary, if $\bK$ is exponential, \cite[Thm.~1]{kuhlmann1997exponentiation}, entails that no field of the form $\bK \Lhp \fM \Rhp$ admits an expansion to an elementary extension of $\bK$.
		
It is also known that for suitably constructed pairs $(\fM, B)$, $\bR \Lhp \fM \Rhp_B$ can naturally be expanded to elementary extension of $\bR_{an, \exp}$: most notable examples are the field of LE and EL transexponential series and certain constructions of $\lambda$-bounded Hahn fields (cf \cite{dries2001logarithmic}, \cite{kuhlmann2005kappa}). %check what of these satisfy T4
		 
Such naturally constructed extensions often enjoy the following two properties of compatibility of the exponential with the ``serial'' structure: 
\begin{enumerate}[label = \textbf{(D)}, ref=(D), align= left, leftmargin = *, itemindent = !]\label{Intro:D&T4}
	\item\label{Intro:D} $\log(\fM) \subseteq \bK \Lhp \fM^{>1} \Rhp_B$ (cf \cite{Ressayre1993Integer})
\end{enumerate}
\begin{enumerate}[label = \textbf{(T4)}, ref=(T4), align= left, leftmargin = *, itemindent = !]
	\item\label{Intro:T4} for every sequence of monomials $(\fm_n)_{n \in \bN}$ with $\fm_{n+1} \in \Supp(\log \fm_n)$ for all $n \in \bN$, there is $N\in \bN$ such that for all $n\ge N$, $\log \fm_n = c_n \pm\fm_{n+1}$ for some $c_n$ with $\Supp c_n>\fm_{n+1}$ (cf \cite{schmeling2001corps}).
\end{enumerate}

In that regard \ref{loosequestion1} can be further specialized to the question of whether such structure can be chosen so that it satisfies this two extra conditions.
		
\smallskip
		 
Answers to the \ref{loosequestion2} have instead been given in \cite{mourgues1993every} for real closed fields and in \cite{fornasiero2013initial} and \cite{ehrlich2021surreal} respectively for $\bK=\bR_{an, \exp}$ and $\bK=\bR_{W, \exp}$ with $W$ a convergent Weierstrass system and $\bK \Lhp \fM \Rhp_B= \No$ the (class-sized) field of surreal numbers with its natural $(W,\exp)$ structure.
		
More recently, in \cite{rolin2024transasymptotic}, Rolin, Servi, and Speissegger, obtain related results for certain Generalized Quasianalytic Algebras as defined in \cite{rolin2015quantifier}. Specifically, for $\cA=an^*$ or $\cA$ a truncation closed and natural GQA containing the restricted exponential (see \cite[Def.~3.4]{rolin2024transasymptotic}), they explicitly construct a truncation-closed ordered differential field embedding of the Hardy field $\cH(\bR_{\cA, \exp})$ of the o-minimal structure $\bR_{\cA, \exp}$, in the field $\bT$ of transseries. Such embedding is also an $L_{\cA, \exp}$-embedding if $\bT$ is given a suitable natural $L_{\cA}$-structure (\cite[Sec.~4.2]{rolin2024transasymptotic}).
		
\smallskip
		
In \cite[Thm.~B]{freni2024t}, the author showed that every $T$-convexly valued o-minimal field, admits for every cardinal $\lambda$ a so-called \emph{$T$-$\lambda$-spherical completion}, that is, a unique-up-to-non-uinque-isomorphism elementary extension that is prime (i.e.\ weakly initial) among all the $\lambda$-spherically complete elementary extensions, and that such completion preserves the residue field.
		
This provides leverage toward further partial answers to \ref{loosequestion1}: it is not hard to see that given any $T$-convexly valued $(\bE, \cO) \models T_\convex$, for $\lambda$ large enough, the real closed reduct of a $T$-$\lambda$-spherical completion has the form $\bK \Lhp \fM \Rhp_\lambda$ with $\bK$ the residue field of $(\bE, \cO)$ and $\fM$ some ordered group.
		
The present paper aims at providing further partial answers towards the two line of investigation \ref{loosequestion1} and \ref{loosequestion2}. In that regard we will obtain, as corollaries of Theorem~\ref{introthm:C} below, that in the case $T$ defines exponentiation:
		
\begin{enumerate}[label=(C\arabic*), ref=(C\arabic*)]
	\item\label{intro:C1} for every large enough cardinal $\lambda$, the reduct of the $T$-$\lambda$-spherical completion of $(\bE, \cO)$ to the language of valued exponential field is isomorphic to a field of the form $\bK \Lhp \fM \Rhp_\lambda$ endowed with an exponential satisfying (D) and (T4);
	\item\label{intro:C2} if $T$ is the theory of a reduct of $\bR_{an, \exp}$ defining $\exp$ and $+$, then every model of $T$ admits initial elementary embeddings into the field $\No$ of surreal numbers. %TODO:need to say what initial mean somewhere
\end{enumerate}

Along the way of proving Theorem~\ref{introthm:C}, we will lay out some basic theory of truncations in generalized power series and how they behave with respect to compositions and compositional inversion (Theorems~\ref{introthm:A}, \ref{introthm:B}, Proposition~\ref{prop:main_for_bN}), and do so in a generality that certainly is not needed in the present paper but might be of interest per se or in future applications given that recently the notion of truncations sparked some new interest (cf \cite{dries2025truncation} and \cite{rolin2024transasymptotic}).
		
\subsection{Setting and Main Results}

Let $\bK$ be a power-bounded o-minimal structure with field of exponents $\Lambda\subseteq \bK$ and let $\fM$ be a multiplicatively written ordered $\Lambda$-vector space.
Recall that an $n$-varied generalized power series with coefficients from $\bK$ and exponents of $\Lambda$ is an element of $\bK \Lhp \var{x}^\Lambda\Rhp$ where $\var{x}^\Lambda$ is the multiplicatively written free $\Lambda$-vector space generated by some set of $n$ distinct variables $\var{x}_0, \ldots, \var{x}_{n-1}$, partially ordered by stipulating that $\var{x}_i<1$ for all $i<n$. More explicitly an $f \in \bK \Lhp \var{x}^\Lambda\Rhp$ is a formal expression of the form 
\[f(\var{x}) = \sum_{\gamma \in \Lambda^n} c_\gamma \var{x}^\gamma =\sum_{\gamma\in \Lambda^n} c_\gamma \prod_{i<n} \var{x}_i^{\gamma_i}
\]
where $(c_\gamma)_{\gamma \in\Lambda^n} \in \bK^\Lambda$, such that $\{\gamma: c_\gamma \neq 0\}$ is a well-partial order in $\Lambda^n$ with the product partial order.
Each such $f$ can be interpreted as a function on the positive infinitesimals of the Hahn field $\bK \Lhp \fM \Rhp$, by defining its value at tuples of the form $(k_i\fm_i)_{i<n}$ with $k=(k_i)_{i<n} \in (\bK^{>0})^n$ and $\fm=(\fm_i)_{i<n} \in (\fM^{\prec 1})^n$
as the formal sum $\sum_{\gamma} c_\gamma k^\gamma \fm^{\gamma}$ and extending its evaluation to any $n$-tuple $(x_i)_{i<n}$ of positive infinitesimals, by writing it as $x_i=\fm_i(k_i+\varepsilon_i)$ where $\varepsilon_i$ is infinitesimal, $k_i \in \bK^{>0}$ and $\fm_i\in \fM^{\prec 1}$ and evaluating $f$ by using a formal Taylor expansion at $\fm_i k_i$
\[f(x)=\sum_{\gamma}\sum_{m \in\bN^n} \prod_{i<m}\frac{(\varepsilon_i\fm_i)^{m_i}}{m_i!} \cdot (\partial_{\var{x}_{0}}^{m_0} \cdots \partial_{\var{x}_{n-1}}^{m_{n-1}} f)(k_0 \fm_0, \ldots, k_{n-1} \fm_{n-1}),\]
where $\partial_{\var{x}_i}f$ is the formal derivative in the variable $\var{x}_i$, defined on monomials by $\partial_{\var{x}_i}\var{x}_j=0$ for $j\neq i$, $\partial_{\var{x}_i}\var{x}_i^{\gamma_i}=\gamma_i \var{x}_i^{\gamma_i-1}$ and extended so as to be $\bK$-linear and to preserve infinite sums.

The first main result, on which Theorem~\ref{introthm:C} relies is a purely formal fact about Hahn fields expanded with a set $\cF$ of such series. Namely:
			
\begin{thmA}[\ref{cor:mainT}]\label{introthm:A}
	If the family $\cF$ is closed under taking truncations in every variable and under the operation $f(\var{x}_0, \ldots \var{x}_{n-1}) \mapsto \var{x}_i \partial_{\var{x}_i} f(\var{x}_0, \ldots, \var{x}_{n-1})$, and $X \subseteq \bK \Lhp \fM \Rhp$ is closed under truncations, then the smallest $(\cF, +, \cdot)$-structure generated by $X \cup \bQ$ is closed under truncations.
\end{thmA}

This will come as a corollary of a more general result concerning families of generalized series lying in rings of series of the form $\bK \Lhp \fM \times \var{x}^{\Lambda}\Rhp$ where:
\begin{enumerate}[label=(H\arabic*), ref=(H\arabic*)]
    \item\label{hypothesis:monomials_setting} $R^{\ge 0} \subseteq \Lambda \subseteq R$ for some (possibly partially) ordered ring $R$ such that $(R,+,0)$ has no torsion $R^{> 0}$ contains no zero-divisors, $R/pR$ is a $p$-boolean algebra for all prime numbers $p$, and $\Lambda$ is a subsemiring of $R$; $\var{x}^\Lambda$ is the free $\Lambda$-module on some finite set of variables $\var{x}$ with the minimal $\Lambda$-module order satisfying $\var{x}<1$; $(\fM, \cdot, 1, \le, (-^\gamma)_{\gamma \in \Lambda})$ is a $\Lambda$-submodule of some ordered $R$-module $\fN$, and $\lambda\mapsto \lambda \cdot \fm$ is one-to-one for all $\fm \le 1$; $\fM \times \var{x}^{\Lambda}$ is given the product partial order;
    \item\label{hypothesis:coefficients_setting} $(\bK, +, \cdot, 0, 1)$ is an integral domain equipped with a multiplicative submonoid $\bK^\bullet$, a semiring action of $\Lambda$ on the monoid $(\bK^{\bullet}, \cdot)$ and semiring homomorphism $\iota: \Lambda \to \bK$; furthermore we assume that $1+\Lambda\supseteq \Lambda^{\neq 0}$ or that all elements of $\bK^{\bullet}$ are units of $\bK$.
\end{enumerate}

Given a series $f \in \bK \Lhp \fM \times \var{x}^\bN\var{y}^\Lambda\Rhp$ it is then possible to define its (renormalized) formal derivatives (resp.\ renormalized formal Hasse-Schmidt derivatives) and, in a fashion similar to the one described for the simpler case above, to define the formal substitutions of any of the $\var{x}$-variables with any power series $g\in \bK\Lhp \fM \times \var{z}^\Lambda\Rhp$ with support smaller smaller than $1$, and of any of the $\var{y}$-variables with a series whose support has a maximum element smaller than $1$ and whose leading coefficient is in $\bK^\bullet$.

\begin{thmA}[\ref{thm:mainT}]\label{introthm:B}
    In the setting of \ref{hypothesis:monomials_setting} and \ref{hypothesis:coefficients_setting}. Given any family of series
    \[\cF\subseteq \bigcup\{\bK\Lhp \fM \times \var{x}^\Lambda \Rhp: \var{x}\;\text{finite set of variables}\},\]
    which contains $\bK$, $\fM \times \var{x}^\Lambda$, $\{(k+\var{t})^\lambda: \var{t}\in\Var, k \in \bK^\bullet, \lambda \in \Lambda\}$, and is closed under truncations, under renormalized Hasse-Schmidt derivatives of any order, and under monomial division, the following hold:
    \begin{enumerate}
        \item the closure of $\cF$ under $\cdot, +$, and allowed substitutions is closed under truncations;
        \item the closure of $\cF$ under $\cdot, +$, allowed substitutions, and implicit functions is closed under truncations.
    \end{enumerate}
\end{thmA}

We refer to Definition~\ref{def:rn_formal_derivatives} for the precise definition of the renormalized Hasse-Schmidt formal derivatives, to Definitions~\ref{defn:genps_interpretation} and Remark~\ref{rmk:compositions} for the one of the allowed substitutions, and to Definition~\ref{defn:closure_ppties} for the meaning of closure under implicit functions and under monomial division.
It is worth mentioning here that Theorem~\ref{introthm:B} generalizes several known facts about preservation of the trucnation-closedness property, under various closure opeartions (cf \cite[Thm.~1.2(1) and 1.3]{dries2014truncation}, \cite[Thm~4.15]{fornasiero2006embedding}, \cite[Lem.~3.5]{mourgues1993every}, see Corollaries~\ref{cor:dries_composition} and \ref{cor:dries_fornasiero_composition} below), and that it actually hinges on a more fundamental analogous result (Proposition~\ref{prop:main_for_bN}), which has quite simpler settings and hypotheses.

It is also worth observing that when $\Lambda$ is a ring, the hypothesis of closure under monomial division as well as the one of containing $\{(k+\var{y})^{\lambda}: \var{y} \in \Var, k \in \bK^\bullet, \lambda \in \Lambda\}$ are unnecessary, and that if furthermore $\bK$ is a $\bQ$-algebra, then the closure under renormalized Hasse-Schmidt derivatives amounts to closure under formal derivatives (Remark~\ref{rmk:ring-simplied_hypthoesis}).

\medskip

With the aim of applying Theorem~\ref{introthm:A}, and motivated by the notion of Generalized Quasianalytic Algebra (abbreviated GQA) of \cite{rolin2015quantifier}, we will then consider power bounded structures $\bK$ with a \emph{piecewise Skolem theory}\footnote{by this I mean a universally axiomatized model complete theory, see Definition~\ref{defn:piecewise_Skolem}.}%by this we mean a theory $T$ in some language $L$ such that for each tuple of variables $\var{y}$, each variable $\var{x}$, and each $L$-formula $\varphi(\var{x}, \var{y})$, there are finitely many $L$-terms $t_0(\var{y}), \ldots, t_{n-1}(\var{y})$, such that $\exists \var{x}, \;\varphi(\var{y}, \var{x})$ is equivalent modulo $T$ to $\bigvee_{i<n} \varphi(\var{y}, t_i(\var{y}))$. See Definition~\ref{defn:piecewise_Skolem}.}
, given by certain families $L:=\{L(r) : n \in \bN, r \in (\bK^{>0})^n\}$ of algebras $L(r)$ of functions on $\prod_{i<n} (0, r_i)$ that satisfy a relativized analyticity condition. We will suppose furthermore that they are endowed with a family of differential algebra morphisms $\cT:=(\cT_n)_{n \in \bN}$ allowing to represent the germs at the origin of functions in $L(r)$ by generalized power series. This will allow to attempt a natural interpretation of the functions in each $L(r)$ on $\bK \Lhp \fM \Rhp$ for $\fM$ a multiplicatively written ordered $\Lambda$-vector space, by setting for $x$ a tuple in $\bK$ and $\varepsilon$ a tuple of infinitesimal elements in $\bK \Lhp \fM\Rhp$, $f(x+\varepsilon)=\cT(g)(\varepsilon)$ where $\cT(g)$ is the generalized power series representing the germ at the origin of the function $g(\var{z}):=f(x+\var{z})$.

We will denote by $\bK \Lhp \fM\Rhp^\cT$ the expansion of $\bK\Lhp \fM \Rhp$ to an $L$-structure given by the interpretations above and say that $(\bK, \cT)$ is a \emph{serial power bounded structure} if $\bK \Lhp \fM \Rhp^\cT$ is an $L$-elementary extension of $\bK$ for all $\fM$ (see Definition~\ref{def:serial}).
			
The main example of serial structure is $\bR_{an}$, but it is easy to show that all reducts of $\bR_{an}$ can be defintionally expanded to serial structures (Corollary~\ref{cor:anreducts_serial}).
Although the definition of serial power bounded structure is modeled upon the definition of Generalized Quasianalytic Algebra (GQA) in \cite{rolin2015quantifier}, we leave open the question of whether all expansion of the reals by a GQA is interdefinable with a serial structure (\ref{quest:GQAserial}).
			
If the algebras of generalized power series used to interpret these germs are \emph{closed under truncations}, Theorem~\ref{introthm:A} then ensures that each $\bK \Lhp \fM \Rhp^{\cT}$, has the property that whenever $\bK \preceq \bE \preceq \bK \Lhp \fM \Rhp^\cT$ is truncation closed and $x \in \bK \Lhp \fM\Rhp$ has all proper truncations in $\bE$, the definable closure $\bE\langle x \rangle$ of $\bE \cup \{x\}$ is still truncation closed (Proposition~\ref{prop:tc-ppty}).
			
This fact together with results from \cite{freni2024t} allows to redo Mourgues' and Ressayre's constructions in \cite{mourgues1993every} and \cite{Ressayre1993Integer} of truncation closed embeddings (satisfying \ref{Intro:D} and (T4) in the case the models in question are expanded with a compatible exponential).
We will need a minor modification of their construction to take into account the fact that we are aiming in general at embeddings into the $\lambda$-bounded versions of the Hahn fields.

The setting for our final result will be a given \emph{serial} power-bounded $(\bK_L, \cT)$ and an o-minimal expansion $\bK$ of $\bK_L$ in some language $L'\supseteq L$. The theories of $\bK_L$ and $\bK$ are denoted respectively by $T$ and $T'$.

\begin{thmA}[\ref{thm:MRanalogue} and \ref{thm:final}]\label{introthm:C}
	Let $\bE$ be a proper tame extension of $\bK\models T'$, $\lambda$ be a large enough cardinal, and $\bE_\lambda$ be the $T'$-$\lambda$-spherical completion of $(\bE, \CH(\bK))\models T'_\convex$. 

	Then there is a $L$-isomorphism $\eta: \bE_\lambda \to \bK \Lhp \fM \Rhp_\lambda^\cT$ over $\bK_L$, such that $\eta(\bE)$ is truncation-closed. Furthermore if $T'$ defines an exponential, then $\eta$ can be constructed so that $\exp_* := \eta \circ \exp \circ \eta^{-1}$ satisfies \ref{Intro:D} and (T4).
\end{thmA}
		
As suggested by Mantova, combined with \cite[Thm~8.1]{ehrlich2021surreal}, this can be used to prove \ref{intro:C2} above (Corollary~\ref{cor:initial}).
		
\smallskip

Theorem~\ref{introthm:C} and the above mentioned \cite[Main Theorem]{rolin2024transasymptotic} are related, but differ in several regards. To explain the differences consider the specialization of Theorem~C to the case in which $T'$ is the theory of the expansion $\bR_{0, \exp}$ by the unrestricted exponential, of a \emph{serial} polynomially bounded structure $\bR_0$ over the reals, already defining the restricted exponential. In that case, Theorem~C entails that every elementary extension $\bE \succeq \bR_{0, \exp}$ has a truncation-closed \emph{elementary} embedding (over $\bR_{0,\exp}$) in some elementary extension of the form $\bR_{0, \exp}\Lhp \fM \Rhp_\lambda$ where the functions definable in the structure $\bR_0$ are interpreted in a natural way and the exponential is interpreted in such a way that \ref{Intro:D} and (T4) are satisfied.
		
Note in particular that Theorem~\ref{introthm:C} is conditional to the seriality hypothesis on $\bR_{0}$ (which, for now, we only know to hold in the case $\bR_0$ is a reduct of $T_{an}$), that there is no required compatibility with derivatives (as one could for example require if $(\bE, \CH(\bR))$ is expanded to a $T$-convex $T$-differential field as defined in \cite{kaplan2023t}), and that the group $\fM$ arises from a rather abstract, and not very explicit, completion machinery.
		
On the contrary, \cite[Main Theorem]{rolin2024transasymptotic}, concerns the case $\bE:=\bR_{\cA, \exp}\langle t \rangle$ for some $t>\bR$ and gives an \emph{explicit} construction for a truncation-closed \emph{differential} embedding with respect to the natural ``derivation at $t$'', into the (explicitly constructed) classical field of transseries $\bT$. Moreover it does not require the seriality of $\bR_\cA$, and only requires an extra minor hypothesis (being $an^*$ or being natural) on the GQA $\cA$.
		
\subsection{Structure of the paper} 

The first, longer, Section~\ref{sec:main} deals with the technical results on formal power series, that is, Theorems~\ref{introthm:A} and \ref{introthm:B}.

Subsection~\ref{ssec:wpocombinatorics} is concerned with some basic (mostly known) facts about well partial orders.
Subsections~\ref{ssec:series} and \ref{ssec:Gseries} review the definition of rings of generalized series in the sense of Hahn-Higman-Ribenboim.
Subsection~\ref{ssec:rings_with_powers} introduces some terminology and gives some basic facts concerning the hypothesis on $\Lambda$ and $\bK$ in Theorem~\ref{introthm:B}. Subsections \ref{ssec:ders} and \ref{ssec:compositions}, deal with the set-up of Theorem~\ref{introthm:B}, giving the various notions of formal derivatives and of allowed compositions, as well as laying out what are the needed contexts for all definitions.
Subsection~\ref{ssec:mainres_tc_statements} gives the statements of Theorems~\ref{introthm:A} and \ref{introthm:B}, as well as their proof from a more basic fundamental result (Theorem~\ref{thm:mainT0}).
Finally, subsection~\ref{ssec:weakly_restricted_power_series} is dedicated to the proof of Theorem~\ref{thm:mainT0} and contains most of the key arguments of the section.

Section~\ref{sec:serial} introduces \emph{serial} power-bounded structures (Definition~\ref{def:serial}) giving some examples and the main application (Proposition~\ref{prop:tc-ppty}) of Theorem~\ref{introthm:A}.

Section~\ref{sec:Spherical_completion_review} reviews the results and definitions from \cite{dries2000field}, \cite{tyne2003t}, and \cite{freni2024t} relevant for the last section.

Finally Section~\ref{sec:MourguesRessayre} is dedicated to reviewing Ressayre's construction in \cite{Ressayre1993Integer} in light of the new ingredients available. In particular it contains the last main result, Theorems~\ref{introthm:C}, as well as the answer to an open question in \cite{ehrlich2021surreal}.

\subsection{Acknowledgments}
A preliminary version of this paper was part of a PhD project at the University of Leeds and is supported by a Scholarship of the School of Mathematics. The current version, with the new more general versions of Theorem~\ref{introthm:B}, the treatment of truncations of implicit functions, and the improvement of Corollary~\ref{cor:initial} was supported by the Czech Science Foundation grant 26-22545I and by the Czech Academy of Sciences CAS (RVO 67985840).
    
The author is very grateful to his supervisors Vincenzo Mantova for his support, constant feedback and for suggesting Corollary~\ref{cor:initial}, Dugald Macpherson for his support and feedback, and to Lou van den Dries for some remarks about the previous version which was included as Chapter~4 in the authors Doctoral dissertation \cite{wreo36168}.

\section{Truncations in languages of power series}\label{sec:main}
	
This section deals with the technical results about languages of generalized power series and their truncations. In particular, Subsections~\ref{ssec:mainres_tc_statements} and \ref{ssec:weakly_restricted_power_series} give the statements and proof of the Theorems~\ref{introthm:A} and \ref{introthm:B} of the introduction. %The first two Subsections~\ref{ssec:wpocombinatorics} and \ref{ssec:series} review some basic theory of truncations in rings of generalized power series, subsections~\ref{ssec:Gseries}, \ref{ssec:ders}  and \ref{ssec:compositions} are mostly notational/definitional in nature and deal with some nuances of the definitions.

\subsection{Well partial orders and their segmentations}\label{ssec:wpocombinatorics}
We collect here some known results about segmentations of well partial orders that will be needed throughout the paper.

\begin{definition}
	Let $(\Gamma, \le)$ be a partially ordered set. I call \emph{segment} an order-convex subset of $\Gamma$ and I call a partition $\cP$ of $\Gamma$ a \emph{segmentation} if every $P \in \cP$ is order-convex. Given two partitions $\cP$ of $\Gamma$ and $\cQ$ of $\Delta$, I denote by $\cP \otimes \cQ$ the partition $\{P \times Q: P \in \cP,\; Q \in \cQ\}$ of $\Gamma \times \Delta$.
\end{definition}

\begin{remark}
	Given two partitions $\cP$ and $\cQ$ of the same set $X$ one says that $\cQ$ is \emph{finer} than $\cP$ (or that $\cP$ is \emph{coarser} than $\cQ$) if every $P\in \cP$ is a union of elements of $\cQ$. %Sometimes we also say that $\cQ$ refines $\cP$.
	Given a family $\cF$ of subset of a set $X$, we call a partition $\cP$ of $X$, the partition \emph{generated} by $\cF$ if it is the coarsest partition such that each element of $F$ is a union of elements of $\cP$. The elements of $\cP$ are the elements minimal under inclusion in the family
    \[\big\{(\textstyle{\bigcap} S_0)  \cap (\textstyle{\bigcap} S_1) : S_0 \in \cF,\; S_1\in \{X\setminus X': X' \in \cF\}\big\} \setminus \{\emptyset\}\]
\end{remark}

\begin{remark}\label{rmk:finite_segmentations_are_filtered}
    The intersection of two segments is again a segment, thus any finite set of segments generates a finite segmentation. In particular given two finite segmentations $\cP$ and $\cQ$ of a partial order $\Gamma$, there is a finite segmentation finer than both of them.
\end{remark}

\begin{notation}
    If $(\Gamma, \le)$ is a partial order and $A$, $B$ are sets, we will denote by $[A, B)$ the convex set
    \[[A, B):=\{\gamma \in \Gamma: (\exists \alpha \in A, \,\alpha \le\gamma)\; \&\; (\forall \beta \in B, \, \gamma \not \le \alpha)\}.\]
    If $A=\{\alpha\}$ and $B=\{\beta\}$ are singletons, then we will just write $[\alpha, \beta)$ for $[A, B)$.
\end{notation}

\begin{remark}
    If $X\subseteq\Gamma$ is a segment, then $X=[X,B)$ where $B=\{\beta \in \Gamma: \nexists x \in X, x \le \beta\}.$
\end{remark}

\begin{definition}
    Recall that a partial order $(\Gamma, \le)$ is called a \emph{well partial order}, \emph{wpo} for short, if one of the following equivalent condition holds:
    \begin{enumerate}
        \item $(\Gamma, \le)$ does not contain infinite antichains nor infinite descending sequences;
        \item every infinite sequence $(\gamma_i)_{i<\omega}$ in $\Gamma$ has a weakly increasing subsequence $(\gamma_{i(k)})_{k <\omega}$;
        \item every total order $\le'$ extending $\le$ is a well order;
        \item every subset has a finite number of minimal elements.
    \end{enumerate}
    See \cite[Thm.~2.1]{higman1952ordering} for a proof of the equivalence in the more general context of quasi-orders (well partial orders are called there ``orders with the \emph{finite basis property}'', the terminology ``\emph{well partial order}'' is from \cite{kruskal1960well}, another early name for the same concept is ``\emph{partial well order}'' \cite{rado1954partial}). 
\end{definition}

\begin{remark}\label{rmk:initial_segments_of_wpo}
    A partial order $(\Gamma, \le)$ is a well partial order if and only if the set of initial segments of $\Gamma$ ordered by inclusion is well-founded (see \cite[Thm.~2.1]{higman1952ordering}, see also \cite{rado1954partial} for an example of a well partial order whose poset of initial segments has infinite antichains).
\end{remark}

\begin{remark}
    If $\Gamma$ is a wpo, then every segment of $\Gamma$ is of the form $X = [A, B)$ for finite antichains $A,B$. Furthermore if $A$ and $B$ are minimal with $X=[A,B)$ then $\forall \alpha \in A,\, \forall \beta \in B,\, \beta \not \le \alpha$. %explain minimal
\end{remark}

\begin{remark}\label{rmk:reduction_to_principal}
    If $M$ is a finite set and $\cS$ is a segmentation which refines the partition generated by $\{[m, \emptyset): m \in M\}$, then $\cS$ refines $[M, \emptyset)$.
\end{remark}

\begin{fact}[Dickson's Lemma]\label{fact:dickson}
    If $\Gamma, \Delta$ are wpos, then $\Gamma\sqcup \Delta$ and $\Gamma \times \Delta$ are wpos.
\end{fact}

\begin{proposition}\label{prop:basic_segmentation}
    If $\cR$ is a finite segmentation of $\Gamma \times \Delta$, then there are finite segmentations $\cS$ of $\Gamma$ and $\cT$ of $\Delta$ such that $\cS \otimes \cT$ refines $\cR$.
    \begin{proof}
        Note that by Remarks~\ref{rmk:finite_segmentations_are_filtered} it suffices to show it in the case $\cR=\{U, (\Gamma \times \Delta) \setminus U\}$ where $U$ is an upper set. Furthermore, by Fact~\ref{fact:dickson} and Remarks~\ref{rmk:finite_segmentations_are_filtered} and \ref{rmk:reduction_to_principal} we can further reduce to the case $U=[(\gamma, \delta), \emptyset)$ for some $(\gamma, \delta) \in \Gamma\times \Delta$.
        But in that case it suffices to take $\cS = \{[\gamma, \emptyset), \Gamma \setminus [\gamma, \emptyset)\}$ and $\cT = \{[\delta, \emptyset), \Delta \setminus [\delta, \emptyset)\}$.
    \end{proof}
\end{proposition}

\begin{remark}\label{rmk:fibers}
    If $f: (\Gamma, \le) \to (\Delta, \le)$ is a strictly increasing function, then each fiber of $f$ is an antichain: indeed if $f(\gamma_0)=f(\gamma_1)$ and $\gamma_0 \le \gamma_1$, then we must have $\gamma_0= \gamma_1$ for, if $\gamma_0<\gamma_1$, we would have $f(\gamma_0)<f(\gamma_1)$.
\end{remark}

\begin{corollary}\label{cor:basic_segmentation}
    Let $(\Gamma_0,\le), \ldots, (\Gamma_{n-1}, \le)$ be well partial orders and $(\Delta, \le)$ be a partially ordered set. If $f: \prod_{i<n} \Gamma_i \to \Delta$ is order-preserving, and $\cT$ is a finite segmentation of $\Delta$, then there are finite segmentations $\cS_i$ of $\Gamma_i$, such that $\bigotimes \cS_i$ refines the partition $f^*\cT:=\{ f^{-1}T : T \in \cT\}$. Furthermore if $f$ is strictly increasing then $f$ has finite fibers.
    \begin{proof}
        The first claim follows by a simple induction from Proposition~\ref{prop:basic_segmentation}. As for the second one it follows from the fact that $\prod_{i<n} \Gamma_i$ has no infinite antichains.
    \end{proof}
\end{corollary}

\begin{definition}
    Recall that a \emph{(partially) ordered monoid} is a structure $(\Gamma, \le, \cdot, 1)$ such that $(\Gamma, \le)$ is a partial order, $(\Gamma, \cdot, 1)$ is a monoid and $-\cdot -$ is increasing in both arguments (w.r.t.\ $\le$).
    
    Also recall that a monoid $(\Gamma, \cdot, 1)$ is \emph{cancellative} if for all $\alpha, \beta, \gamma \in \Lambda$, whenever $\alpha \beta=\alpha \gamma$ or $\beta \alpha=\gamma \alpha$, we have $\beta=\gamma$ and that an ordered monoid $(\Gamma, \le, \cdot , 1)$ is \emph{order-cancellative} if for all $\alpha, \beta, \gamma \in \Lambda$, whenever $\alpha \beta\le\alpha \gamma$ or $\beta \alpha\le\gamma \alpha$, we have $\beta\le\gamma$.
    
    We will abbreviate the expression \emph{commutative cancellative monoid} by \emph{c.c.\ monoid} and the expression \emph{commutative order cancellative} monoid by \emph{c.o.c.\ monoid} (with the understanding that this is an ordered monoid).
    We call a partially ordered monoid $(\Gamma, \le ,\cdot, 1)$ \emph{positive}, if $1= \min (\Gamma)$.
\end{definition}

\begin{remark}[C.c.\ monoids and Abelian groups]\label{rmk:ccmonoids_vs_groups}
    Although the results in the remainder of this subsection don't require commutativity, from Subsection~\ref{ssec:rings_with_powers} on we will essentially be only concerned with c.c.\ monoids.
    To help the intuition it is thus worth mentioning that a c.c.\ monoid $(\Gamma, \cdot, 1)$ can always be embedded into its \emph{group of quotients} $(\Gamma^{-1}\Gamma, \cdot,1)$ (when the operation is written additively so when we have $(\Gamma, +, 0)$ this is also referred to as the \emph{group of differences} $\Gamma-\Gamma$).
    The group $\Gamma^{-1} \Gamma$ is formally built as the quotient of the product monoid $\Gamma \times \Gamma$ under the equivalence relation $(\beta, \alpha)\sim (\beta',\alpha') \Leftrightarrow \alpha'\beta=\alpha \beta'$, the embedding $\iota: \Gamma \to \Gamma^{-1} \Gamma$ is given by $\gamma \mapsto (1, \gamma)/_\sim$ and the group inversion of $\Gamma^{-1} \Gamma$ is given by $(\alpha, \beta)/_{\sim} \mapsto (\beta, \alpha)/_{\sim}$, whence after identifying each $\gamma\in \Gamma$ with $\iota(\gamma)$ every element of $\Gamma^{-1}\Gamma$ can be written as $\beta^{-1} \alpha=(\beta, \alpha)/_{\sim}$.

    We refer to \cite[Ch.~II, Sec.~3]{heibsch1998semirings} for further details of the construction (there carried out in greater generality) and in particular to \cite[Ch.~II, Cor.~3.6]{heibsch1998semirings} for the claim that that $\iota$ is an embedding.

    If $\le$ makes $(\Gamma, \cdot, 1, \le)$ into an ordered monoid, then $\le$ can be naturally extended along $\iota$ to a group order on $(\Gamma^{-1} \Gamma, \cdot, 1, \le)$ by setting $\gamma^{-1}\alpha \le^g \delta^{-1} \beta \Leftrightarrow \exists \eta: \delta \alpha \eta \le \gamma \beta \eta$. The embedding $\iota$ will only in general be \emph{order preserving}. It will be \emph{order reflecting} if and only if $(\Gamma, \cdot, 1, \le)$ is \emph{order-cancellative}.
    We refer to \cite[Ch.~II, Thm.~4.1]{heibsch1998semirings} for further details on this.
\end{remark}

\begin{definition}    
    If $(\Gamma, \le)$ is a partial order, then the \emph{ordered monoid of words from $\Gamma$}, is formed as $(\Gamma^*, \le, \emptyset, \cdot)$ where $\Gamma^*=\bigcup_{n<\omega} \Gamma^n$, $\cdot$ is word concatenation, i.e.\ if $\alpha=(\alpha_0,\ldots, \alpha_{n-1}) \in \Gamma^n$ and $\beta=(\beta_0, \ldots, \beta_{m-1}) \in \Gamma^m$, then $\alpha \cdot \beta=(\alpha_0, \ldots, \alpha_{n-1}, \beta_0, \ldots, \beta_{m-1}) \in \Gamma^{m+n}$, and $\le$ is defined by setting $\alpha \le \beta$ iff there is a strictly increasing $i: n\to m$ such that $\beta_{i(j)}\ge \alpha_j$ for all $j < n$.
\end{definition}

\begin{remark}
    Each $\Gamma^n\subseteq \Gamma^*$ is order convex and $\{\Gamma^n: n \in \omega\}$ is a segmentation of $\Gamma^*$.
\end{remark}

\begin{fact}
    $(\Gamma^*, \le, \emptyset, \cdot)$ is the free positive ordered monoid on $\Gamma$. Namely, given any other ordered monoid $(\Delta, \le, 1, \cdot)$ with $1=\min(\Delta)$, the restriction map $-|_{\Gamma}: \Hom\big((\Gamma^*, \le, \emptyset, \cdot), (\Delta, \le, 1, \cdot)\big) \to \Pos(\Gamma, \Delta)$ is a bijection.
\end{fact}

\begin{fact}[Higman {\cite[Thm.~4.3]{higman1952ordering}}]\label{fact:higman}
    If $(\Gamma, \le)$ is a wpo, then $(\Gamma^*, \le)$ is a wpo. 
\end{fact}

The following consequence of the Fact~\ref{fact:higman} is usually referred to as Neumann's Lemma, as Neumann proved it in the case $(\Gamma, \le , \cdot, 1)$ is the monoid of positive elements from a totally ordered group, \cite[Thm.s~3.4 and 3.5]{neumann1949ordered}.
\begin{corollary}[Neumann]\label{cor:neumann}
    If $(\Gamma, \le, \cdot, 1)$ is a positive partially ordered monoid and $S\subseteq \Gamma$ is a wpo, then the submonoid $\langle S \rangle$ generated by $S$ is a wpo, furthermore if $1\notin S$ and $(\Gamma, \cdot, 1)$ is cancellative, then the canonical surjection $S^* \to \langle S \rangle$ has finite fibers.
    \begin{proof}
        Since $\langle S \rangle$ is an image of $S^*$ under an order preserving map, and by Fact~\ref{fact:higman} $S^*$ is a wpo, it follows that $\langle S \rangle$ is a wpo.
        Now suppose $(\Gamma, \cdot)$ is cancellative and $1 \notin S$, so $S>1$ by positivity of $\Gamma$.
        In such a case the canonical surjection $S^* \to \langle S \rangle$ is strictly increasing, so by Remark~\ref{rmk:fibers}, it must have finite fibers.
    \end{proof}
\end{corollary}

\begin{definition}
    If $\cS$ is a segmentation of $(\Gamma, \le)$, then we define $\cS^*$ as the segmentation of $\Gamma^*$ given by
    \[\cS^*:= \bigcup_{n<\omega}\cS^{\otimes n}.\]
\end{definition}

\begin{proposition}\label{prop:word_segmentation}
    If $\cR$ is a finite segmentation of $\Gamma^*$, then there is a finite segmentation $\cS$ of $\Gamma$ such that $\cS^*$ refines $\cR$.
    \begin{proof}
        We can suppose that $\cR=\{U, \Gamma^*\setminus U\}$, furthermore, since $U$ has finitely many minimal elements by Fact~\ref{fact:higman}, we can, applying Remarks~\ref{rmk:reduction_to_principal} and \ref{rmk:finite_segmentations_are_filtered}, further assume that $U=[\overline{\gamma}, \emptyset)$ for some $\overline{\gamma}=(\gamma_0, \ldots, \gamma_{n-1})$.

        I claim that the segmentation $\cS$ generated by the final segments of the form $[\gamma, \emptyset)$ with $\gamma \in M:= \{\gamma_0, \ldots, \gamma_{n-1}\} \cup \min (\Gamma)$ is enough. Let $S= \prod_{i<m} S_i \in \cS^*$ with $m>n$. The minimal elements of $U \cap \Gamma^m$ all have coordinate projections in the set $M$, and $\cS^{\otimes n}$ clearly refines the partition generated by $[\gamma, \emptyset)$ for all $\gamma \in M$.
    \end{proof}
\end{proposition}

\begin{corollary}\label{cor:word_segmentation}
    Let $(\Gamma_0,\le), \ldots, (\Gamma_{n-1}, \le)$ be well partial orders and $(\Delta, \le)$ be a partially ordered set. If $f: \prod_{i<n} \Gamma_i^* \to \Delta$ is order-preserving, and $\cT$ is a finite segmentation of $\Delta$, then there are finite segmentations $\cS_i$ of $\Gamma_i$, such that $\bigotimes_{i<n} \cS_i^*$ refines the partition $f^*\cT:=\{ f^{-1}T : T \in \cT\}$.
\end{corollary}

\subsection{Rings of generalized series}\label{ssec:series} Throughout this section we fix a ring $\bK$.

Let $(\fM, \le)$ be a partial order, the set of generalized series with coefficients from $\bK$ and monomials from $\fM$ is defined as as the $\bK$-module of $\fM$-tuples $(f_\fm)_{\fm \in \fM} \in \bK^\fM$ whose support $\Supp f:=\{\fm: f_\fm\neq 0\}$ is a well-partial order for the (restriction of the) reverse order $\le^{\op}$ on $\fM$ (that is, $\Supp f$ does not contain infinite strictly $\le$-increasing sequences nor infinite $\le$-antichains). Following \cite{hoeven2001operators} we will refer to this condition as $\Supp(f)$ being \emph{Noetherian}.

A family $(f_i)_{i \in I} \in \bK\Lhp \fM \Rhp$ is said to be \emph{summable} if for each $\fm \in \fM$, $\{i : (f_i)_\fm \neq 0\}$ is finite and $\bigcup_{i \in I} \Supp f_i$ is still a well-partial order for $\le^\op$. If $(f_i)_{i \in I}$ is summable its sum is defined as the function	
\[\left(\sum_{i \in I} f_i\right)_\fm := \sum_{i \in I} (f_i)_\fm.\]

Elements of $\fM$ are regarded as elements of $\bK \Lhp \fM \Rhp$ by identifying $\fm$ with the function that is $1\in \bK$ at $\fm$ and $0$ at $\fn\neq \fm$. In that sense every element $f \in \bK \Lhp \fM \Rhp$ can be regarded as the sum of the summable family $(f_\fm \fm)_{\fm \in \fM}$, so \[f= \sum_{\fm \in \fM} \fm f_\fm.\]

If $(\fM, \cdot, 1, \le)$ is a (possibly partially) ordered \emph{cancellative} monoid, then $\bK \Lhp\fM\Rhp$ has a natural ring structure.
The product $\cdot$ on $\bK \Lhp \fM \Rhp$ is defined as the only $\bK$-bilinear extension of the product on $\fM$ which is strongly bilinear in the sense that for every pair $(g_i)_{i \in I}$ and $(f_j)_{j \in J}$ of summable families in $\bK \Lhp \fM \Rhp$,
			
\[\left(\sum_{i \in I} f_i \right) \cdot \left( \sum_{j \in J} f_j\right) = \sum_{(i,j)\in I \times J} f_i \cdot f_j.\]
			
It is not hard to verify that defining the product as
\[\left(\left(\sum_{\fm \in \fM} a_{\fm} \fm \right) \cdot \left(\sum_{\fn \in \fM} b_{\fn} \fn \right) \right)_\fp := \sum_{\fm\fn =\fp} a_\fm b_\fn\]
yields the required property. We call a generalized series $f=\sum_{\fm} f_\fm \fm$ \emph{infinitesimal} if $\Supp(f)<1$ and \emph{non-singular} if $\Supp(f)\le 1$ and \emph{normal} if $\Supp(f)$ has a maximum. We will denote the sets of infinitesimal and non-singular series respectively by $\bK \Lhp \fM \Rhp^{\prec 1}=\bK\Lhp \fM^{<1}\Rhp$ and $\bK \Lhp \fM \Rhp^{\preceq 1}=\bK \Lhp \fM^{\le 1}\Rhp$.

\begin{remark}
    $\bK \Lhp \fM \Rhp^{\preceq 1}$ is a unital subring of $\bK\Lhp \fM \Rhp$ and $\bK \Lhp \fM \Rhp^{\prec 1}$ is an ideal in $\bK \Lhp \fM \Rhp^{\preceq 1}$. If $\fM$ is totally ordered, then all series in $\bK\Lhp \fM \Rhp\setminus \{0\}$ are normal.
\end{remark}

\begin{remark}
    If $\bK$ is a domain, then $\bK \Lhp \fM \Rhp$ is a domain.
\end{remark}

\begin{remark}[Units]\label{rmk:units}
    If $\bK$ is a domain, then a series $f\in \bK \Lhp \fM \Rhp$ is a unit if and only if $\Supp(f)$ has a maximum $\fm$ admitting a multiplicative inverse in $\fM$, and the coefficient of $\fm$ in $f$ is a unit of $\bK$.
    In fact if $f=k\fm(1+\varepsilon)$ with $k \in \bK$, $\fm \in \fM$, $\varepsilon \in \bK \Lhp \fM \Rhp^{\prec 1}$, and there are $\fm^{-1} \in \fM$, $k^{-1} \in \bK$ multiplicative inverses of $\fm$ and $k$ respectively, then the multiplicative inverse of $f$ is given by $k^{-1}\sum_{n} (-1)^n \varepsilon^n \fm^{-1}$. Conversely, if a series $f$ is invertible then its support must have a maximum, because otherwise $\Supp(fg)$ must always have more than one maximal element for any non-zero series $g$. The same should be true for $f^{-1}$, whence we see that it must be $f=\fm k(1+\varepsilon)$ with $k \in \bK$, $\fm \in \fM$, $\varepsilon \in \bK \Lhp \fM \Rhp^{\prec 1}$ and that $k$ and $\fm$ must have multiplicative inverses.%
\end{remark}

\begin{remark}[Monomial transformations]\label{rmk:monomial_transforms}
	If $\sigma: \fM \to \fN$ is a morphism of ordered cancellative monoids whose fibers are antichains, then $\sigma$ induces a ring homomorphism which we still denote by $\sigma$,
	\[\sigma: \bK \Lhp \fM\Rhp \to \bK \Lhp \fN \Rhp \qquad \sigma \sum_{\fm\in \fM} f_\fm \fm = \sum_{\fm \in \fM} f_\fm \sigma \fm,\]
	which is well defined because for all $\fn \in \fN$, $\sigma^{-1}(\fn)\cap \Supp(f)$ must be finite.
\end{remark}
			
\begin{definition}[Truncations and Segments]
	If $S$ is a subset of $\fM$, and $f :=\sum_{\fm} f_\fm \fm \in \bK \Lhp \fM \Rhp$ we will call the \emph{$S$-fragment of $f$} the series $f|S = \sum_{\fm \in S} f_\fm \fm$. 
    If $S$ is a segment of $\fM$, then the $S$-fragment $f|S$ is called a \emph{segment} of $f$.
    
    If $\fm \in \fM$, the \emph{$\fm$-truncation} of $f$, denoted by $f|\fm$ is the $S$-segment of $f$ with $S=\{\fn: \fn \not \le \fm\}$ and the \emph{$\fm$-tail} of $f$ is $f-f|\fm$.
	A subset $X \subseteq \bK \Lhp \fM \Rhp$ will be said to be \emph{weakly closed under truncations} (resp.\ \emph{weakly closed under tails}) if for each $f \in X$ and each $\fm \in \fM$, $f|\fm \in X$ (resp.\ $f-f|\fm\in X$).
    It will be said to be \emph{closed under truncations} (resp.\ \emph{tails}) if for each $f\in X$ and every finial segment $U$ of $\fM$, $f|U\in X$ (resp.\ $f-f|U \in X$). Finally $X$ will be said to be \emph{closed under segments} if for every $f\in X$ and every segment $S$ of $\fM$, $f|S \in X$.
\end{definition}

\begin{remark}
    If $<_1$ and $<_2$ are orders on $\fM$ and $<_1 \subseteq <_2$, then each $<_2$ segment is also a $<_1$ segment. Hence if $X \subseteq \bK \Lhp \fM \Rhp$ is closed under $<_1$-truncations (resp.\ segments), then it is also closed under $<_2$-truncations (resp.\ segments).
\end{remark}

\begin{remark}
    If an $X \subseteq \bK \Lhp \fM \Rhp$ is closed under truncations or under tails, and is an additive subgroup, then it is closed under segments.
\end{remark}

\begin{remark}
    If $\fM$ is totally ordered, then an $X \subseteq \bK \Lhp \fM \Rhp$ is closed under truncations if and only if it is weakly closed under truncations.
\end{remark}

% \begin{remark}
%     By definition of $\bK \Lhp \fM \Rhp$, every $f \in \bK \Lhp \fM \Rhp$ is contained in some $\fm \bK \Lhp \fN \Rhp$ for some $\le^{\op}$-wpo submonid $\fN$ of $\fM$ with $\fN \le 1$. In fact, $\Supp(f)$ is a $\le^\op$-wpo and we can find $\fm \in \fM$, such that $\fm \Supp(f) < 1$.
% \end{remark}
 
\begin{remark}\label{rmk:wtc_vs_segc}
    If $(\fM, \le)$ has finite meets (so $(\fM, \le^\op)$ has finite joints), and $X$ is a subgroup of $\bK \Lhp \fM \Rhp$, then $X$ is closed under weak truncations if and only it is closed under segments.
    In fact let $S$ be a segment of $\fM$: then it can be written as $S=L_0\setminus L_1$ where $L_0$ and $L_1$ are lower segments of $(\fM, \le)$ and $L_1 \subseteq L_0$, so that $f|S=f|L_0 - f|L_1$, thus we can reduce to the case in which $S$ is itself a lower segment of $(\fM, \le)$ and thus an upper segment of $(\fM, \le^\op)$.
    Now $\Supp(f) \cap L$ has a finite set $M=\{\fm_0, \ldots, \fm_{n-1}\}$ of maximal elements, thus we can write $f$ as 
    \[f=f|[\fm_0, \emptyset)_{\le^\op} + \sum_{0<i<n} \left(f\bigg|\bigg[\fm_{i}, \emptyset\bigg)_{\le^\op} - f\bigg|\bigg[\bigvee_{j\le i}\fm_j, \emptyset\bigg)_{\le^\op}\right).\]
    Thus noting that $f|[\fm_, \emptyset)_{\le^\op}=f-(f|\fm)$ we can conclude.
    Note that if $\fM$ is a product of partial orders with finite meets, then $\fM$ has finite meets.
\end{remark}

\begin{remark}\label{rmk:partial_truncations}
    If $X \subseteq \bK \Lhp \fM \times \fN\Rhp$ is an additive subgroup, then for $X$ to be closed under truncations it is enough that $X$ is closed under \emph{partial} truncations, i.e.\ that it is closed under truncations of the form $f|(S \times \fN)$ and $f|(\fM \times T)$ for $S$, $T$ final segments of $\fM$ and $\fN$ respectively. Indeed it immediately follows that since $X$ is an additive subgroup, if it is closed under partial truncations, then it is then closed in general under segments of the form $f|(S \times \fN)$ and $f|(\fM \times T)$ where $S$ and $T$ are just segments resp.\ of $\fM$ and $\fN$.
    
    Now given a final segment $U \subseteq \fM \times \fN$, by Proposition~\ref{prop:basic_segmentation}, we can write $U\cap \Supp(f)= \bigsqcup_{i<n} S_i \times T_i$ where $S_0, \ldots S_{n-1}$ are segments of $S \cap \pi_{\fM}(\Supp(f))$, $T_i$ are segments of $\pi_{\fN}(\Supp(f))$ and $\pi_\fM, \pi_\fN$ are the projections of $\fM \times \fN$ respectively on the first and second component.
    Now clearly $f|(S_i \times T_i)=\big(f|(S_i\times \fN)\big)|(\fM \times T_i)$, so letting $S_i'$ and $T_i'$ be respctively the convex hulls of $S_i$ in $\fM$ and $T_i$ in $\fN$, we can then write $f|L=\sum_{i<n}(f|(S_i \times \fN)\big)|(\fM \times T_i)$.
\end{remark}

%TODO: I need to remark that the previous version had a different definition of truncation in the poset case.

The following Lemma is essentially a restatement of \cite[Lem.~3.3]{mourgues1993every} in the slightly more general setting of partially ordered cancellative monoids.

\begin{lemma}\label{lem:product_segmentation}
    Let $(\fM, \le, \cdot, 1)$ be a partially ordered cancellative monoid and let $S_0, S_1$ be Noetherian subsets of $\fM$. Then for each final segment $U$ of $\fM$, there are \emph{final segments} $U_0, \ldots, U_{n-1}$ of $S_0$ and \emph{segments} $T_0, \ldots, T_{n-1}$ of $S_1$ such that
    $(S_0 \times S_1) \cap U = \bigsqcup_{i<n} U_i \times T_i$.
    In particular whenever $f,g \in \bK \Lhp \fM\Rhp$ are such that with $\Supp(f)\subseteq S_0$ and $\Supp(g) \subseteq S_1$, $(f \cdot g)|U = \sum_{i<n} (f|U_i) \cdot (g_i|T_i)$.
    \begin{proof}
        This follows directly from Corollary~\ref{cor:basic_segmentation}: we get finite segmentations $\cS_0$ and $\cS_1$ of $S_0$ and $S_1$ such that $\cS_0 \otimes \cS_1$ refines the partition generated by $\{(\fm, \fn) \in S_0 \times S_1: \fm\fn \in U\}$.
        Set for each $T \in \cS_1$, $U_T=\bigcup\{R \in \cS_0: R \times T \subseteq U$ and notice that $U_T$ is always a final segment. Then let $T_0, \ldots, T_{n-1}$ be an enumeration of those $T$s such that $U_T \neq \emptyset$ and $U_i:=U_{T_i}$.
    \end{proof}
\end{lemma}
    
\begin{lemma}\label{lem:alg_generated_by_tc}
    Let $(\fM, \le , \cdot, 1)$ be a partially ordered cancellative monoid.
    If $X\subseteq \bK \Lhp \fM \Rhp$ is closed under truncations, and $R$ is a subring of $\bK$, then the $R$-algebra generated by $X$ is closed under truncations.
    \begin{proof}
        Note that the $R$-module generated by $X$ is always closed under segmentation, so we can assume $X$ is already a $R$-module closed under segmentation.
        
        Now to prove the statement it suffices to show that if two $R$-modules $X$ and $Y$ are closed under truncations, then so is the $R$-module generated by the set $\{f \cdot g: f \in X, g \in Y\}$. For this, let $L$ be an initial segment of $\fM$: by Lemma~\ref{lem:product_segmentation}, we have that $f=\sum_{i<n} (f|U_i) \cdot (g_i|T_i)$ for some final segments $(U_i)_{i<n}$ of $\fM$ and for some segments $(T_i)_{i<n}$ of $\fM$.
    \end{proof}
\end{lemma}

\begin{definition}
    An ordered commutative monoid $(\fM, \cdot, 1, \le)$ is \emph{positively cone-ordered} if for all $\fn \in \fM$, $\fM^{\ge \fn}=\fn \fM^{\ge 1}$ (cf \cite[Ch.~III, Sec.~1]{heibsch1998semirings} for some basic considerations around this notion which is however not explicitly named).
    We will say that $(\fM, \cdot, 1, \le)$ is \emph{negatively cone-ordered} if $(\fM, \cdot, 1, \le^\op)$ is positively cone-ordered.

    A subset $X\subseteq \bK \Lhp \fM \Rhp$ is said to be \emph{closed under monomial division} if whenever $f\in \bK \Lhp \fM^{\le 1} \Rhp$, $\fm \in \fM$ and $\fm f \in X$ we have $f \in X$.
\end{definition}

\begin{remark}[Cone orderings are order-cancellative]
    If an ordered c.c.\ monoid $(\Lambda, +, 0, \le)$ is positively cone-ordered, then it is order-cancellative, in particular the natural group embedding $\iota$ in the ordered group of differences $(\Lambda -\Lambda, +, 0, \le^g)$ is an order-embedding (cf Remark~\ref{rmk:ccmonoids_vs_groups}). Furthermore if $\Lambda$ is a c.o.c.\ monoid, then it is positively cone-ordered if and only if $\iota(\Lambda)$ contains the positive cone of $\Lambda-\Lambda$.

    It follows that positively cone-ordered monoids are precisely those ordered monoids that are realized as a submonoid of an ordered Abelian group containing the positive cone of the group. 
\end{remark}

\begin{lemma}\label{lem:md-closed_gen_subgroup}
    If a c.c.\ ordered monoid $(\fM, \cdot, 1, \le)$ is negatively cone-ordered and $X\subseteq \bK \Lhp \fM \Rhp$ is closed under monomial division and weakly closed under tails, then the subgroup generated by $X$ is closed under monomial division.%TODO:introduce tails
    \begin{proof}
        If $x_0, \ldots, x_{n-1} \in X$ and $\sum_{i<n} x_i=\mathfrak{m}y$ for some $y \in \bK \Lhp \fM^{\le 1} \Rhp$, then since $X$ is closed under tails $x_i|(\fM^{\le \fm}) \in X$, and since $\fM$ is negatively cone ordered $\fM^{\le \fm}=\fm \fM^{\le 1}$, so we get by closure under monomial division of $X$, that there are $y_i \in X \cap \bK \Lhp \fM^{\le 1}\Rhp$ such that $x_i|(\fM^{\le \fm})=\fm y_i$ and we easily see that $y=\sum_{i<n} y_i$ so $y$ is in the subgroup generated by $X$.
    \end{proof}
\end{lemma}

\begin{lemma}\label{lem:md-closed_gen_subring}
    If a c.c.\ ordered monoid $(\fM, \cdot, 1, \le)$ is negatively cone-ordered, $X\subseteq \bK \Lhp \fM \Rhp$ is closed under monomial division and tails, and $\fM \subseteq X$, then the subring generated by $X$ is closed under monomial division and truncations.
    \begin{proof}
        It follows from Lemma~\ref{lem:md-closed_gen_subring} that to prove the statement we can assume that $X$ is already an additive subgroup. Also note that the additive group $\fM \cdot X$ must be contained in the ring generated by $X$ and will be as well closed under monomial division, thus we can also assume $\fM \cdot X= X$.
        %Then since $X$ is a subgroup and closed under tails, it is closed under fragments and thus contains all terms appearing in a series in $X$, since $X$ is furthermore closed under monomial division, we get that $S\coloneqq \fM \cap X$ is exactly the union of the supports the series appearing in $X$. Now $S$ is itself closed under monomial division.
        
        %We also have In particular the ring generated by $X$ must contain the monoid $\fN$ generated by the series contained in $X$. Also 
        
        To conclude thus it will suffice to prove that if $X$ is an additive subgroup closed under truncations and monomial division, then so is the subgroup generated by $\{f_0 f_1: f_0, f_1 \in X\}$.
        
        Now if $f_0,f_1 \in X$, since $X$ is a truncation additive closed subgroup, we can find segmentations $\cS_0$ and $\cS_1$ of $\fM$, such that for all $S \in \cS_i$, $f_i|S$ is normal and in $X$.
        %Suppose that $f_0\cdot f_1=\fm \cdot g$ with $g \in \bK\Lhp \fM^{\le 1}\Rhp$. Let $L=\fm \fM^{\le1}=\fM^{\le\fm}$, we can find segmentations $\cS_0$ and $\cS_1$ of $\Supp(f_0)$ and $\Supp(f_1)$, respectively such that $S_0S_1\subseteq L$ or $S_0 S_1\cap L=\emptyset$ for all $S_0 \in \cS_0$ and $S_1 \in\cS_1$ and furthermore $S_i$ has a unique maximum for all $S_i \in \cS_i$, $i\in 2$.%TODO add Lemma sayign we can refine a segmentation of a wpo to one in which every segment has a unique minimum
        It follows that for all $i\in 2$ and $S_i\in \cS_i$, $(f_i|S_i)=\fm_{i,S_i} g_{i,S_i}$ for some $g_{i,S_i}\in X \cap \bK \Lhp \fM^{\le 1} \Rhp$. Furthermore we must have $\fm_{0, S_0}\fm_{1, S_1} \le \fm$ for all $(S_0, S_1) \in \cS_0 \times \cS_1$. Thus in particular $\fm_{0, S_0}\fm_{1, S_1}=\fp_{S_0, S_1}\fm$ for some $\fp_{S_0, S_1} \le 1$. %Note that the hypothesis that $X$ is closed for monomial division thus gives $\fm_{0, S_0},\fm_{1, S_1},\fp_{S_0, S_1} \in X$.
        Finally we can write
        \[f_0\cdot f_1 = \sum_{(S_0,S_1) \in \cL} (f_0|S_0) \cdot (f_1|S_1)=\fm\cdot \sum_{(S_0,S_1) \in \cL} \fp_{S_0, S_1} g_{0, S_0}g_{1, S_1}\]
        where each of the addends on the right side is in $\{h_0h_1: h_0, h_1 \in X\}$ because $\fM \cdot X=X$.
    \end{proof}
\end{lemma}

\subsection{Languages of \texorpdfstring{$\Lambda$}{Λ}-power series}\label{ssec:Gseries}
This subsection is dedicated to some formalities about power series with exponents from an ordered semiring $\Lambda$.

We fix for the rest of the section an infinite set of formal variables $\Var$. We will denote variables and sets of variables by $\var{x}, \var{y}, \var{z}, \ldots$.
If $\var{x}$ is a set of variables, we will denote by $S^\var{x}$ the set of assignments from $S$ for $\var{x}$, i.e.\ the set of functions from $\var{x}$ to $S$.

If $n \in \bN$ and $\var{x}=\{\var{x}_i: i<n\}\subseteq \bN$ is a finite set $n$ distinct variables, and $(a_i)_{i<n} \in S^{n}$ is an $n$-tuple from $S$, sometimes we will use the notation $[\var{x}_0 \mapsto a_0, \ldots \var{x}_{n-1} \mapsto a_{n-1}]=[\var{x}_i \mapsto a_i]_{i<n}$ for the assignment $a\in S^{\var{x}}$ given by $a(\var{x}_i)=a_i$.
With a slight abuse of notation we will also write $\var{x}$ for the identical assignment $[\var{x}_0\mapsto\var{x}_0, \ldots, \var{x}_{n-1}\mapsto \var{x}_{n-1}]$.

If $\var{x}, \var{y}, \var{z}, \ldots$ are pairwise disjoint sets of variables, and $a \in S^\var{x}, \, \in S^\var{y}, \, c \in s^{\var{z}},\, \cdots$ are assignments, I will denote by $(a,b,c, \ldots) \in S^{(\var{x}, \var{y}, \var{z}, \ldots)}$ the assignment which is $a$ on $\var{x}$, $b$ on $\var{y}$, $c$ on $\var{z}$ and so on.

\begin{notation}\label{not:monomials}
    Let $n \in \bN$ and $\var{x}=\{\var{x}_i: i<n\}\subseteq \Var$ a set of $n$ distinct variables.
    If the an assignment $m \in M^{\var{x}}$ takes values in some multiplicatively written $\Lambda$-module $(M, \cdot, (-)^{\lambda}_{\lambda \in \Lambda})$ and $\gamma \in \Lambda^{\var{x}}$, then we will use the notation $m^\gamma$ to denote the product $m^\gamma=\prod_{i<n}m(\var{x_i})^{\gamma(\var{x}_i)}$.
\end{notation}

\begin{definition}[$\Lambda$-power series]\label{defn:gen_power_series}
     Let $\Lambda$ be an ordered commutative and additively cancellative semiring and $\var{x}= \{\var{x}_i :i<n\}$ a set of $n$ distinct formal variables.
    The set of \emph{momonials in $\var{x}$ with exponents from $\Lambda$} is the free $\Lambda$-module on the set $\var{x}$, which we denote by $(\var{x}^\Lambda, \cdot ,1)$ writing its operation multiplicatively.
    A \emph{monomial} in $\var{x}$ with exponents from $\Lambda$ can thus be written uniquely as a product $\var{x}_0^{\gamma_0} \cdots \var{x}_{n-1}^{\gamma_{n-1}}$, which, following Notation~\ref{not:monomials} we write as $\var{x}^\gamma$, where $\gamma \in \Lambda^{\var{x}}$ is the assignment $\var{x}_i \mapsto \gamma_i$.
    
	We order the set of monomials $(\var{x}^\Lambda, \cdot, 1, \le)$ by
	\[\var{x}^{\alpha}\var{x}^\beta=\var{x}^{\alpha+\beta}, \qquad \var{x}^{\gamma}\ge \var{x}^{\beta} \Longleftrightarrow \gamma \le \beta,\]
	where $\Lambda^{\var{x}}$ is given the product (partial) order. The intuition for reversing the order here is that positive infinitesimals will be assigned to the $\var{x}_i$.
	
	For any ring $\bK$, the ring $\bK \Lhp \var{x}^{\Lambda} \Rhp$ is called the \emph{ring of $\Lambda$-power series in $\var{x}$} with coefficients from $\bK$. A \emph{$\Lambda$-power series} with coefficients form $\bK$ is an element of $\bK \Lhp \var{x}^{\Lambda} \Rhp$, so it can be expressed as a formal sum
	\[f=\sum_{\gamma \in \Lambda^\var{x}} f_\gamma \var{x}^{\gamma}\]
	where $\{\gamma \in \Lambda^\var{x}: f_\gamma \neq 0\}$ is a well-partial order. Also note that $f(\var{x}) \in \bK \Lhp \var{x}^\Lambda\Rhp^{\prec 1}$, i.e.\ $f$ is infinitesimal if and only if $\{\gamma \in \Lambda^{\var{x}}: f_\gamma \neq 0\}>0$.
    When $\Lambda=\bN$ a $\Lambda$-power series is just a formal power series.
\end{definition}

%Add some singposting explaining why this is important

\begin{definition}[Semiring of exponents, Classical variables]
    Let $\var{y}\subseteq \var{x}\subseteq \Var$ and $f \in \bK \Lhp \var{x}^\Lambda\Rhp$. We call semiring of $\var{y}$-exponents of $f$, the smallest subsemiring $\Lambda_{f, \var{y}}\coloneqq \RExp(f,\var{y})$ of $\Lambda$ such that $f \in \bK \Lhp \var{y}^{\Lambda_0}(\var{x}\setminus \var{y})^{\Lambda}\Rhp$.
    
    We say that the variables $\var{y}$ are \emph{classical} in $f$ or that the $\Lambda$-power series $f$ is \emph{classical} in the variables $\var{y}$ when $\RExp(f,\var{y})$ is the minimal subsemiring of $\Lambda$ (i.e. the image of $\bN$ under the unique semiring homomorphism $\bN \to \Lambda$).
\end{definition}

\begin{remark}[Reindexings]%TODO:write somewhat more precisely
	If $\var{x}\subseteq \var{y}$, then every monomial $\var{x}^\gamma$ is identified with the monomial in $\var{y}$ given by $\var{y}^{\beta}$ where $\beta \in \Lambda^{\var{y}}$ is given by $\beta(\var{z}) = \gamma(\var{z})$ for each $\var{z} \in \var{X}$ and $\beta(\var{z})=0$ otherwise. This induces natural inclusions 
	\[\var{x}^{\Lambda}\subseteq \var{y}^{\Lambda}, \qquad \bK \Lhp \var{x}^{\Lambda} \Rhp \subseteq \bK \Lhp \var{y}^{\Lambda} \Rhp.\]
    The set of \emph{all} monomials with exponents from $\Lambda$ is then $\Var^{(\Lambda)}$, the free $\Lambda$-module on the set $\Var$.
	Similarly if $\sigma : \var{x} \to \var{y}$ is an injection, it induces natural inclusions
	\[\sigma : \var{x}^\Lambda \to \var{y}^\Lambda, \qquad \sigma : \bK \Lhp \var{x}^{\Lambda} \Rhp \to\bK \Lhp \var{y}^{\Lambda} \Rhp.\]
			
	It is convenient to extend this reindexing to non-injective $\sigma: \var{x} \to \var{y}$, to take into account the operation of forming for example $f(\var{x}, \var{x})$ out of $f(\var{x}, \var{y})$. This is entirely possible because even if $\sigma:\var{x} \to \var{y}$ is not injective, $\sigma: \var{x}^{\Lambda} \to \var{y}^{\Lambda}$ is still strictly increasing (cf Remarks~\ref{rmk:fibers} and \ref{rmk:monomial_transforms}). Thus we will call \emph{reindexing} any function $\sigma: \var{x} \to \var{y}$ for $\var{x}$ and $\var{y}$ finite sets of variables.
\end{remark}

\begin{definition}[Weakly restricted $\Lambda$-power series]\label{def:weakly_restricted}
    It will be useful to consider rings of generalized series of the form $\bK \Lhp \fM \times \var{x}^{\Lambda}\Rhp$ where $\fM$ is some ordered $\Lambda$-module. These live naturally in the ring of $\Lambda$-power series $\bE\Lhp \var{x}^\Lambda\Rhp$ where $\bE$ is the ring $\bK\Lhp \fM \Rhp$,
    in fact every $f \in \bK \Lhp \fM \times \var{x}^\Lambda\Rhp$ can be written, after identifying $(\fm,1)$ with $\fm$ and $(1,\var{x}^\gamma)$ with $\var{x}^{\gamma}$, as 
    \[f=\sum_{\gamma \in \Lambda^{\var{x}}} f_\gamma \var{x}^\gamma \quad \text{where} \quad f_\gamma \in \bE\coloneqq\bK \Lhp \fM \Rhp\]
    and such an $f \in \bE\Lhp \var{x}^\Lambda \Rhp$ is in $\bK \Lhp \fM\times \var{x}^\Lambda\Rhp$ if and only if $\Supp_\fM(f)\coloneqq \bigcup_{\gamma \in \Lambda^\var{x}} \Supp(f_\gamma)$ is a Noetherian subset of $\fM$.

    We call the series in $\bK \Lhp \fM \times \var{x}^\Lambda\Rhp$ \emph{weakly restricted} $\Lambda$-power series with coefficients in $\bE=\bK \Lhp \fM \Rhp$. If $\fM=1$ is the trivial $\Lambda$-module, then we get the back the notion of $\Lambda$-power series with coefficients from $\bK$. The reason for the name ``weakly restricted'' is that (under suitable conditions on $\bK$, $\Lambda$, and $\fM$) such series are the ones that can be evaluated (i.e.\ converge) at infinitesimal powerable elements of $\bK \Lhp \fM \Rhp$ (see Definition~\ref{defn:genps_interpretation} and Remark~\ref{rmk:on_terminology} below).
    Again when $\Lambda=\bN$ we call elements of $\bK\Lhp \fM \times \var{x}^\bN\Rhp$ just \emph{weakly restricted power series} (thus dropping the $\bN$).
\end{definition}

\begin{remark}[Monomial substitution]\label{rmk:monomial_substitution}
    If $f\in \bK \Lhp \fM \Rhp$ is as in Definition~\ref{def:weakly_restricted}, and $\fM$ has the property
    \begin{equation}\label{eq:ppty-adequate}
    \forall \fm \in \fM^{<1}, \;\Lambda \ni \lambda \mapsto \fm^\lambda \in \fM \; \text{is one-to one},
    \end{equation}
    then for any $\fm \in \fM^{<1}$, the family $(\fm^\gamma f_\gamma: \gamma \in \Lambda^{\var{x}})$ is summable in $\bK \Lhp \fM\Rhp$, because its support is contained in $\Supp_\fM(f) \cdot \{\fm^\gamma: \gamma \in \Lambda, f_\gamma \neq 0\}$ which is the product of two Noetherian sets, by (\ref{eq:ppty-adequate}), for all $\fn \in \fM$ there are only finitely many pairs $(\fp, \alpha) \in \Supp_{\fM}(f)\times \{\gamma: \gamma \in \Lambda, f_\gamma \neq 0\}$ such that $\fp \fm^{\alpha}=\fn$, and for each such pair only finitely many $\gamma$ with $\fp \in \Supp(f_\gamma)$. Also note that in order for $\fM$ to have property (\ref{eq:ppty-adequate}) it is necessary that $\Lambda$ has such a property as $\Lambda$-module.
\end{remark}

\begin{definition}\label{def:algebras&languages}
	A \emph{family} $\cF$ of weakly restricted $\Lambda$-power series with coefficients from $\bK \Lhp \fM \Rhp$ is the datum for every finite $\var{x}\subseteq \Var$ of a subset $\cF(\var{x})\subseteq \bK \Lhp \fM \times \var{x}^\Lambda\Rhp$.
    If $\var{x}, \var{y}, \ldots$ are pairwise disjoint finite sets of variables, will write $\cF(\var{x}, \var{y}, \ldots)$ as a short hand for $\cF(\var{x} \cup \var{y} \cup \cdots)$.
    
    A family $\cF$ is said to be \emph{truncation-closed} if each $\cF(\var{x})$ is closed under truncations (see Remark~\ref{rmk:partial_truncations}).

     We say that a family $\cF$ is a \emph{language} of weakly restricted $\Lambda$-power series if for any reindexing function $\sigma: \var{x} \to \var{y}$, $\sigma \cF(\var{x}) \subseteq \cF(\var{y})$.
			 
     Given a ring homomorphism $S \to \bK\Lhp \fM \Rhp$, the language $\cF$ is said to be an $S$-\emph{algebra} if $\cF(\var{x})$ is a $S$-subalgebra of $\bK \Lhp \fM \times \var{x}^\Lambda\Rhp$ for every $\var{x}$.    
\end{definition}

\begin{remark}
    The reason for the terminology ``$\cF$ is a language" is that a first-order functional signature can be identified with a set of terms closed under reindexing (the set of its unnested terms). At some point we will want to regard the generalized power series of $\cF$ as function symbols for a first-order signature. However by construction every generalized power series comes ``already applied to some formal variables'', hence it is more correct to identify $\cF$ with a set of unnested terms and call it a \emph{language} when it is closed under reindexing.
\end{remark}

\begin{remark}
	Note that for brevity we are adopting the non-standard convention that when we say that a family $\cF$ is an algebra we always mean it is also a language (so we require it to be also closed under reindexings). 
\end{remark}

%add remark on truncations

\begin{lemma}\label{lem:tc_generates_tc}
	If a family $\cF$ of weakly restricted $\Lambda$-power series with coefficients from $\bK\Lhp \fM \Rhp$ is truncation-closed, and $S\subseteq \bK \Lhp \fM \Rhp$ is a truncation closed subring, then the $S$-algebra generated by $\cF$ is truncation-closed.
	\begin{proof}
		First observe that by Lemma~\ref{lem:alg_generated_by_tc}, the $S$-algebra generated by each $\cF(\var{x})$ is closed under truncation provided that $\cF(\var{x})$ was.

        Thus to prove the Lemma it suffices to show that if $\cF$ is such that each $\cF(\var{x})$ is a truncation-closed algebra then the language generated by $\cF$ is truncation-closed. Note that the language generated by $\cF$ is given by $\cG:=\{\cG(\var{x}): \var{x} \subseteq \Var\}$ where $\cG(\var{x})=\bigcup\{\sigma \cF(\var{y}): \var{y} \subseteq \Var,\, |\var{y}|<\aleph_0,\, \sigma: \var{y} \to \var{x}\}$. Since a union of truncation-closed subsets is truncation-closed, to show $\cG$ is truncation-closed it suffices to show $\sigma \cF(\var{x})$ is truncation-closed for each reindexing $\sigma: \var{x} \to \var{y}$.
                
        Every $\sigma$ can be written as a composition $\iota \cdot \sigma_0 \cdots \sigma_{k-1}$ where $\iota$ is an injection and each $\sigma_i$ identifies only two distinct variables. Since injective reindexings send truncation-closed subsets to truncation-closed subsets, the proof of the Lemma boils down to observing that if $f\in \cF(\var{x}, \var{y}, \var{z})$, $\var{x}, \var{y} \in \Var$, and $\var{z}\subseteq \Var\setminus \{\var{x}, \var{y}\}$, then $f | \{\fM \cdot \var{x}^\beta\var{y}^\gamma\var{z}^\delta: \gamma+\beta< \alpha\} \in \cF(\var{x}, \var{y}, \var{z})$. %This follows from Remark~\ref{rmk:partial_tc} and the fact that $\{\fM \cdot \var{x}^\beta\var{y}^\gamma\var{z}^\delta: \gamma+\beta< \alpha\}$ is a segment for the partial order on $\fM \cdot \var{x}^\Lambda \var{y}^\Lambda \var{z}^{\Lambda}$. %somewhere observe that $\fM$ has finite meets or that the argument works also without finite meets
	\end{proof}
\end{lemma}

\subsection{Rings with powers and ordered semirings}\label{ssec:rings_with_powers}
Although Theorem~\ref{introthm:A} will be used in the case when $\bK$ is a power bounded o-minimal structure with field of exponents from $\Lambda$, we carry on the treatment in a slightly greater generality, so that the definitions always include the case in which the semiring of exponents $\Lambda$ is given by the smeiring of natural numbers $\bN$ or by the ring of integers $\bZ$.

This subsection is dedicated to introduce some conditions on the semiring $\Lambda$ and the $\Lambda$-module $\fM$, which will be required to define powers of exponent $\lambda \in \Lambda$ of series in $\bK \Lhp \fM \Rhp$. We will often point to \cite{heibsch1998semirings} for some facts on monoids, semirings and their orderings, \cite{fuchs1963partially} is also a standard reference.

\begin{definition}
    Recall the following definitions:
    \begin{enumerate}
        \item A (unital) \emph{semiring} $(\Lambda, +, \cdot, 0,1)$ is a structure such that $(\Lambda, +, 0)$ and $(\Lambda, \cdot, 1)$ are monoids, $(\Lambda, +, 0)$ is commutative, and $\cdot$ distributes on both sides over $+$, i.e.\ $x(y+z)=xy+ xz$ and $(y+z)x=yz+yx$ for all $x,y,z \in \Lambda$; a semiring $\Lambda$ is said to be:\begin{enumerate}
            \item \emph{commutative} when $(\Lambda, \cdot ,1)$ is commutative;
            \item \emph{additively cancellative} when $(\Lambda,+,0)$ is cancellative;
            \item \emph{cancellative} when $(\Lambda, +, 0)$ and $(\Lambda, \cdot, 1)$ are both cancellative.
        \end{enumerate} 
        For basic observations around these notions we refer to \cite[Ch.~I]{heibsch1998semirings}, where however semirings are not necessarily assumed to have a $0$ and a $1$.
        We will abbreviate commutative additively cancellative semiring by \emph{c.a.c.\ semiring} and cancellative semiring by \emph{c.c.\ semiring}.
        \item An \emph{ordered} semiring is a semiring endowed with a partial order $\le$, such that $x\le y \Rightarrow z+x\le z+y$ for all $x,y,z$ and $x\le y \Rightarrow \big(x\cdot z \le y\cdot z \;\; \&\;\; z\cdot x \le z \cdot y)$ for all $x,y,z$ with $z\ge 0$ (cf \cite[Ch.~III, Def.~2.1]{heibsch1998semirings}).
        \item An ordered semiring $(\Lambda, +, \cdot, 0,1, \le)$ is \emph{cone-ordered}, if the underlying additive monoid $(\Lambda, +, 0, \le)$ is positively cone-ordered.
        \item Given a semiring $\Lambda$, a \emph{$\Lambda$-action} on a commutative monoid $(\fM, \cdot,1)$ is a semiring homomorphsim $\varphi: \Lambda \to \End(\fM, \cdot)$. Recall that when the monoid is written multiplicatively, the standard notation for the trivial action $\id: \End(\fM, \cdot) \to \End(\fM,\cdot)$, which we will adopt in such case, is the power notation, i.e.\ we write $\fm^h$ for the image of $\fm \in \fM$ by $h \in \End(\fM, \cdot)$.
        \item If $\Lambda$ and the commutative monoid $(\fM, \cdot, 1)$ are both ordered, then we say that an action $\varphi: \Lambda \to \End(\fM,\cdot)$ is \emph{ordered} when $\varphi(h): \fM \to \fM$ is order preserving for all $h \in \Lambda^{\ge0}$.  
        \item A commutative monoid $\fM$ with a $\Lambda$-action is called \emph{$\Lambda$-module}, an ordered commutative monoid $\fM$ with an \emph{ordered} $\Lambda$-action is called an \emph{ordered $\Lambda$-module}.
    \end{enumerate}
\end{definition}

\begin{remark}[Modules of differences]\label{rmk:modules_of_differences}
    If $\Lambda$ is a semiring $M=\big(M, +, 0, \le, \varphi: \Lambda \to \End(M, +,0)\big)$ is an (additively written) $\Lambda$-module, then the group of differences $M-M$ has a unique $\Lambda$-action $\bar{\varphi}:\Lambda \to \End(M-M, +, 0)$ making $\iota: M \to M-M$ into a $\Lambda$-module morphism: this is formally given setting $\bar{\varphi} (\lambda) (m-n) = \varphi(\lambda)(m)-\varphi(\lambda)(n)$ for all $m, n \in M$.
    If furthermore $\Lambda$ is an ordered semiring and $M=\big(M, +, 0, \le, \varphi: \Lambda \to \End(M, +,0)\big)$ is an ordered $\Lambda$-module then the action $\bar{\varphi}$ is ordered.
\end{remark}

\begin{remark}[Additively cancellative semirings and rings]\label{rmk:semirings_vs_rings}
    If $(\Lambda, +, \cdot, 0, 1)$ is an additively cancellative semiring, then the group of differences $(\Lambda-\Lambda, +,0)$ of the additive c.c.\ monoid $(\Lambda, +, 0)$ carries a natural product (still denoted by $-\cdot-$) making it into a ring and extending $-\cdot-$ along the natural embedding $\iota: \Lambda \hookrightarrow \Lambda-\Lambda$: namely one sets $(-\gamma+\alpha)\cdot(-\delta+\beta)\coloneqq -(\gamma \beta+\alpha\delta) + (\alpha\beta+\gamma \delta)$ (cf \cite[Ch.~II, Thmm.~5.8 and 5.11]{heibsch1998semirings}). The resulting structure $(\Lambda-\Lambda, +,\cdot, 0,1)$ is called ring of differences and is the reflection of $\Lambda$ onto the category of rings, i.e.\ it has the universal property that for all semiring homomorphisms $f: \Lambda \to R$ where $R$ is a ring, there is a unique ring homomorphism $\bar{f}: \Lambda-\Lambda \to R$ s.t.\ $\bar{f} \circ \iota=f$.
    
    In $\le$ is an order on $\Lambda$ which makes it an ordered semiring, then the induced group order $\le^g$ on $(\Lambda-\Lambda, +, 0)$ (see Remark~\ref{rmk:ccmonoids_vs_groups}) doesn't need to make $(\Lambda-\Lambda, +, \cdot, 0)$ into an ordered ring. In fact $(\Lambda-\Lambda, +, \cdot, 0, 1, \le^g)$ will be an ordered ring if and only if $(\Lambda, +, \cdot, 0, 1, \le)$ satisfies
    \[\forall \alpha, \beta, \gamma, \delta,\; (\gamma \le \alpha\; \&\; \delta\le\beta) \rightarrow \exists \eta, \theta,\;
    \left\{\begin{aligned}
    \alpha \beta +\gamma \delta + \eta \ge \gamma \beta + \alpha \delta + \eta\\
    \beta \alpha + \delta \gamma + \theta \ge \beta \gamma + \delta \alpha + \theta
    \end{aligned}\right.\]
    see \cite[Ch.~III, Thm.~4.3]{heibsch1998semirings} for more details.
    Thus the notion of an ordered additively cancellative semiring is slightly more general than the one of a subsemiring of some ordered ring even when the underlying additive monoid $(\Lambda, +, 0, \le)$ is order-cancellative.
    Finally we note that when $\le_g$ is a ring-order, the natural ordered semiring homomorphism 
    \[\iota: (\Lambda, +, \cdot, 0, 1) \to (\Lambda-\Lambda, +, \cdot, 0,1, \le^g)\]
    is the reflection of $\Lambda$ in the category of ordered rings.
\end{remark}

\begin{remark}[Cone-ordered semirings]\label{rmk:cone-ordered_semirings}
    In the situation of Remark~\ref{rmk:semirings_vs_rings}, if $\Lambda$ is cone-ordered, then it is additively order-cancellative and furthermore the condition to have $\le^g$ be a ring order is satisfied. Thus much as for commutative monoids, a positively cone-ordered semiring is an ordered semiring that can be realized as a subsemiring of an ordered ring containing the positive cone of the ring.
\end{remark}

\begin{remark}\label{rmk:semiring_to_ring_action}
    If $\Lambda$ is a semiring, $M$ is an Abelian group and $\varphi: \Lambda \to (M,+, 0)$ is an action of $\Lambda$ on $M$, then there is a unique extension of $\varphi$ along the canonical $\iota: \Lambda \to \Lambda-\Lambda$, i.e.\ there is a unique ring homomorphism $\bar{\varphi}: (\Lambda-\Lambda) \to \End(M, +, 0)$ such that $\bar{\varphi}\circ \iota=\varphi$ (this is a consequence of the universal property of the ring of differences).
    Similarly if $\Lambda$, $M$, and $\varphi$ are ordered, and $(\Lambda-\Lambda, +, \cdot, 0,1, \le^g)$ is an ordered ring, the action $\bar{\varphi}$ will be ordered as well.
\end{remark}

\begin{definition}[Rings with powers]
    Let $\Lambda$ be a semiring. A \emph{$\Lambda$-powering} on a ring $(R, +, \cdot, 0,1)$ is a $\Lambda$-module structure $\pi: \Lambda \to \End(R^{\bullet\pi}, \cdot)$ on some multiplicative submonoid $R^{\bullet\pi}$ of $R$ called monoid of \emph{$\pi$-powerable elements} and the action is written in power notation $(\lambda, r) \mapsto r^{\pi(\lambda)}$. 
    
    A \emph{ring with $\Lambda$-powers} is a ring $R$ endowed with a $\Lambda$-powering $\pi$ (formally it is the pair $(R, \pi)$) and a \emph{ring with powers} is a triple $(R, \Lambda, \pi)$ where $R$ is a ring, $\Lambda$ is a semiring and $\pi$ is a $\Lambda$-powering of $R$.
    
    When the powering $\pi$ is clear from the context, we suppress the $\pi$ from the notation, so we write $R^\bullet$ for the $\pi$-powerable elements and $r^\lambda$ for $r^{\pi\lambda}$. Likewise sometimes we will say \emph{powerable} or \emph{$\Lambda$-powerable} for $\pi$-powerable.

    A ring with \emph{internal $\Lambda$-powers} is a ring with $\Lambda$-powers $R$ further endowed with a semiring homomorphism $\iota: \Lambda \to R$ (the \emph{internalization map}).
    A \emph{ring with internal powers} is then a quadruple $(\bK, \Lambda, \pi, \iota)$ where $\pi$ is a $\Lambda$-powering of $\bK$ and $\iota: \Lambda \to \bK$ is a semiring homomorphism.
    
    As usual we often omit the $\iota$ map when writing expressions for elements in $R$.
\end{definition}

\begin{example}\label{ex:powerings}
The following are natural examples of powering on rings. 
\begin{enumerate}
	\item every ring $R$ has a unique natural $\bN$-powering $\pi_0$ such that $R^{\pi_0 \bullet}=R$;
	\item every ring $R$ has a unique $\bZ$-powering $\pi_1$ such that $R^{\pi_1\bZ}=R^*$ is the group of units of $R$;
	\item a real closed field $K$ has a unique $\bQ$-powering $\pi_2$ s.t.\ $K^{\pi_2\bullet}=K^{>0}$.
	%\item the complex field $\bC$ can be $\bC$-powered with powers defined on the positive real semiline. %TODO: add other examples
\end{enumerate}
All of the examples of $\Lambda$-powerings above have an obvious internalization map.
\end{example}

For the purpose of the subsequent construction we will need the following condition on a c.a.c.\ semiring $\Lambda$.

\begin{definition}
    If $\Lambda$ is a commutative ring and $\Lambda_0$ is a subring, I will say that $\Lambda_0$ \emph{has binomial sequences in $\Lambda$} if for all $\lambda \in \Lambda_0$ and all $n \in \bN$ there is a \emph{unique} $x\coloneqq \binom{\lambda+n-1}{n} \in \Lambda$ such that $n!x=\prod_{i<n}(\lambda+i)$.
    
    If $\Lambda_0$ is just a subsemiring of $\Lambda$, I will say it \emph{has binomial sequences in $\Lambda$} if $\Lambda_0-\Lambda_0$ has binomials equences in $\Lambda$. 
    When $\Lambda$ is a c.a.c.\ semiring, we will say it has binomial sequences if it has binomial sequences in $\Lambda-\Lambda$.
\end{definition}

The Lemmas and Corollaries in the reminder of this subsection are quite simple and mostly well known, we include references for the ones we could find one for, but also the ones for which a reference is missing are unlikely to be original to the author, given they are quite basic.

\begin{lemma}\label{lem:binomial_existence}
    Let $\Lambda$ be a commutative ring and $\Lambda_0\subseteq \Lambda$ a subring. The following are equivalent:
    \begin{enumerate}
        \item for all $x \in \Lambda_0$, and $n \in \bN$, there is $y \in \Lambda$ such that $n!y=\prod_{i<n}(x-i)$;
        \item for all prime natural numbers $p$, and all $x \in \Lambda_0$, $x^p-x \in p\Lambda$ (i.e. $(\Lambda_0+p\Lambda)/p\Lambda$ is a $p$-Boolean algebra).
    \end{enumerate}
    \begin{proof}
        $(2) \Rightarrow (1)$ Note that letting $P$ denote the set of prime natural numbers we have
        \[n!=\prod_{p \in P} p^{v(p,n)} \quad \text{where} \quad v(p,n)=\sum_{i>0} \lfloor n/p^i\rfloor\]
        and the a priori infinite product has actually only finitely many non-one factors. Also note that given any $m \in \bN$ and $p \in P$ not dividing $m$ (in $\bZ$) we have $m\Lambda+ p^k\Lambda=\Lambda$ for every $k \in \bN$. Thus for any $n=\prod_{p \in S} p^{k_p}$ we have $n\Lambda=\bigcap_{p \in S}(p^{k_p}\Lambda)$ and we can reduce to prove that
        \begin{equation}\label{eq:goal0}
        \prod_{i<n} (x+i) \in p^{v(p,n)}\Lambda \qquad \text{for all} \quad p \in P,\; x \in \Lambda_0.   
        \end{equation}
        
        For this observe that (2) implies that for all $k \in \bN$, $\prod_{i<p^k}(x+i) \in p^k \Lambda$. For $k=1$ this follows from the familiar fact that $x^p-x-\prod_{i<p}(x+i)\in p\Lambda$. For larger $k$, we can derive the claim from the case $k=1$ applied to each element of $\{x+jp: j<p^{k-1}\}\subseteq \Lambda_0$.
        Thus for a given $p \in P$, $k\in \bN$, and $x \in \Lambda_0$
        \[|\{x+i : 0 \le i < n\} \cap p^k\Lambda|\ge \lfloor n/p^k\rfloor,\]
        from which we derive the validity of (\ref{eq:goal0}).

        $\lnot(1) \Rightarrow \lnot(2)$ Suppose that there is a $p\in P$ and $x \in \Lambda_0$ such that $x^p-x \notin p \Lambda$, then since $\prod_{i<p}(x+i)-x^p+x \in p\Lambda$ we have $\prod_{i<p}(x+i) \notin p \Lambda$.
    \end{proof}
\end{lemma}

\begin{corollary}\label{cor:binomial_char}
    Let $(\Lambda, +, \cdot, 0,1)$ be a commutative ring and $\Lambda_0$ be a subsemiring. Then $\Lambda_0$ has binomial sequences in $\Lambda$ if and only if the Abelian group $(\Lambda, +,0)$ has no torsion and for all prime numbers $p$ and all $x \in \Lambda_0$, $x^p-x \in p\Lambda$.
    \begin{proof}
        By Lemma~\ref{lem:binomial_existence} the only thing that needs to be proven is that if $\Lambda$ allows binomial sequences, then $(\Lambda, +, 0)$ has no torsion. For this suppose that there is $y \in \Lambda^{\neq0}$ is such that $py=0$ for some prime number $p$. Then for any $\lambda\in \Lambda_0$ the equation in $\var{x}$, $p!\var{x}=\prod_{i<p}(\lambda+i)$ cannot have a unique solution, because if $x$ is a solution, then so is $x+y$.
    \end{proof}
\end{corollary}

\begin{lemma}\label{lem:binom_semiring_powers_of_series}
    If $\Lambda$ is a c.a.c.\ semiring ring having binomial coefficients and $R=\Lambda-\Lambda$, then
    \[(\lambda, 1+h) \mapsto (1+h)^\lambda\coloneqq \sum_{n\in \bN}\binom{\lambda}{n}h^n\]
    defines a semiring action of $(\Lambda, +, \cdot, 0,1)$ on the monoid $\big(1+\var{x}R[\![\var{x}]\!], \cdot, 1\big)$.
    \begin{proof}
        It is not hard to compute that for the definition above $(1+h)^0=1$ and $(1+h)^1=1+h$. So we need to show that $(1+h)^{\lambda}(1+h)^\mu=(1+h)^{\lambda \mu}$ and $\big((1+h)^{\lambda}\big)^\mu=(1+h)^{\lambda\mu}$.
        Also note that it is enough to prove the statement for $h=\var{x}$ as precomposing with $h$ is an endomorphism of the semiring $(\Lambda[[\var{x}]], +, \cdot, 0, 1)$ and thus a fortiori an endomorphsim of the monoid $\big(1+\var{x}\Lambda[[\var{x}]], \cdot, 1\big)$.
        
        Note that from the hypothesis on $\Lambda$ it follows that $\Lambda[[\var{x}]]$ is a subsemiring of $R[[\var{x}]]$ where $R=(\Lambda-\Lambda)\otimes \bQ$, so the validity of $(1+\var{x})^{\lambda}(1+\var{x})^\mu=(1+\var{x})^{\lambda \mu}$ and $\big((1+\var{x})^{\lambda}\big)^\mu=(1+\var{x})^{\lambda\mu}$ in $\Lambda[[x]]$ would follow from the validity of the same identities in $R[[\var{x}]]$. Thus in proving the thesis, we can further reduce to the case in which $\Lambda=R$ is a $\bQ$-algebra.

        In such a case the validity of the claim can be derived by observing that the usual Taylor expansions of $\exp$ and $\log$, give isomorphisms $\log:1+\var{x}\Lambda[[\var{x}]] \leftrightarrow \var{x}\Lambda[[\var{x}]]:\exp$. So $(\lambda, 1+h) \mapsto \exp(\lambda\log(1+h))$ yields a $\Lambda$-action. But then, since in $\bQ[\var{y}][[\var{x}]]$ we have
        \[\exp(\var{y}\log(1+\var{x}))=\sum_{m\in \bN} \frac{\var{x}^m}{m!}\prod_{i<m} (\var{y}-i),\]%TODO:Check
        it follows that $\exp(\lambda\log(1+\var{x}))=(1+\var{x})^\lambda$ as defined above; in particular $(\lambda, 1+h)\mapsto (1+h)^\lambda$ defines a $\Lambda$-action.
    \end{proof}
\end{lemma}

When $\Lambda$ is an ordered c.a.c.\ semiring with binomial sequences and $(\bK, \pi, \iota)$ is a ring with internal $\Lambda$-powers, we can naturally define a $\Lambda$-powering and an internalization map on any ring of series $\bK \Lhp \fM \Rhp$ where $\fM$ is an ordered cancellative $\Lambda$-module.

\begin{lemma}\label{lem:powers_for_series}
    If $\Lambda$ is a c.a.c.\ ordered semiring with binomial sequences, $(\bK, \pi, \iota)$ is a ring with internal $\Lambda$-powers, and $\fM$ is an ordered cancellative $\Lambda$-module, then $\bK\Lhp \fM \Rhp$ has itself a natural structure $(\tilde{\pi}, \tilde{\iota})$ of ring with internal $\Lambda$-powers, where:
    \begin{enumerate}
        \item the powerable elements are given by
        \[\bK\Lhp \fM\Rhp^{\tilde{\pi}\bullet}\coloneqq \{k\fm (1+h): k \in \bK^{\bullet\pi}, \fm \in \fM, h \in \bK \Lhp \fM^{<1}\Rhp\};\]
        \item the $\Lambda$-powers of a series $f=k\fm(1+h) \in \bK\Lhp \fM\Rhp^{\tilde{\pi}\bullet}$ are given by
        \[f^\lambda=f^{\tilde{\pi}\lambda}\coloneqq k^\lambda\fm^\lambda \sum_{n} h^n\iota\binom{\lambda}{n} \quad \text{for all}\; \lambda \in \Lambda.\]
        \item the internalization map $\tilde{\iota}$ is just the composition of $\iota$ with the natural inclusion $\bK \subseteq \bK \Lhp \fM \Rhp$.
    \end{enumerate}%TODO, prove and check
    \begin{proof}
        Note that in (1), the fact that the normal series with powerable leading coefficients precisely correspond to the series of the form $k \fm (1+h)$ with $k \in \bK^{\bullet}$, $\fm \in \fM$, $h \in \bK \Lhp \fM^{<1}\Rhp$ relays on the hypothesis that $\fM$ is negatively cone-ordered.
        Note that that in (2), the series $(1+h)^\lambda$ is well defined by Lemma~\ref{lem:binom_semiring_powers_of_series} and because $(h^n)_{n \in \bN}$ is summable by Corollary~\ref{cor:neumann}.
        The only things left to be checked are that the powerable series are a multiplicative monoid and that $(\lambda, k\fm (1+h))\mapsto k^\lambda \fm^\lambda (1+h)^\lambda$ is an action of $\Lambda$, which readily follow from Lemma~\ref{lem:binom_semiring_powers_of_series}.
    \end{proof}
\end{lemma}

\begin{remark}
    Note that if $\Lambda=\bZ$, $\bK$ is a domain and it is given the natural $\bZ$-powering of Example~\ref{ex:powerings}(2), (so $\bK^{\bullet}=\bK^*$), then Remark~\ref{rmk:units} implies that $\bK\Lhp \fM \Rhp^{\bullet}=\bK \Lhp \fM \Rhp^*$.
\end{remark}

\begin{remark}
    Note that in Lemma~\ref{lem:powers_for_series}, if $\fM$ is cone-ordered, then $\bK\Lhp \fM \Rhp$ consists exactly of the normal series with powerable leading coefficient.
\end{remark}

\begin{remark}\label{rmk:comment_on_ring_or_discrete}
    At some point below and in the next subsections (Definition~\ref{def:serially_adequate} and \ref{def:formal_derivatives}) we will also require from the c.a.c.\ semiring $(\Lambda, +, \cdot, 0,1)$ that it satisfies $1+\Lambda\supseteq \Lambda^{\neq 0}$.
    This really is better seen as a disjunction of two conditions:
    \begin{enumerate}
        \item either $1+\Lambda =\Lambda$ in which case $-1 \in \Lambda$ and thus $\Lambda=\Lambda-\Lambda$ is a ring;
        \item or $1+\Lambda=\Lambda^{\neq0}$ in which case we have for all $n \in \bN$, that $n+\Lambda=\Lambda \setminus \{i\in \bN: i<n\}$ from which it follows that for every $x \in \Lambda \setminus \bN$, $x+\bZ \subseteq \Lambda$.
    \end{enumerate}
    It is worth observing that in case (2), in the difference ring $R=\Lambda-\Lambda$ we have $\Lambda \cap (-\Lambda)=0$, so the algebraic preorder on $R$, given by $y\ge x\Leftrightarrow y-x\in \Lambda$ is in fact an order (possibly partial), thus case (2) just amounts to $\Lambda$ being the positive cone of a discreetly ordered ring.
    Also, the condition $1+\Lambda\supseteq \Lambda^{\neq 0}$ is quite independent from the requirement of allowing binomial coefficients.%ELABORATE
\end{remark}

Under some additional assumptions on the ring with internal $\Lambda$-powers $\big(\bK, \pi: \Lambda \to \End(\bK^{\bullet}, \cdot), \iota: \Lambda \to \bK\big)$ we can slightly extend the $\Lambda$-powering on rings of series given in Lemma~\ref{lem:powers_for_series}.

\begin{definition}\label{def:serially_adequate}
    A ring with internal powers $(\bK, \Lambda, \pi, \iota)$ will be said to be \emph{serially adequate} if $\bK^\bullet$ contains no zero-divisor, $\Lambda$ allows binomial coefficients and one of the two conditions hold:
    \begin{enumerate}
        \item $1+\Lambda \supseteq \Lambda^{\neq 0}$;
        \item all elements of $\bK^\bullet$ are units of $\bK$.
    \end{enumerate}
    Note that by Remark~\ref{rmk:comment_on_ring_or_discrete}, the condtion ``(1) or (2)'' above can be reformulated as ``$1+\Lambda = \Lambda ^{\neq0}$ or all elements of $\bK^\bullet$ are units of $\bK$''.
\end{definition}

\begin{corollary}\label{cor:powers_for_series}
    If $(\bK, \Lambda, \pi, \iota)$ is serially adequate, then we can extend the powering on $\bK \Lhp \fM \Rhp$ given in Lemma~\ref{lem:powers_for_series} by defining the set of $\Lambda$-powerable elements as
    \[\bK \Lhp \fM \Rhp^{\bullet}\coloneqq \{\fm(k+ h): k \in \bK^{\bullet}, \fm \in \fM, h \in \bK \Lhp \fM^{<1} \Rhp\}\]
    and setting
    \[\big(\fm(k+h)\big)^{\lambda}\coloneqq\fm^{\lambda} \sum_{n} \iota\binom{\lambda}{n} h^n k^{\lambda-n}.\]
    \begin{proof}
        Since by the hypothesis $\bK$ \emph{embeds} into the localization of $\bK$ by $\bK^\bullet$, passing to $\bE\Lhp \fM \Rhp^\bullet$ we can reduce to the case in which all elements of $\bK^\bullet$ are invertible. In such a case we have that the definition of $\bK\Lhp \fM\Rhp^\bullet$ and of powers match the one of Lemma~\ref{lem:powers_for_series} so we can conclude.
    \end{proof}
\end{corollary}

\subsection{\texorpdfstring{$\Lambda$}{Λ}-powerings and derivatives of \texorpdfstring{$\Lambda$}{Λ}-power series}\label{ssec:ders}

%TODO in this section: the definition of interpretation in the classical variables does not require adequateness, maybe add a comment

In order to define the formal derivatives of a $\Lambda$-power series $f$ in some variables $\var{x}$ we generally need some conditions on the semiring $\RExp(f, \var{x})$ (which are always satisfied when $\Lambda=\bN$ or when the variables $\var{x}$ are classical in $f$). Such an assumption is not necessary to define the renormalized formal derivatives.

\begin{definition}[Formal derivatives]\label{def:formal_derivatives}
    If $\var{x}$ is a single variable, $\var{y}$ is a set of variables not containing $\var{x}$, $f \in \bK \Lhp \var{x}^\Lambda \var{y}^{\Lambda} \Rhp$, and $\RExp(\var{x}, f)\setminus \{0\}\subseteq 1+\Lambda$, then $\partial_{\var{x}}f$ denotes the \emph{formal derivative} of $f$ in $\var{x}$ which is defined as follows. 
    If $f= \sum_{\alpha, \beta} f_{\alpha, \beta} \var{x}^\alpha \var{y}^\beta$ with $f_{\alpha, \beta} \in \bK$, we set
	\[\partial_\var{x} f := \sum_{\alpha, \beta} c_{\alpha, \beta} \alpha \var{x}^{\alpha-1} \var{y}^\beta=\sum_{\alpha, \beta} c_{\alpha,\beta} \iota(\alpha) \var{x}^{\alpha-1} \var{y}^\beta.\]
    Note that the hypothesis $\RExp(\var{x}, f)\setminus \{0\}\subseteq 1+\Lambda$ implies that $\iota(\alpha) \var{x}^{\alpha-1} \in \iota(\Lambda) \var{x}^\Lambda$ because if $\alpha \in \RExp(\var{x},f)$ and $\alpha \notin 1+\Lambda$, then $\alpha=0$. 
    Similarly from $h \in \bN$, $\partial_\var{x}^h f$ denotes the $n$-th order derivative of $f$ in the variable $\var{x}$.
    
    More generally if $\var{x}=\{\var{x}_0, \ldots, \var{x}_{m-1}\}$ is a finite set of variables and $f(\var{x}) =\sum_{\alpha\in \Lambda^\var{x}}f_\alpha \var{x}^\alpha \in \bK \Lhp\var{x}^\Lambda\Rhp$ and $\RExp(x,f) \setminus \{0\} \subseteq 1+\Lambda$ we will adopt the following notation for iterated partial derivatives: if $h \in \bN^{\var{x}}$ is the assignment $\var{x}_i \mapsto h_i \in \bN$, then we set
    \[\partial^hf\coloneqq\partial_{\var{x}_0}^{h_0} \cdots \partial_{\var{x}_{m-1}}^{h_{m-1}} f = \sum_{\alpha \in \Lambda^{\var{x}}} f_{\alpha} \cdot (\alpha)_{h} \cdot \var{x}^{\alpha-h},\]
    where $\alpha \in \Lambda^\var{x}$ and $h\in \bN^{\var{x}}$ are subtracted component-wise and $(\alpha)_h$ is the product of falling factorials
    \[(\alpha)_h= \prod_{i<m} \prod_{j<h_i} (\alpha(\var{x}_i)-j).\]
    %Again the hypothesis on $\Lambda$ ensures that $\iota\binom{\alpha}{h}\var{x}^{\alpha}\in \iota(\Lambda)\var{x}^\Lambda$.

    If $\RExp(f, \var{x})$ has binomial sequences in $\Lambda-\Lambda$ and $\RExp(f, \var{x})\setminus \{0\} \subseteq 1+\Lambda$, we can also define the formal \emph{Hasse-Schmidt} derivative of multi-index $h\in \bN^{\var{x}}$ %(or also the renormalized derivative of multi-index $\var{x}$)
    which is instead defined as
    \[\HSder^h(f)=\sum_{\alpha \in \Lambda^\var{x}} f_\alpha \cdot \iota \binom{\alpha}{h} \var{x}^{\alpha-h}.\]
    Note the relation $h! \HSder^hf = \partial^h f$ where $h!\coloneqq\prod_{j<m} h_j!$, and note that if $\bK$ has positive characteristic it might be that $h!=0$ in $\bK$.
    %We will use $\HSder$ essentially as a way to bypass some nuances that appear in the positive characteristic case.
    We extend the notations $\partial^h$ and $\HSder^h$ to the case $h \in \bN^{\var{y}}$ for $\var{y}\subseteq \var{x}$, in the obvious way: namely we regard $f$ as an element of $\bE\Lhp \var{y}^{\Lambda}\Rhp$ with $\bE=\bK\Lhp (\var{x}\setminus \var{y})^\Lambda\Rhp$.

    Finally for $n\in \bN$ and $\var{y}$ a single variable, simliarly to the notation $\partial_\var{x}^n(f)$, we write $\HSder^n_\var{y}(f)$ for $\HSder^{[\var{y}\mapsto n]}(f)$. Thus for $\var{x}=\{\var{x}_0, \ldots, \var{x}_{m-1}\}$ and $h\in \bN^\var{x}$ the assignment $\var{x}_i \mapsto h_i$ we can write $\HSder^h(f)=\HSder_{\var{x}_0}^{h_0} \cdots \HSder_{\var{x}_{m-1}}^{h_{m-1}}(f)$.
\end{definition} 

As anticipated, defining the renormalized formal derivatives of any order requires less control on the supports of the series.

\begin{definition}[Renormalized formal derivatives]\label{def:rn_formal_derivatives}
The renormalized formal derivative of order $h \in \bN^\var{x}$ of a series $f=\sum_{\alpha \in\Lambda^{\var{x}}} f_\alpha \var{x}^\alpha$ with $f_\alpha \in \bK$ is given by
\[\var{x}^h\partial^hf=\sum_{\alpha \in\Lambda^{\var{x}}} f_\alpha \cdot (\alpha)_h \cdot \var{x}^{\alpha}.\]
Note that we consider it to be defined also when $\partial^hf$ might not be defined.

Similarly when $\RExp(f, \var{x})$ has binomial sequences in $\Lambda$, we consider the renormalized Hasse-Schmidt derivative of order $h\in \bN^{\var{x}}$
\[\var{x}^h\HSder^h(f)=\sum_{\alpha \in \Lambda^\var{x}} f_\alpha \cdot \binom{\alpha}{h} \cdot \var{x}^{\alpha},\]
to be defined also when the condition $\RExp(f, \var{x})\setminus \{0\} \subseteq 1+ \Lambda$ is not satisfied.
\end{definition}

\begin{remark}\label{rmk:ring_rnder-closed}
    If $\cA\subseteq \bK \Lhp \fM \times \var{x}^\Lambda \Rhp$ is a subring, then it is closed under renormalized derivatives of any order if and only if it is closed under renormalized derivatives of order $1$ in all variables. In fact assuming $\var{x}$ is a single variable we have $\var{x}^{n+1}\partial_\var{x}^{n+1}(f)=\var{x}\partial_{\var{x}}(\var{x}^n\partial_{\var{x}}^n(f))-n\var{x}^n$.
\end{remark}

\begin{lemma}\label{lem:mon-div_closure_preserves_almostHSder-closed}
    Let $\Lambda$ be a c.a.c.\ ordered semiring with binomial coefficients and suppose that $\cA\subseteq \bK\Lhp \fM \times \var{x}^\Lambda\Rhp$ is a $(\Lambda-\Lambda)$-submodule closed under taking renormalized Hasse-Schmidt derivatives. Then the smallest set containing $\cA$ and closed under monomial division, is still closed under taking renormalized Hasse-Schmidt derivatives.
    \begin{proof}
        Suppose that $\var{z}\in \var{x}$, $\alpha \in \Lambda$, and $f=\var{z}^\alpha g\in \cA$ with $g$ non-singular. It suffices to show that then for all $n \in \bN$, $\var{z}^\alpha\var{z}^n\HSder^n_{\var{z}}(g)\in \cA$. This can be seen by induction: it is trivial for $n=0$ and the inductive step is dealt with observing that 
        \[\begin{aligned}
            \cA \ni \var{z}^n \HSder_{\var{z}}^n(f)= \var{z}^{n}\HSder_{\var{z}}^n(\var{z}^\alpha g) = \sum_{k} \binom{n}{k}\var{z}^{n-k}\HSder_{\var{z}}^{n-k}(\var{z}^\alpha) \var{z}^k\HSder_{\var{z}}^k(g)= \\
            =\var{z}^\alpha\sum_{k} \binom{n}{k}\binom{\alpha}{n-k} \var{z}^k \HSder_{\var{z}}^k(g) = \var{z}^\alpha\var{z}^n \HSder^n_{\var{z}}(g) + \sum_{k<n} \binom{n}{k}\binom{\alpha}{n-k} \var{z}^\alpha\var{z}^k \HSder_{\var{z}}^k(g),
        \end{aligned}\]
        and noting that by the inductive hypothesis the rightmost sum in the last expression is in $\cA$.
    \end{proof}
\end{lemma}

\begin{lemma}\label{lem:leibniz_for_rn}
    If $\Lambda$ be a c.a.c.\ ordered semiring with binomial coefficients and suppose that $\cA\subseteq \bK\Lhp \fM \times \var{x}^\Lambda\Rhp$ is a subset closed under renormalized Hasse-Schmidt derivatives, then the subring generated by $\cA$ is closed under renormalized Hasse-Schmidt derivatives.
    \begin{proof}
        The subgroup generated by $\cA$ will still be closed because renormalized Hasse-Schmidt derivative opeartors are $\bK\Lhp \fM \Rhp$-linear. The closure under renormalized Hasse Schmidt derivatives of the generated subring then follows from the identity
        \[\var{z}^n\HSder^n_{\var{z}}(fg)=\sum_{k}\binom{n}{k}\cdot \var{x}^{n-k}\HSder_{\var{z}}^{n-k}(f)\cdot \var{x}^k\HSder_{\var{z}}^{k}(g),\]
        valid for all $n \in \bN$ and $\var{z} \in \var{x}$.
    \end{proof}
\end{lemma}

\subsection{Interpretations and compositions of \texorpdfstring{$\Lambda$}{Λ}-power series}\label{ssec:compositions}
Throughout this subsection $(\bK, \Lambda, \pi, \iota)$ will be assumed to be a serially adequate commutative ring with internal powers.
In order to interpret $\Lambda$-power series as functions on $\bK \Lhp \fM \Rhp$ we will add another assumption on $\fM$ and $\Lambda$: namely that $\fM$ and $\var{x}^{\Lambda}$ satisfy the property (\ref{eq:ppty-adequate}) given in Remark~\ref{rmk:monomial_substitution}.

\begin{definition}\label{def:adequate}
    An ordered semiring $\Lambda$ is an \emph{adequate semiring of exponents} if $\Lambda$ allows binomial coefficients and satisfies
    \[\forall \lambda\in \Lambda^{>0}, \lambda\cdot -: \Lambda \to \Lambda \text{ is one-to-one}\]
    We then call a $\Lambda$-module $\fM$ an \emph{adequate $\Lambda$-module of monomials} if it is cancellative and satisfies condition (\ref{eq:ppty-adequate}). We will say that the ring with internal powers $(\bK, \Lambda, \pi, \iota)$ is \emph{adequate} if it is serially adequate and $\Lambda$ is an adequate semiring of exponents.
    
    On such $(\bK, \Lambda, \pi, \iota)$ we consider the ``extended'' $\Lambda$-powering given in Corollary~\ref{cor:powers_for_series}, for which the powerable elements are
    \[\bK\Lhp \fM \Rhp^\bullet=\{\fm (k+\varepsilon): \varepsilon \in \bK \Lhp \fM \Rhp^{\prec 1}, k \in \bK^\bullet, \fm \in \fM\}.\]
\end{definition}

\begin{remark}
    Note that a product of adequate $\Lambda$-modules is still adequate.
\end{remark}

\begin{definition}[Interpretations]
\label{defn:genps_interpretation}
	Suppose that $(\bK, \Lambda, \pi, \iota)$ is an adequate ring with powers and $\fM$ is an adequate $\Lambda$-module of monomials. We call the powerable infinitesimal series of $\bK\Lhp \fM \Rhp$ \emph{formally $\Lambda$-composable} and write their set as $\bK\Lhp \fM \Rhp^{\square}\coloneqq \bK \Lhp \fM \Rhp^{\bullet}\cap \bK \Lhp \fM\Rhp^{\prec 1}$.

    Let $m,n \in \bN$ and $\var{x}=\{\var{x}_0, \ldots, \var{x}_{n-1}\}, \var{y}=\{\var{y}_0, \ldots \var{y}_{m-1}\}$ be finite disjoint subsets of $\Var$ of sizes $n$ and $m$ respectively. Let $f \in \bK\Lhp \fM \times \var{x}^\Lambda \var{y}^\Lambda\Rhp$ be classical in $\var{y}$ so that $f$ has the form
    \[f=\sum_{\substack{i \in \bN^{\var{x}}\\\alpha \in \Lambda^{\var{y}}}} f_{i, \alpha}\var{x}^{n} \var{y}^{\alpha} \quad \text{with}\; f_{i, \alpha}\in \bK \Lhp \fM\Rhp\; \text{and}\; \bigcup_{\substack{i \in \bN^{\var{x}}\\\alpha \in \Lambda^{\var{y}}}} \Supp(f_{n, \alpha}) \; \text{Noetherian}.\]

    Given such a series $f$ we can interpret it as a $\bK \Lhp\fM \Rhp$-valued function on the set of assignements $\big(\bK\Lhp \fM \Rhp^{\prec 1}\big)^{\var{x}} \times \big(\bK \Lhp \fM \Rhp^{\square}\big)^{\var{y}}$, by setting for $g \in \big(\bK\Lhp \fM \Rhp^{\prec 1}\big)^{\var{x}}$ and $h \in \big(\bK \Lhp \fM \Rhp^{\square}\big)^{\var{y}}$

    \[f(g,h)\coloneqq \sum_{\substack{i \in \bN^{\var{x}}\\\alpha \in \Lambda^{\var{y}}}} f_{i, \alpha} g^i h^\alpha = \sum_{\substack{i \in \bN^{\var{x}}\\\alpha \in \Lambda^{\var{y}}}} \sum_{j \in \bN^{\var{y}}} \iota \binom{\alpha}{j} f_{n, \alpha}g^i k^{\alpha-j}\fm^\alpha \varepsilon^j,\]
    where $k \in (\bK^\bullet)^{\var{y}}$, $\fm \in (\fM^{<1})^{\var{y}}$ and $\varepsilon \in \big(\bK \Lhp \fM\Rhp^{\prec 1}\big)^{\var{y}}$ are such that for component-wise operations $h=\fm(k+\varepsilon)$. 

    The sum is well defined and the equality with the rightmost expression holds because because by Corollary~\ref{cor:neumann} and Remark~\ref{rmk:monomial_substitution}, the family $\big(f_{\gamma, i}\fm^\gamma\varepsilon^jg^i: i \in\bN^{\var{x}}, j \in \bN^{\var{y}}, \gamma \in \Lambda^{\var{x}}\neq 0\big)$ is summable.
			
	Given $X \subseteq \bK \Lhp \fM \Rhp$, we will denote by $\langle X \rangle_\cF$, the $\cF\cup \{+, \cdot\}$-substructure of $\bK \Lhp \fM \Rhp$ generated by $X$ where the symbols in $\cF$ are interpreted as above.
\end{definition}

\begin{remark}
    Note that in the situation above, if say $\var{y}_0$ also happens to be classical in $f$, then the above definition gives also an interpretation of $f$ as function on $\big(\bK\Lhp \fM \Rhp^{\prec 1}\big)^{\var{x} \cup\{ \var{y_0}\}}\times \big(\bK \Lhp \fM \Rhp^{\square}\big)^{\var{y} \setminus \{\var{y_0}\}}$. This new interpretation agrees with the previous interpretation on $\big(\bK\Lhp \fM \Rhp^{\prec 1}\big)^{\var{x}} \times \big(\bK \Lhp \fM \Rhp^{\square}\big)^{\var{y}}\subseteq \big(\bK\Lhp \fM \Rhp^{\prec 1}\big)^{\var{x} \cup\{ \var{y_0}\}}\times \big(\bK \Lhp \fM \Rhp^{\square}\big)^{\var{y} \setminus \{\var{y_0}\}}$.
\end{remark}

\begin{remark}\label{rmk:on_terminology}
    When $\bK$ is a field, the use of the term \emph{restricted} for what we call weakly restricted $\Lambda$-power series is consistent with its use in rigid geometry.%add reference... see e.g. ...
    
    By \cite[Lem.~5.23 and Thm.~5.29]{freni2023vector} the summability of the family $\big(f_{i, \alpha} g^i h^\alpha: i \in \bN^{\var{x}}, \alpha \in \Lambda^\bN\big)$ is equivalent to the convergence of the net of finite partial sums when $\bK \Lhp \fM \Rhp$ is given the $\bK$-linear topology described in \cite[Sec.~5.2]{freni2023vector}.
        
    Thus in particular, When $\Lambda=\bN$, and $\fM$ is a totally ordered group (whence $\bK \Lhp \fM \Rhp$ is naturally a valued field whose valuation ring is $\bK \Lhp \fM\Rhp^{\preceq 1}$ and whose valuation ideal is $\bK \Lhp \fM \Rhp^{\preceq 1}$), the weakly restricted power series are precisely those that converge on the infinitesimal ball. It is worth mentioning though that if $\fM$ is not isomorphic to $\bZ$, then the $\bK$-linear topology considered here on $\bK \Lhp \fM \Rhp$ would not be the valuation topology. We won't make any use this observation.
\end{remark}

\begin{remark}[Compositions]\label{rmk:compositions}
    Note the special case of Definition~\ref{defn:genps_interpretation} when $\fM=\fN \times \var{z}^\bN\var{t}^\Lambda$ for some disjoint sets of variables $\var{z}, \var{t}$ and some ordered cancellative $\Lambda$-module $\fN$. Given $f \in \bK \Lhp \fN \times\var{x}^{\bN}\var{y}^\Lambda\Rhp \cong \bK \Lhp \fN \times \{1\} \times \var{x}^{\bN} \var{y}^\Lambda\Rhp \subseteq \bK \Lhp \fM \times \var{x}^\bN\var{y}^\Lambda\Rhp$, $g \in \big(\bK\Lhp \fN \times \var{z}^\bN\var{t}^\Lambda\Rhp^{\prec 1}\big)^{\var{x}}$, and $h \in \big(\bK\Lhp\fN \times \var{z}^\bN\var{t}^\Lambda\Rhp^{\square}\big)^{\var{y}}$ we get a composition $f(g,h)\in \bK \Lhp \fN \times \var{z}^\Lambda\var{t}^\bN\Rhp$.
    %TODO: change to associativity
    This has the property that for all $(z,t) \in \big(\bK \Lhp \fN \Rhp^{\prec 1}\big)^{\var{z}} \times \big(\bK \Lhp \fN \Rhp^{\square}\big)^{\var{t}}$, $(f(g,h))(z,t)=f(g(z,t), h(z,t))$. %TODO:change and make into a fact/lemma
\end{remark}

\begin{remark}[Blow-ups]\label{rmk:blow-ups}
    Reindexings are special examples of composition.
    Another special example of composition is given by \emph{blow-ups}, i.e.\ compositions with powerable series of the form $\var{z}\var{t}$ and $\var{z}(k+\var{t})$ where $k \in \bK^\bullet$ and $\var{z}, \var{t}$ are variables (not necessarily distinct).
    Given $n \in \bN$, a set $\var{x}=\{\var{x}_0, \ldots \var{x}_{n-1}\}\subseteq \Var$ of $n$ distinct variables, a finite $\var{y} \subseteq \Var\setminus \var{x}$, and $f = \sum_{\alpha, \beta}f_{\alpha, \beta}\var{x}^\alpha\var{y}^\beta \in \bK\Lhp \fM \times \var{x}^\Lambda\var{y}^\Lambda\Rhp$ with $f_{\alpha, \beta}\in \bK \Lhp \fM \Rhp$, we note that the blow-ups of $f$ have the following forms
	\[\begin{aligned}
	    &f[\var{x}_i\mapsto \var{z}_i(k_i+\var{t}_i)]_{i<n} =  \sum_{\alpha, \beta, m} f_{\alpha, \beta} k^{\alpha-m} \binom{\alpha}{m} \var{z}^\alpha \var{t}^m \var{y}^\beta =\\
        &=\sum_{m\in \bN^{\var{x}}} (\var{x}^m[\var{x}_i \mapsto t_i]_{i<n})\cdot \big((\HSder^mf)[\var{x}_i \mapsto k_i\var{z}_i]_{i<n}\big)\\
        &f[\var{x}_i \mapsto\var{z}_i \var{t}_i]_{i<n} =  \sum_{\alpha, \beta, m} f_{\alpha, \beta} \var{z}^\alpha \var{t}^\alpha \var{y}^\beta,
	\end{aligned}\]
    %careful here!
    where for all $i<n$, $k_i\in\bK^\bullet$, $\var{z}_i, \var{t}_i \in\Var$.
    Note that composing with a blow-up of the form $\var{z}(k+\var{t})$ where $k \in \bK^\bullet$ and $\var{z} \neq \var{t}$, yields a series in which the variable $\var{t}$ is classic.
    Also note that if $f$ as above lies in $\bK \Lhp \fM \times \var{x}^\Lambda \var{y}^\bN\Rhp$, $h \in \big(\bK \Lhp \fN \Rhp^{\prec 1}\big)^\var{y}$, and $g=[\var{x}_i\mapsto k_i\fn_i(1+\varepsilon_i)]_{i<n} \in \big(\bK \Lhp \fN \Rhp^{\square}\big)^\var{x}$, with $k_i \in \bK^\bullet$, $\fn_i \in \fN$, and $\varepsilon_i \in \bK \Lhp \fN \Rhp^{\prec1}$ for all $i<n$, then the composition $f(g,h)$ can be written as
    \[f(g,h)=\big(f(b, h)\big)[\var{z}_i\mapsto \fn_i, \var{t}_i \mapsto \varepsilon_i]_{i<n},\]
    where $b$ is the blow-up $[\var{x}_i \mapsto \var{z}_i(k_i+\var{t}_i)]_{i<n}$.
\end{remark}

\begin{remark}
    If $\fM$ is totally cone-ordered, then $\bK\Lhp \fM\Rhp^\bullet=\{f \in \bK\Lhp \fM \Rhp: \lc(f)\in \bK^\bullet\}$, in particular if $\bK$ is ordered and $\bK^\bullet=\bK^{>0}$, then $\bK\Lhp \fM\Rhp^\bullet=\bK\Lhp \fM \Rhp^{>0}$ so $\bK\Lhp \fM \Rhp^{\square}$ consists of the positive infinitesimals.
\end{remark}

\subsection{Closure properties and main results}\label{ssec:mainres_tc_statements}
Throughout this subsection $\Lambda$ will be a c.a.c.\ ordered semiring, $\fM$ will be a multiplicatively written cancellative $\Lambda$-module, and $(\bK, \pi, \iota)$ will be a commutative ring with internal $\Lambda$-powers. Before stating the main result we define and recollect from previous sections some closure properties.

\begin{definition}\label{defn:closure_ppties}
    Let $\cA(\var{x})$ be a subset of $\bK \Lhp \fM \times \var{x}^\Lambda\Rhp$ for some finite $\var{x} \subseteq \Var$.
    \begin{enumerate}[label = (\arabic*), ref = (\arabic*)]
        \item \label{psFppties:tc} $\cA(\var{x})$ is \emph{truncation-closed} if whenever $f \in \cA(\var{x})$ all truncations of $f$ are in $\cA(\var{x})$;
        \item \label{psFppties:ders_c} $\cA(\var{x})$ is \emph{$\HSder$-closed for the classical variables} if whenever $\var{y}\subseteq \var{x}$, $m \in \bN^\var{y}$, and $f \in \cA(\var{x}) \cap \bK \Lhp \fM \times \var{y}^\bN(\var{x}\setminus \var{y})^\Lambda\Rhp$, we have $\HSder^m(f) \in \cA(\var{x})$;
        \item \label{psFppties:md} $\cA(\var{x})$ is \emph{closed under monomial division}, if for all $\fm\in \fM, \var{x}^\alpha\in \var{x}^\Lambda$ and $f \in \bK\Lhp \fM \times \var{x}^\Lambda\Rhp^{\preceq1}$, if $\fm \var{x}^\alpha f \in \cF(\var{x})$ then $f \in \cF(\var{x})$ and $\fm\var{x}^\alpha \in \cF$.
    \end{enumerate}
    When $\Lambda$ has binomial sequences we also consider:
    \begin{enumerate}[label = (\arabic*), ref = (\arabic*), resume]
        %\item \label{psFppties:ders} $\cA(\var{x})$ is \emph{$\HSder$-closed} (resp.\ \emph{derivation-closed}) if it is closed under taking Hasse-Schmidt derivatives of all multinidexes (resp.\ is closed under taking formal derivatives in any variable);
	    \item \label{psFppties:der} $\cA(\var{x})$ is \emph{almost $\HSder$-closed}
        (resp.\ \emph{almost derivation-closed}) if it is closed under renormalized Hasse-Schmidt derivatives of any index, i.e.\ for all $m \in \bN^{\var{x}}$, if $f \in \cA(\var{x})$, then $\var{x}^m\HSder^m(f) \in \cA(\var{x})$  (resp.\ if it is closed under renormalized formal derivatives i.e.\ for every $f \in \cA(\var{x})$, and $\var{x}_i \in \var{x}$, the series $\var{x}_i\partial_{\var{x}_i}(f)$ is in $\cA(\var{x}))$.
    \end{enumerate}
    We will extend the terminology of the list above also to the case when $\cF := \big(\cF(\var{x})\big)_{\var{x} \subseteq \Var}$ is a family of weakly restricted $\Lambda$-power series with coefficients from $\bK\Lhp \fM \Rhp$, convening that $\cF$ has a property when $\cF(\var{x})$ has it for all finite $\var{x}\subseteq \Var$. On $\cF$ we will also consider the following properties:
	\begin{enumerate}[label = (\arabic*), ref = (\arabic*), resume]
        \item \label{psFppties:scoperad} $\cF$ is \emph{closed under classical compositions} if it contains $\Var$ and for all finite disjoint sets of variables $\var{x}, \var{y}, \var{z}\subseteq \Var$, $\var{x} \subseteq \cF(\var{x})$ and for all $f \in \cF(\var{x}, \var{y}) \cap \bK \Lhp \fM \times \var{x}^\bN\var{y}^\Lambda\Rhp$, and for all $g \in \big(\cF(\var{z})^{\prec 1}\big)^{\var{x}}$, $f(g, \var{y})\in \cF(\var{y},\var{z})$;
	    \item \label{psFppties:implicit} $\cF$ has \emph{implicit functions} if for all $\var{x}\in \Var$, finite $\var{y} \subseteq \Var\setminus \{\var{x}\}$, and $f\in \cF(\var{x}, \var{y}) \cap \bK \Lhp \fM \times \var{x}^\bN\var{y}^\Lambda\Rhp^{\prec 1}$, such that $(\partial_\var{x} f)[\var{x}\mapsto 0]$ is a unit of $\bK \Lhp \fM \times \var{y}^\bN\Rhp^{\preceq 1}$, there is a (necessarily unique) $g \in \cF(\var{x})$, such that $f[\var{x}\mapsto g]=0$.
    \end{enumerate}
    When $(\bK, \Lambda, \pi, \iota)$ is adequate and $\fM$ is an adequate $\Lambda$-module of monomials we will also consider:
    \begin{enumerate}[label = (\arabic*), ref = (\arabic*), resume]
	    \item \label{psFppties:bu} $\cF$ is \emph{$S$-blow-up closed} (for $S \subseteq \bK^\bullet$ a submonoid) if it is closed under composing with series of the form $\var{z}_0(k+\var{z}_1)$ for $k \in S$ and $z_0z_1$ for $\var{x} \in\Var$, $\var{y} \subseteq \Var$, see Remark~\ref{rmk:blow-ups};
	    \item \label{psFppties:soperad} $\cF$ is \emph{closed under compositions} if it contains $\Var$ (empty compositions) and for all finite disjoint sets of variables $\var{x}, \var{y}, \var{z}\subseteq \Var$, $\var{x} \subseteq \cF(\var{x})$ and for all $f \in \cF(\var{x}, \var{y}) \cap \bK \Lhp \fM \times \var{x}^\bN\var{y}^\Lambda\Rhp$ and all $g \in \big(\cF(\var{z})^{\prec 1}\big)^{\var{x}}$, $h\in \big(\cF(\var{x})^{\square}\big)^{\var{y}}$, $f(g, h)\in \cF(\var{z})$;
    \end{enumerate}
    We will say that $\cF$ is \emph{almost fine} if it is truncation closed and almost $\HSder$-closed. We will say that $\cF$ is \emph{fine} if it is almost fine, closed under monomial division and contains $\fM$ and $\Var^{(\Lambda)}$.
\end{definition}

\begin{remark}
    By Remark~\ref{rmk:ring_rnder-closed}, if $\cA$ is almost derivation closed, then it is closed under taking all renormalized derivatives. In particular if $\cF$ is a $\bQ$-algebra then it is almost derivation closed if and only if it is almost $\HSder$-closed.
\end{remark}

\begin{lemma}\label{lem:almost_HS-closed_considerations}
    Let $(\bK, \Lambda, \pi, \iota)$ be an adequate ring with internal powers, $\fM$ an adequate $\Lambda$-module of monomials and $\cF$ a $(\Lambda-\Lambda)$-algebra of $\Lambda$-power series with coefficients from $\bK\Lhp \fM \Rhp$. Suppose furthermore that $\Lambda$ is cone ordered and $\fM$ is negatively cone-ordered.
    \begin{enumerate}
        \item If $\cF$ is $\{1\}$-blow-up closed, truncation-closed, closed under monomial division, and contains $\fM$ and $\Var^{(\Lambda)}$, then $\cF$ is almost $\HSder$-closed (so it is fine).
        \item If $\cF$ is fine, then it is $\HSder$-closed for the classical variables.
    \end{enumerate}
    \begin{proof}
        (1) Given $f$, we can write $f = \fm \var{x}^{\alpha} g$ for some non-singular $g$. %here we are using $\Lambda$ and $\fM$ are cone-ordered
        Now 
        \[\var{x}^n\HSder^n(f)=\fm \var{x}^n\HSder^n(\var{x}^\alpha g)=\fm\sum_k \binom{n}{k} \binom{\alpha}{n-k}\var{x}^{\alpha}\var{x}^k\HSder^{k}(g)\] 
        so it suffices to show that $\var{x}^\alpha\var{x}^k\HSder^{k}(g)\in \cF$ for all $k$. But $\var{z}^k\var{x}^k\HSder^k(g)\in \cF$ because it is a segment of $g[\var{x} \mapsto \var{x}+\var{x}\var{z}]$, so by monomial division, since $\var{x}^k\HSder^k(g)$ is non-singular, we get $\var{x}^k\HSder^k(g)\in \cF$ for all $k$.
        
        (2) Suppose that $\var{x}$ is classical in $f \in \cF(\var{x}, \var{y})$. We can write $f=\fm \var{y}^\alpha g$ for some non-singular $g$. Then note that $\HSder_{\var{x}}^n(g) \in \cF(\var{x}, \var{y})$ because $\cF$ is closed for monomial division and $\var{x}^n\HSder_{\var{x}}^n(g)\in \cF$ by almost $\HSder$-closedness of $\cF$. Thus we have $\HSder_\var{x}^n(f)=\fm \var{y}^\alpha\HSder_\var{x}^n(g)\in \cF$. %NOTE: we don't use the truncation-closedness of $\cF$
    \end{proof}
\end{lemma}

\begin{remark}
	If $\cF$ is a $(\Lambda-\Lambda)$-algebra and satisfies any of the properties of the points \ref{psFppties:tc}, \ref{psFppties:der}, and \ref{psFppties:bu}, then the smallest $(\Lambda-\Lambda)$-algebra containing $\cF$ and closed under monomial division is still a $(\Lambda-\Lambda)$-algebra satisfying the same property.
	This is clear for \ref{psFppties:tc}.

    As for \ref{psFppties:der}, for almost $\HSder$-closedness, this follows from Lemmas~\ref{lem:mon-div_closure_preserves_almostHSder-closed}, \ref{lem:md-closed_gen_subring}, and \ref{lem:leibniz_for_rn}.
    In the case $\cF$ is almost derivation-closed, Lemmas~\ref{lem:md-closed_gen_subring} and \ref{lem:leibniz_for_rn}, can be replaced by simpler similar arguments.
    
	As for \ref{psFppties:bu} note that $f[\var{x}\mapsto \var{z_0}(\var{z}_0+k)]\in \var{y}^\beta \bK \Lhp \fM \times \var{z}_0^\Lambda\var{z}_1^\Lambda\var{y}^\Lambda\Rhp^{\preceq 1}$ if and only if
    $f\in \var{y}^\beta\bK\Lhp \fM \times\var{x}^{\Lambda}\var{y}^\Lambda\Rhp^{\preceq 1}$, and that $f[\var{x}\mapsto \var{z_0}(\var{z}_0+k)]\in \var{z}_0^\alpha \bK \Lhp \fM \times \var{z}_0^\Lambda\var{z}_1^\Lambda\var{y}^\Lambda\Rhp^{\preceq 1}$ if and only if 
    $f \in \var{x}^\alpha \bK \Lhp \fM \times \var{x}^\Lambda \var{y}^\Lambda\Rhp^{\preceq 1}$.
\end{remark}

We are ready to sate the main results of this section

\begin{theorem}\label{thm:mainT0}
    Suppose that $\cF$ is a family of weakly restricted $\Lambda$-power series with coefficients from $\bK\Lhp \fM \Rhp$ which is truncation-closed and $\HSder$-closed for the classical variables.
    Then:
    \begin{enumerate}
        \item the smallest $\bZ$-algebra $\cF^*$ containing $\cF$ and closed under classical compositions is truncation closed;
        \item the smallest $\bZ$-algebra $\cF^{**}$ containing $\cF$ and closed under classical composition and implicit functions is truncation closed.
    \end{enumerate}
    Furthermore assuming $\Lambda$ and $\fM$ are cone ordered, if $\cF$ is closed for monomial division, then so are $\cF^*$ and $\cF^{**}$.
\end{theorem}

\begin{theorem}\label{thm:mainT}
    Suppose that $(\bK, \Lambda, \pi, \iota)$ is an adequate ring with powers and $\fM$ is an adequate $\Lambda$-module of monomials. Suppose also that $\fM$ and $\Lambda$ are both cone-ordered. Let $\cF$ be family of weakly restricted $\Lambda$-power series with coefficients from $\bK\Lhp \fM \Rhp$ which is truncation-closed, almost $\HSder$-closed and closed under monomial division.
    Then:
    \begin{enumerate}
        \item the smallest $\bK$-algebra containing $\cF\cup (\Var^{(\Lambda)}\fM)\cup\{(k+\var{x})^\lambda: k \in \bK^\bullet, \var{x}\in \Var, \lambda \in \Lambda\}$ and closed under compositions is truncation closed and closed for monomial division;
        \item the smallest $\bK$-algebra containing $\cF\cup (\Var^{(\Lambda)}\fM)\cup\{(k+\var{x})^\lambda: k \in \bK^\bullet, \var{x}\in \Var, \lambda \in \Lambda\}$ and closed under composition and implicit functions is truncation closed and closed for monomial division.
    \end{enumerate}
\end{theorem}

\begin{remark}[About the setting of Theorem~\ref{introthm:B}]
    The hypothesis of Theorem~\ref{introthm:B} in the introduction are a restatement of the ones of  Theorem~\ref{thm:mainT}. Indeed if \ref{hypothesis:monomials_setting} and \ref{hypothesis:coefficients_setting} hold, then $(\bK, \Lambda, \pi, \iota)$ is adequate, $\fM$ is an adequate $\Lambda$-module of monomials, and both $\Lambda$ and $\fM$ are cone-ordered. On the other hand if $\Lambda$ is a cone ordered semiring, then as said in Remark~\ref{rmk:cone-ordered_semirings}, $R^{\ge 0}\subseteq \Lambda \subseteq R$ for some ordered ring $R$, furthermore since then $R=\Lambda-\Lambda$, $\Lambda$ having binomial sequences means $R$ having those. Similarly since $\fM$ is negatively cone ordered, it embeds into the group $\fN=\fM^{-1}\fM$ with the natural order in such a way that it contains $\fN^{\le 1}$ and by Remarks~\ref{rmk:modules_of_differences} and \ref{rmk:semiring_to_ring_action}, $\fN$ is naturally an $R$-module.
\end{remark}

\begin{remark}\label{rmk:ring-simplied_hypthoesis}
    When $\Lambda$ is a ring, the hypothesis of Theorem~\ref{thm:mainT0}, simplify considerably: in fact then $\Lambda$ and $\fM$ are necessarily cone-ordered and any $\bK$-algebra containing $\fM\Var^{(\Lambda)}$ will be closed under monomial division and its closure under compositions will need to contain $\{(k+\var{x})^\lambda: k \in \bK^\bullet, \lambda \in \Lambda, \var{x} \in\Var\}$.
\end{remark}

We will prove Theorem~\ref{thm:mainT0} in the next subsection. At the end of this subsection we will give the proof of Theorem~\ref{thm:mainT} assuming Theorem~\ref{thm:mainT0}. Before that we point out some Corollaries: the first is a slight generalization of Theorem~\ref{introthm:A} (just note that if $\cF$ is a family of $\Lambda$-power series closed under truncations, then it is truncation closed qua family of weakly restricted $\Lambda$-power series with coefficients from $\bK \Lhp \fM \Rhp$); the second and third are results from the literature that are recovered from (and generalized by) Theorem~\ref{thm:mainT}.

\begin{corollary}\label{cor:mainT}
    Let $\Lambda$ and $\fM$ be as in Theorem~\ref{thm:mainT0}. Suppose that $\cF$ is a family of weakly restricted $\Lambda$-power series with coefficients from $\bK\Lhp \fM \Rhp$ which is closed under truncations and almost closed under derivatives, then for all truncation closed $X\subseteq \bK\Lhp \fM \Rhp$, the set $\langle X \cup \fM \cup \bK\rangle_\cF$ is truncation closed.
    \begin{proof}
        Apply Theorem~\ref{thm:mainT}(1) to the family $\cF\cup X$
    \end{proof}
\end{corollary}

\begin{corollary}[van den Dries, {\cite[Thm.~1.3]{dries2014truncation}}]\label{cor:dries_composition}
    If $\bK'$ is a $0$-characteristic field, $(\fM, \cdot, <)$ is totally ordered group, $X \subseteq \bK' \Lhp \fM \Rhp$ is a truncation-closed subfield, and $\cG$ is a family of power series with coefficients in $\bK'$ closed under formal derivatives, then the smallest subfield of $\bK'\Lhp \fM \Rhp$ containing $X$ and closed under series in $\cG$ is itself closed under infinite sums.
    \begin{proof}
        Apply Theorem~\ref{thm:mainT0}(1) with $\Lambda=\bN$, $\bK$ the subfield of $\bK'$ generated by the coefficients of the series in $\cG$ and $\cF = \cG \cup X$.
    \end{proof}
\end{corollary}

\begin{remark}
    The main improvement of Theorem~\ref{thm:mainT}(1) over the previous result of van den Dries in the case $\bK$ is a $0$-characteristic field and $(\fM, \cdot, <)$ is a totally ordered group is that we are allowed to take $\cG$ to be a family of power series with coefficients from the Hahn field $\bK \Lhp \fM \Rhp$ itself, as long as they converge on the infinitesimal ball (i.e.\ are weakly restricted) and the family $\cG$ is also closed under coefficient-wise truncations.
\end{remark}
    
\begin{corollary}[Fornasiero {\cite[Thm.~4.15]{fornasiero2006embedding}}\footnote{Fornasiero's statment is actually more general than the one given here as it considers Hahn field with factor sets.}, van den Dries {\cite[Thm.~1.2(1)]{dries2014truncation}}]\label{cor:dries_fornasiero_composition}
    If $\bK$ is a field, $(\fM, \cdot, <)$ is totally ordered group and $X \subseteq \bK \Lhp \fM \Rhp$ is a truncation-closed subfield, then its relative Henselianization is closed under truncations.
    \begin{proof}
        Apply Theorem~\ref{thm:mainT0}(2) or Theorem~\ref{thm:mainT}(2) with $\Lambda=\bN$ and $\cF$ consisting of all the polynomials with coefficients in $X$, and note that the constants of the smallest $\bZ$-algebra containing $\cF$ and closed under composition and implicit functions is the relative Henselianization of $X$ in $\bK \Lhp \fM \Rhp$.
    \end{proof}
\end{corollary}

As explained in \cite[p.121]{fornasiero2006embedding} and \cite{dries2014truncation}, this result is in turn a generalization of previous result (namely that if $\bK$ is a real closed field and $\fM$ is a divisibile abelian group, then the real algebraic closure of a truncation closed subfield of $\bK\Lhp \fM \Rhp$ is truncation closed) first explicitly appearing in \cite[Lem.~3.5]{mourgues1993every} where it is credited to Delon (see subsequent remark in the same reference and \cite{delon1991indecidibilite}, cf also \cite{mourgues1993transfinite} for another proof).

We now turn to the deduction of Theorem~\ref{thm:mainT} from Theorem~\ref{thm:mainT0}. This will essentially amount to putting together a couple observations and a Lemma.

\begin{remark}\label{rmk:blowpups_pcomp}
    Note that the condition that $\cF$ is closed under $S$-blow-ups can be restated by saying that $\cF$ is closed under right-composition with the set of powerable infinitesimal polynomials with coefficients from the monoid $S\subseteq \bK^\bullet$, denoted by $S[\var{z}]$.
    This is because
    \[S[\var{z}]=\{\sigma p(\var{y}): \var{y} \subseteq \Var,\; \sigma : \var{y} \to \var{z},\; p \in P(\var{y})\}\]
    where $P(\var{y})$ is the set of compositions of polynomials of the form $\var{y}_0(k + \var{y}_1)$ with $k \in S \cup \{0\}$ and $\var{y}_0, \var{y}_1\in \var{y}$.
    In particular, given any $\bZ$-algebra $\cF$, the smallest $\bZ$-algebra containing $\cF$ and closed under blow-ups is given by 
	\[\cF_{b,S}(\var{z}) := \bZ[f(p): p \in S[\var{z}]^{\var{x}},\; f(\var{x}) \in \cF(\var{x})].\]
    The letter $b$ stands for ``blow-up''.
\end{remark}

\begin{remark}
    If $\cF$ is closed under monomial division and truncations, then whenever $\fm\var{x}^\alpha(k+\varepsilon) \in \cF$ for some $k \in \bK^\bullet$, $\fm \in \fM$, $\alpha \in \Lambda^{\var{x}}$ and $\varepsilon \in \bK\Lhp \fM \times \var{x}^\Lambda\Rhp^{\prec 1}$, we have $\varepsilon \in \cF$, $k \in \cF$, and $\fm\var{x}^\alpha \in \cF$.
\end{remark}

\begin{lemma}\label{lem:almostfine_bu}
	Suppose $\cF$ is an almost fine $\bZ$-algebra and $M \subseteq \bK^\bullet$ is a submonoid. Then $\cF_b\coloneqq \cF_{b,M}$ is a truncation-closed $\bZ$-algebra. If furthermore $\fM$ and $\Lambda$ are cone-ordered, and $\cF$ is a fine $\bZ$-algebra containing $(k+\var{x})^\lambda$ for all $\lambda \in \Lambda$ and $k \in M$, then $\cF_{b, M}$ is fine.
	\begin{proof}
		We first show that $\cF_b$ is truncation-closed under the assumption that $\cF$ is truncation-closed and almost $\HSder$-closed. For this, by Lemma~\ref{lem:tc_generates_tc}, it suffices to show that for each $p(\var{z}) \in M [\var{z}]^\var{x}$, $\cF(p(\var{z})):=\{f (p(\var{z})): f \in \cF(\var{x})\}$ is a truncation-closed $\bZ$-algebra.% because then the algebra closure of the family $\bigcup_{\var{z},p} \cF(p(\var{z}))$ will be truncation-closed by Lemma~\ref{lem:tc_generates_tc}.
            
        Since each component of such $p$ is a composition of of polynomials of the form $\var{z}_0(k+\var{z}_1)$ with $k \in M \cup \{0\}$, again invoking Lemma~\ref{lem:tc_generates_tc} we can restrict to the case in which $p$ is a tuple of polynomials of the form $\var{z}_0(k+\var{z}_1)$ with $k \in \bK^{\bullet} \cup \{0\}$, where $\var{z}_0$ and $\var{z}_1$ are distinct variables not appearing in $\var{x}$.%NOTE the lemma is invoked to reduce to the case of distinct variables
            
        %Thus the statement is reduced to proving that if $\cF(\var{x}, \var{y})$ is a subalgebra of $\bK \Lhp (\var{x}, \var{y})^\Lambda\Rhp$ closed under truncations (with $\var{x} \in \Var$ and $\var{y} \subseteq \Var$), then so is $\cF(\var{z}_0(\var{z}_1+k), \var{y}):=\{f(\var{z}_0(\var{z}_1+k), \var{y}): f \in \cF(\var{x}, \var{y})\}$ for $k \ge 0$. %Also notice that we can check it only for $\var{z}_0$ and $\var{z}_1$ distinct variables.

        To this end note that by Remark~\ref{rmk:partial_truncations}, it suffices to show that for $\var{t}\in \Var$, $\var{y} \subseteq \Var$, and $\var{t}\notin \var{y}$, if $f(\var{t}, \var{y}) \in \cF(\var{t}, \var{y})$, then $f(\var{z_0}(\var{z}_1+k), \var{y})|S \in \cF(\var{z}_0(\var{z}_1+k), \var{y})$ for $S$ of one of the forms
        \[S_0^{\alpha}=\fM  \times \var{z}_1^\Lambda \var{y}^\Lambda\{\var{z}_0^{\alpha'} :\alpha' < \alpha\},\quad
        S_1^{\beta}=\fM \times \var{z}_0^\Lambda\var{y}^\Lambda \{\var{z_1}^{\beta'}: \beta' < \beta\},\]
        \[\text{and} \quad 
        S_2=\fM\times \var{z}_0^\alpha \var{z_1}^\Lambda\{\var{y}^\gamma:\var{y}^\gamma \in S'\} \quad \text{where $S'$ is a segment of $\var{y}^\Lambda$.}\]
            
        If $S$ is of the last form $S_2$, there is nothing to prove as $f(\var{z}_0(\var{z}_1+k),\var{y})|S_2=h(\var{z}_0(\var{z}_1+k), \var{y})$ where $h(\var{t},\var{y})=f(\var{t}, \var{y})|(\fM \times \var{t}^\Lambda S')$.

        For $S$ of the first two forms we need to distinguish the two cases $k=0$ and $k\in S\setminus \{0\}$. If $k=0$ the statement is trivial as then in every monomial of $f(\var{z}_0\var{z}_1, \var{y})$, $\var{z}_0$ and $\var{z}_1$ appear with the same exponent hence,  \[f(\var{z}_0\var{z}_1, \var{y})|S_0^{\alpha}=f(\var{z}_0\var{z}_1, \var{y})|S_1^{\alpha}=h(\var{z}_0\var{z}_1, \var{y})\]
        where $h(\var{t}, \var{y}):=f(\var{t}, \var{y})|(\fM \times \var{y}^\Lambda\{\var{t}^{\alpha'}:\alpha'<\alpha\})$.
        Hence we can assume $k\in M\setminus \{0\}$. In such a case $\beta$ ranges in $\bN$ because $\var{z}_1$ is classical in $f(\var{z}_0(\var{z}_1+k), \var{y})$. We can then write
        \[f(\var{z}_0(\var{z}_1+k), \var{y})|S_1^n = \sum_{m<n} \var{z}_1^m h_m(k\var{z}_0, \var{y}), \quad \text{where}\quad  h_m(\var{t}, \var{y}) = \var{t}^m\HSder^m f(\var{t}, \var{y})\]
		Finally
		\[f(\var{z}_0(\var{z}_1+k), \var{y})|S_0^\alpha = h(\var{z}_0(\var{z}_1+k), \var{y}), \; \text{where} \;  h(\var{t}, \var{y})=f(\var{t}, \var{y})| (\fM \times \var{y}^\Lambda\{\var{t}^{\alpha'}:\alpha'<\alpha\}),\]
        which concludes the proof of truncation-closedness of $\cF_{b}$.

        In order to show that if $\cF$ is fine and contains $(k+\var{t})^\lambda$ for all $k \in M$, $\lambda \in \Lambda$, then $\cF_{b}$ is still closed under monomial division, by Lemma~\ref{lem:md-closed_gen_subring} it suffices to show that for $\var{x}=\{\var{t}\}\cup \var{y}$ for some $\var{t} \in \Var\setminus \var{y}$, $\var{z}=\{\var{t}, \var{v}\}\cup \var{y}$ with $\var{t}, \var{v} \in \Var\setminus \var{y}$ distinct, and $p\in M[\var{z}]^\var{x}$ of the form $[\var{t} \mapsto \var{t}(k+\var{v}), \var{y}\mapsto \var{y}]$ with $k \in M\cup \{0\}$, we have that the ring generated by $\cF(p(\var{z}))\cup \cF(\var{x})$ is closed for monomial division.
        
        %we only need to observe that each $\cF(p(\var{z}))$ is closed under monomial division and again we can reduce to consider $\var{x}$ of the form $\{\var{t}\}\cup \var{y}$ for some $\var{t} \in \Var\setminus \var{y}$ and $p$ of the form $[\var{t} \mapsto \var{v}(k+\var{w}), \var{y}\mapsto \var{y}]$ with $\var{v}, \var{w} \in \Var\setminus \var{y}$ distinct and $k \in M\cup \{0\}$.

        For this observe that for $f \in \bK \Lhp \fM \times \var{x}^\Lambda\Rhp=\bK\Lhp \fM \times \var{t}^\Lambda \var{y}^\Lambda\Rhp$ and $\lambda \in \Lambda$, we have that:
        \begin{enumerate}
            \item $f[\var{t}\mapsto \var{t}(k+\var{v})] \in \var{t}^{\lambda}\bK \Lhp \fM \times \var{t}^\Lambda\var{v}^\Lambda\var{y}^\Lambda\Rhp^{\preceq 1}$ if and only if $f \in \var{t}^\lambda\bK \Lhp \fM \times \var{t}^\Lambda \var{y}^\Lambda\Rhp^{\preceq 1}$, so if $f\in \cF$ and $f[\var{t}\mapsto \var{t}(k+\var{v})] \in \var{t}^{\lambda}\bK \Lhp \fM \times \var{t}^\Lambda\var{v}^\Lambda\var{y}^\Lambda\Rhp^{\preceq 1}$, we can find $g\in \cF$ be such that $f=t^{\lambda} g$, and $f[\var{t}\mapsto \var{t}(k+\var{v})]= \var{t}^\lambda g[\var{t}\mapsto \var{t}(k+\var{v})] \cdot (k+\var{v})^\lambda$ where $(k+\var{v})^\lambda \in \cF(\var{x})$ and $g[\var{t}\mapsto \var{t}(k+\var{v})]\in \cF(p(\var{z}))$;
            \item if $f[\var{t}\mapsto \var{t}(k+\var{v})] \in \var{v}^{\lambda}\bK \Lhp \fM \times \var{t}^\Lambda\var{v}^\Lambda\var{y}^\Lambda\Rhp^{\preceq 1}$ then either $k=0$ and $f \in \var{t}^\lambda\bK \Lhp \fM \times \var{t}^\Lambda\var{y}^\Lambda\Rhp^{\preceq 1}$ or $f\in \bK \Lhp \fM \times \var{y}^\Lambda \Rhp$ and $\lambda=0$; it follows that if $f \in \cF$ and $f[\var{t}\mapsto \var{t}(k+\var{v})]= \var{v}^\lambda h$ with $h \in \bK \Lhp \fM \times \var{t}^\Lambda\var{v}^\Lambda\var{y}^\Lambda\Rhp^{\preceq 1}$, then $h \in \var{t}^\Lambda \cF(p(\var{z}))$. 
        \end{enumerate}
        Thus each element of the closure under monomial division of $\cF(p(\var{z}))$ is in the algebra generated by $\cF(p(\var{z}))\cup \cF(\var{x})$, which is therefore closed under monomial division by Lemma~\ref{lem:md-closed_gen_subring}. This concludes the proof of the claim that $\cF_{b}$ is closed for monomial division.

        %Previous argument
        % As for showing that if $\cF$ is fine then $\cF_{b}$ is still closed under monomial division, note that by Lemma~\ref{lem:md-closed_gen_subring}, we only need to observe that each $\cF(p(\var{z}))$ is closed under monomial division and again %we can assume that $p\in M[\var{z}]^\var{x}$ is an assignment of the form 
        % we can assume that $p$ is of form 
        % \[p=[\var{x}_i \mapsto \var{t}_i(k_i+\var{y}_i):i<m]\]
        % with $k_i \in (M \cup \{0\})$. This is because given $f \in \bK \Lhp \fM \times \var{x}^\Lambda\Rhp$ and $\lambda \in \Lambda^{\var{x}}$, we have that $f[\var{x}\mapsto \var{y}(k+\var{t})] \in \var{y}^{\lambda}\bK \Lhp \fM \times \var{y}^\Lambda\var{t}^\Lambda\Rhp^{\preceq 1}$ if and only if $f \in \var{x}^\lambda\bK \Lhp \fM \times \var{x}^\Lambda\Rhp^{\preceq 1}$, and if $f[\var{x}\mapsto \var{y}(k+\var{t})] \in \var{t}^{\lambda}\bK \Lhp \fM \times \var{y}^\Lambda\var{t}^\Lambda\Rhp^{\preceq 1}$ then either $k=0$ and $f \in \var{x}^\lambda\bK \Lhp \fM \times \var{x}^\Lambda\Rhp^{\preceq 1}$ or $f\in \bK \Lhp \fM \Rhp$ and $\lambda=0$.

        Finally we see that if an algebra is closed for monomial division, truncations and blow-ups then it is in fact almost $\HSder$-closed by Lemma~\ref{lem:almost_HS-closed_considerations}(1).
	\end{proof}
\end{lemma}

For $\bK^\bullet$-blow-up-closed fine algebras, closing under compositions is the same as closing under classical compositions.

\begin{lemma}\label{lemma:comp_is_ccomp_for_fine_algebras}
    Let $\cF$ be a $\bK^\bullet$-blow-up closed fine algebra of restricted $\Lambda$-power series with coefficients from $\bK \Lhp \fM \Rhp$. Then its closure under classical compositions is closed under compositions.
    \begin{proof}
        Note that the closure of $\cF$ under classical compositions must be still closed under blow ups, so the statement boils down to the observation that if $\cF$ is closed under classical compositions, blow-ups, and monomial division, then it is closed under compositions. But this follows from Remark~\ref{rmk:blow-ups}.
    \end{proof}
\end{lemma}

\begin{proof}[Proof of Theorem~\ref{thm:mainT} from Theorem~\ref{thm:mainT0}]
    Up to passing to the smallest $\bK$-algebra containing $\cF\cup (\Var^{(\Lambda)} \times \fM) \cup\{(k+\var{x})^\lambda, k \in \bK^\bullet, \lambda \in \Lambda, \var{x} \in \Var\}$, we can assume that $\cF$ is a \emph{fine} $\bK$-algebra containing $(k+\var{x})^\lambda$ for all $k \in \bK^\bullet$, $\lambda \in \Lambda$, and $\var{x} \in \Var$.
    By Lemma~\ref{lem:almostfine_bu}, we have that $\cF_{b,\bK^\bullet}$ is still fine. Furthermore the smallest $\bK$-algebra containing $\cF$ and closed under compositions must clearly be closed under $\bK^\bullet$-blow-ups so it must contain $\cF_{b, \bK^\bullet}$.
    Thus we can conclude by Lemma~\ref{lem:almost_HS-closed_considerations}(3) and Theorem~\ref{thm:mainT0} applied to $\cF_{b, \bK^\bullet}$ and by Lemma~\ref{lemma:comp_is_ccomp_for_fine_algebras}.
\end{proof}

\subsection{Weakly restricted power series and main proofs}\label{ssec:weakly_restricted_power_series}
We fix throughout this section a commutative ring $\bK$ and an ordered commutative cancellative monoid $\fM$.
Theorem~\ref{thm:mainT0} will be deduced from its special case when $\Lambda=\bN$ with the usual order, in which case we call elements of $\bK\Lhp \fM \times \var{x}^\bN\Rhp$ \emph{weakly restricted power series} (thus dropping the ``$\bN$''): in fact as it will appear Theorem~\ref{thm:mainT0} essentially amounts to a packaging of its specialization to $\Lambda=\bN$ together with Lemma~\ref{lem:tc_generates_tc} and some syntactical manipulation.

\begin{notation}
    Recall Remark~\ref{def:weakly_restricted} implying that elements of $\bK\Lhp \fM \times \var{x}^\bN\Rhp$ can be regarded as power series
    \[f(\var{x})=\sum_{n \in \bN^{\var{x}}} f_n \var{x}^n\in \bK \Lhp\fM \Rhp [ \![\var{x}]\!]\]
    where $\Supp_\fM(f)=\bigcup_{n \in \bN^{\var{x}}}f_n$ is a Noetherian subset of $\fM$. We are going to use a special notation for segments of $f$ of the form $f|(S \times \var{x}^\bN)$ for some segment $S\subseteq \fM$, setting
    \[f\|S\coloneqq f|(S \times \var{x}^\bN)=\sum_{n \in \bN^\var{x}} (c_n|S) \var{x}^n.\]
    We will say that a subset $X$ of $\bK\Lhp \fM \times \var{x}^\bN\Rhp$ is \emph{closed under $\|$-truncations} if whenever $f \in X$, $f\|S \in X$ for all segments $S \subseteq \fM$.
\end{notation}

\begin{remark}\label{rmk:coefficient-wise_tc_is_tc}
    Note that if $\cF$ is a $\bZ$-algebra of weakly restricted power series with coefficients from $\bK \Lhp \fM \Rhp$ which is $\HSder$-closed, then it is necessarily closed under taking segments of the form $f\|(\fM \times S)$ where $S \subseteq \var{x}^\bN$ for $f \in \cF(\var{x})$. 
    To see it note that in fact by Remark~\ref{rmk:partial_truncations} we can reduce to the case in which $\var{x}$ is a single variable and $S=\{1, \var{x}, \ldots, \var{x}^{N-1}\}$ for some $N\in \bN$. Now if we write $f=\sum_{n\in \bN} \var{x}^n f_n$, then $f|(\fM \times \var{x}^S)=\sum_{n<N}f_n\var{x}^n$ so the claim follows form the fact that $\var{x}^nf_n= \var{x}^n\big((\HSder^n f)[\var{x}\mapsto 0]\big)=\var{x}^n\big((\HSder^n f)\|(\fM \times \{1\})\big)$.
    It follows that an $\HSder$-closed $\bZ$-algebra is truncation closed if and only if it is closed under $\|$-truncations.
\end{remark}

\begin{lemma}\label{lem:tc_product}
	If $\cL(\var{x}) \subseteq \bK \Lhp \fM \times \var{x}^{\bN} \Rhp$ is a subring (resp.\ subgroup, non-unital subring, closed under derivatives), then
	\[\cL^{\|}(\var{x})\coloneqq\{f \in \cL(\var{x}): \forall U \text{ final segment of } \fM, \; f\|U \in \cL(\var{x})\}\]
	is a subring (resp.\ subgroup, non-unital subring, closed under derivatives). 
	\begin{proof}
        If $\cL(\var{x})$ is closed under sums (or $R$-linear combinations for any $R \subseteq \bK$), then so is $\cL^{\|}(\var{x})$: this follows from the fact that $f \mapsto f\|U$ is $\bK$-linear for every $U \subseteq \fM$. Similarly for the closure under derivatives, as $-\|U$ commutes with formal derivatives.
                
		Thus to prove the Lemma it suffices to show that if $f\coloneqq\sum_m f_m \var{x}^m$, $g\coloneqq \sum_{l} g_l \var{x}^l$ are in $\cL^{\|}(\var{x})$, then $f \cdot g \in \cL^{\|}(\var{x})$. Let $S=\Supp_\fM(f)$ and $T=\Supp_\fM (g)$. These are both Noetherian, so for fixed $U$, by Lemma~\ref{lem:product_segmentation}, 
        there are $a \in \bN$, final segments $S_0, \ldots, S_{a-1}$ in $\fM$ and pairwise disjoint segments $T_0, \ldots,  T_{a-1}$ of $\fM$ such that for each $m$ and $l$
		\[(f_m \cdot g_l)|U = \sum_{j<a} (f_m|S_j) \cdot (g_l|T_j).\] 
		But then it suffices to observe that
		\[\begin{aligned}
			(f \cdot g) \| U =\sum_{p} \left(\left(\sum_{m+l=p} f_m g_l\right)\big\|U\right)\var{x}^p=\\ =\sum_{j<a} \sum_p \left(\sum_{m+l=p} (f_m|S_j) \cdot (g_l|T_j)\right)\var{x}^p=\sum_{j<a} (f\|S_j) \big( g\| T_j\big) 
		\end{aligned}\]
		which concludes the proof.
	\end{proof}
\end{lemma}

\begin{notation}
    Given $\cA(\var{x}) \subseteq \bK \Lhp \fM \times \var{x}^\bN\Rhp$, $\cB(\var{y})\subseteq \bK \Lhp \fM \times \var{y}^\bN\Rhp$, we will write $\cA(\cB(\var{x}))$ for the set of compositions of the restricted series in $\cA$ with the (assignments of) infinitesimal restricted series in $\cB(\var{y})$, i.e.
    \[\cA(\cB(\var{x}))\coloneqq \big\{f(g): f \in \cA(\var{x}),\; g \in  (\cB(\var{y})^{\prec 1})^{\var{x}}\big\}.\]
\end{notation}

\begin{lemma}\label{lem:tc_composition}
	Assume $\cA(\var{x})\subseteq \bK \Lhp \fM \times \var{x}^\bN\Rhp$ is a $\bZ$-subalgebra containing $\var{x}$ closed under truncations and $\HSder$-closed, and that $\cB(\var{y})\subseteq \bK \Lhp \fM \times \var{y}^\bN\Rhp$ is closed under $\|$-truncations. Then $\cA(\cB(\var{y}))$ is closed under truncations.
    \begin{proof}
        Let $\var{x}=\{\var{x}_i:i<n\}$ have $n$ distinct variables. Let $f:=\sum_{m\in \bN^\var{x}}f_m\var{x}^m \in \cA(\var{x})$, $g_i\in \cB(\var{y})^{\prec 1}$ for $i<n$, and $g \in (\cB(\var{y})^{\prec 1})^{\var{x}}$ be the assignment $g:=[\var{x}_i \mapsto g_i]_{i<n}$. We will use the notation $g\|R$ where $R$ is a segment of $\fM$ to denote the assignment given by $(g\|R) =[\var{x}_i \mapsto g_i\|R]$.

        Let $M \subseteq \fM$ be a final segment, $\bar{S}\coloneqq \Supp_\fM(f)$ and $\bar{T}\coloneqq \bigcup_{i<n}\Supp_{\fM}(g_i)$. Denote by $\pi: \bar{S} \times \bar{T}^* \to \fM$ the natural map given by the product in $\fM$. Let $\{L, U\}$ be the segmentation of $\bar{S} \times \bar{T}^*$ given by
        \[U\coloneqq \pi^{-1} (M)=\{(\mathfrak{s}, \mathfrak{t}_0, \ldots, \mathfrak{t}_{l-1}): l \in \bN,\, \mathfrak{s} \in S,\, (\mathfrak{t}_i)_{i<l} \in T^l,\, \mathfrak{s} \mathfrak{t}_0 \cdots \mathfrak{t}_{l-1}\in M\}.\]
        By Corollary~\ref{cor:word_segmentation}, we can find finite segmentations $\cS$ of $\bar{S}$ and $\cT$ of $\bar{T}$ such that $\cS \otimes \cT^*$ refines $\{L, U\}$, whence for all $l\in \bN$ and $S \in \cS$, $T_0, \ldots, T_{l-1} \in \cT$, $\big((c_l|S) \cdot \prod_{i<l} (g_i|T_i)\big)\|U$ is either $0$ or itself. Now consider for each $S\in \cS$, the sets $\cT_{0,S}\coloneqq\{T\in \cT: \forall l\in \bN, \pi(S \times T^l)\subseteq U\}$ and $\cT_{1,S}=\cT \setminus \cT_{0,S}$. Note that $T_{0,S}\coloneqq \bigcup \cT_{0,S}$ is a final segment and $T_{1,S}=T \setminus T_{0,S}$ is an initial segment. Now $g=g\|{T_{0,S}}+g\|{T_{1,S}}$ and %need to say that I am doing truncations of assignements
        \[(f\|S)(g)=(\tilde{f}\|S)(g\|_{T_{0,S}}, g\|_{T_{1,S}})=\sum_{h \in \bN^{\var{x}}} (g\|{T_{1,S}})^h \cdot \big(\HSder^h (f\|S)\big)(g\|{T_{0,S}}).\] 
        Note that $\Supp_{\fM}((g\|T_{1,S})^h\big)\subseteq \pi(T_{1,S}^{|h|})$ where $|h|=\sum_{i<n} h(\var{x}_i)$, and that for large enough $|h|$, we must have $\pi(S\times T_{1,S}^{|h|})\subseteq L$, because, by construction, for all $T \in \cT_{1,S}$, there is some $m_T$ such that $\pi(S\times T^m)\not\subseteq U$ for all $m\ge m_T$ and hence $\pi(S\times T^m)\subseteq L$, by the property of the segmentations $\cS$ and $\cT$. Thus it suffices to have $|h|\ge m_{S}\coloneqq \max\{m_T: T \in \cT_{1,S}\}$.
        But then for all $h$ with $|h| \ge m_S$ we get $\Supp_{\fM} \big((g\|{T_{1,S}})^h (\partial_h f\|S)(g\|{T_{0,S}} \big)\subseteq L$ because it must be contained in the image by $\pi$ of a union of sets $X$ from the segmentation $\cT_{1,S}^{\otimes |h|} \otimes S \otimes \cT_{0,S}^*$, but each $X$ in such a family is either such that $X\subseteq L$ or $X\subseteq U$ and since by hypothesis $\pi (T_{1,S}^{|h|} \times S) \subseteq L$ and $T\le 1$, it must be $X\subseteq L$. It follows that
        \[f(g)\|M=\sum_{S \in \cS} (f\|S)(g)=\sum_{S \in \cS} \sum_{\substack{h\in \bN^{\var{x}}\\|h|<m_S}} (g\|{T_{1,S}})^h \cdot \big(\HSder^h (f\|S)\big)(g\|{T_{0,S}}),\]
        so $f(g)\|M \in \cA(\cB(\var{y}))$ as for all $i<n$ and $S \in \cS$, $g_i\|T_{0,S} \in \cB(\var{y})^{\prec 1}$ and $g_i\|T_{1,S} \in \cA(\cB(\var{x}))$, for all $h \in \bN^{\var{x}}$, $\HSder^h(f\|S) \in \cA(\var{x})$.
    \end{proof}
\end{lemma}

\begin{lemma}\label{lem:md_composition}
    Let $\var{y}$ be a finite set of variables and let $\cB(\var{y})\subseteq \bK \Lhp \fM \times \var{y}^\bN\Rhp^{\prec 1}$ contain $\var{y}$ and be closed under segments.
    If $\cL$ is a $\bZ$-algebra of weakly restricted power series containing $\fM$ and $\cB(\var{y})$, which is closed under composition on the right with series of the form $\var{x}+\var{y}$, and closed under truncations and monomial division, then
    \[\cL(\cB(\var{y})) = \{ f (g): \var{x} \subseteq \Var, f \in \cL(\var{x}), g \in (\cB(\var{y})^{\prec 1})^{\var{x}} \}\]
    is a subring of $\bK \Lhp \fM \times \var{y}^\bN\Rhp$ closed under truncations and monomial division.
    \begin{proof}
        That $\cL (\cB(\var{y}))$ is closed under under truncations follows from Lemma~\ref{lem:tc_composition}. Let us address monomial division. To avoid proliferation of letters, whenever $f$ is a series and $\Supp(f)\le \var{x}^m\fm$ we write $f/\var{x}^m\fm$ for the only non-singular $\tilde{f}$ such that $f=\var{x}^m\fm\tilde{f}$ (it exists because of the hypothesis that $\fM$ is cone ordered).
        
        Let $f \in \cL(\var{x})$ and $g \in \cB(\var{y})^{\var{x}}$ be such that $f(g) \in \fm\bK\Lhp \fM \times \var{y}^\bN\Rhp^{\preceq 1}$ for some $\fm \in \fM$.
        Now set $M=\fM^{>\fm}$ and argue as in Lemma~\ref{lem:tc_composition} and with the same notation get $\cS$ and $\cT$. Without loss of generality up to further refining $\cS$ and $\cT$ we can assume that each $S \in \cS$ and $T \in \cT$ has a maximum (which we denote by $\fp_S$ and $\fn_T$ respectively).

        Let $\var{z}\coloneqq\{\var{z}_{i, T}: i<n, T \in \cT\}\subseteq \Var$ be a set of $|\cT|n$ variables, and write \[h\coloneqq f[\var{x}_i \mapsto \sum_{T \in \cT}\var{z}_{i,T}: i<n]=\sum_{m \in\bN^{\var{z}}}h_m \var{z}^m \quad \text{with}\quad h_m \in \bK \Lhp \fM \Rhp.\]
        Let also $G\in \cB(\var{y})^{\var{z}}$, $\tilde{G} \in \cL(\var{y})^{\var{z}}$ and $\fn \in \fM^{\var{z}}$ be the assignments
        \[\begin{aligned}
            G\coloneqq[\var{z}_{i,T}\mapsto (g_i\|T): i<n, T \in \cT],\\
            \tilde{G}\coloneqq [\var{z}_{i,T}\mapsto (g_i\|T)/\fm_T: i<n, T \in \cT],\\
            \fn \coloneqq [\var{z}_{i,T} \mapsto \fn_T: i < n, T \in \cT]
        \end{aligned}\]
        For each $S \in \cS$, let $N_S\subseteq \bN^{\var{z}}=\{m \in \bN^{\var{z}}: \Supp_{\fM}(h_mG^m)\subseteq \fM^{\le \fm}\}$ and note that $N_S$ is a final segment of $\bN^{\var{z}}$.
        We can partition each $N_S$ into final finitely many segments $\{N_{S,j}: j<l_S\}$ each with a minimum $m_{S,j}$.
        Finally set $h_{S,j}\coloneqq h |(S \times \var{z}^{N_{S,j}})$ and note that $\tilde{h}_{S,j}\coloneqq h_{S,j}/(\fp_S\var{z}^{m_{S,j}}) \in \cL$ for all $S\in \cS$ and $j<l_S$, and that
        \[f(g)=f(g)\|\fM^{\le \fm}=\sum_{\substack{S \in \cS\\j<l_S}} h_{S,j}(G)%=\sum_{\substack{S \in \cS\\j}} G^{m_{S,j}}\fp_S\tilde{h}_{S,j}(G)
        =\sum_{\substack{S \in \cS\\j<l_S}}\fp_S\fn^{m_{S,j}} \cdot \tilde{G}^{m_{S,j}}\cdot \tilde{h}_{S,j}(G).\]
        By construction each $\Supp_{\fM}(\tilde{G}^{m_{S,j}}\cdot \tilde{h}_{S,j}(G)) \le 1$ and each $\fp_S\fn^{m_{S,j}}\le \fm$, thus we can conclude that $f(g)/\fm \in \cL (\cB(\var{y}))$. 

        As for showing that if $f \in \cL(\var{x})$ and $g \in \cB(\var{y})^{\var{x}}$ are such that $f(g) \in \var{y}^m\bK\Lhp \fM \times \var{y}^\bN\Rhp^{\preceq 1}$ then $f(g)/\fm \in \cL(\cB(\var{y}))$, we can restrict to the case in which $\var{y}^m=\var{y}_0$ for some single variable $\var{y}_0 \in \var{y}$.
        Also, up to replacing $f$ with a composition of $f$ with some sum of variables, can assume that all $g_i$s are either such that $g_i=\var{y}_0\tilde{g}_i$ for some non-singular $\tilde{h}_i$ or that $g_i \in \bK \Lhp \fM \times (\var{y} \setminus \{\var{y}_0\})^\bN\Rhp$.
        Thus we only have to deal with the situation where $\var{x}=\var{z} \cup \var{t}$ for some disjoint sets of variable $\var{z}$, $\var{t}$, and $f \in \cL(\var{z},\var{t})$ is evaluated at an assignment $g=(c, h)$ where $c \in \bK \Lhp \fM \times (\var{y} \setminus \{\var{y}_0\})^\bN\Rhp^{\var{z}}$ and $h\in \big(\cB(\var{y}) \cap (\var{y}_0\bK \Lhp \fM \times \var{y}^\bN\Rhp^{\preceq 1})\big)^{\var{t}}$. It follows from the hypothesis that then $f=\sum_{j \in \bN^{\var{t}}} F_j\var{t}^{j}$ with $F_j \in \bK \Lhp \fM \times \var{z}^{\bN}\Rhp$ and $F_0(c)=0$, so $f(c,h)=\bar{f}(c,h)$ where $\bar{f}=f-F_0 \in \cL$ can be written as sum $\bar{f}=\sum_{i<l} \var{t}^{m_i}f_i$ for some $l\in \bN$, $m_i\in\bN^{\var{t}}\setminus \{0\}$ and some non-singular $f_i \in \cL$. Finally note 
        \[\bar{f}(c,h)/\var{y}_0=\sum_{i<l} \var{y}_0^{|m_i|-1} \tilde{h}_i^{m_i}\cdot f_i(c,h) \in \cL(\cB(\var{y})),\]
        where $\tilde{h}_i=h_i/\var{y}_0 \in \cL(\var{y})^{\var{t}}$ and $|m_i|$ is the sum of the components of $m_i \in \bN^{\var{t}}$.
    \end{proof}
\end{lemma}

\begin{definition}\label{def:inverses}
    Let $\var{x}$ be a single variable. Given $g\in\bK \Lhp \fM \times \var{x}^\bN\Rhp$ and $f \in \bK \Lhp \fM \times \var{x}^\bN\Rhp^{\prec 1}$, we will write $g \circ f$ for the series $g(f)=g[\var{x}\mapsto f]$. We write $\bK \Lhp \fM \times \var{x}^\bN\Rhp^\circ$ for the set
    \[\bK \Lhp \fM \times \var{x}^\bN\Rhp^\circ=\bigg\{f \in \bK\Lhp \fM \times \var{x}^\bN\Rhp^{\prec 1}: (\partial_{\var{x}}f)[\var{x}\mapsto 0]\in \big(\bK\Lhp \fM\Rhp^{\preceq 1}\big)^{*}\bigg\},\]
    and set for $\cF(\var{x})\subseteq \bK\Lhp \fM \times \var{x}^\bN\Rhp$, $\cF(\var{x})^\circ \coloneqq \cF(\var{x})\cap \bK \Lhp \fM \times \var{x}^\bN\Rhp^\circ$.
    For $f \in \bK\Lhp \fM \times \var{x}^\bN\Rhp^\circ$ we write its \emph{compositional inverse} as $f^{\circ(-1)}$, this is the unique series $f \in \bK \Lhp \fM \times \var{x}^\bN\Rhp$ such that $f\circ f^{\circ(-1)}=f^{\circ(-1)} \circ f=\var{x}$.
    %The set of series with a compositional inverse will be denoted by $\bK \Lhp \fM \times \var{x}\Rhp^{\circ}$. given by 
    %\[\bK \Lhp \fM \times \var{x}^\bN\Rhp^\circ=\bigg\{f \in \bK\Lhp \fM \times \var{x}^\bN\Rhp^{\prec 1}: (\partial_{\var{x}}f)[\var{x}\mapsto 0]\in \big(\bK\Lhp \fM\Rhp^{\preceq 1}\big)^{*}\bigg\},\]
    %and we will write $\cF(\var{x})^\circ$ for $\cF(\var{x})\cap \bK \Lhp \fM \times \var{x}^\bN\Rhp^\circ$. %TODO: write the explicit definition
\end{definition}

\begin{notation}
    We use the classical notation $[\var{x}^n]f$ for the coefficient of $\var{x}^n$ of a power series $R[\![\var{x}]\!]$ with coefficients in a ring $R$.
\end{notation}

\begin{remark}\label{rmk:lagrange}
    When $\fM=1$, the set $\bK\Lhp \var{x}^\bN\Rhp^{\circ}$ is the usual set of power series with coefficients from $\bK$ that have a formal inverse, i.e.\ those $f=\sum_{n\in \bN}f_n \var{x}^n \in \bK\Lhp \var{x}^\bN\Rhp =\bK [\![\var{x}]\!]$ where $f_0=0$ and $f_1 \in \bK^*$, whence $f$ can be written as $\var{x}/g$ where $g \in \bK [\![\var{x}]\!]^*$. Recall the following forms of Lagrange inversion formula relating the coefficients of $f^{\circ(-1)}$ and the coefficients of $g$ (see \ \cite[Thm.~2.1.1(1\-2)]{gessel2016lagrange})
    \begin{equation}\label{eq:lagrange_inversion}
    \begin{aligned}
    n[\var{x}^n](f^{\circ(-1)})^m = m[\var{x}^{n-m}]g^n \quad \text{for all}\quad n,m\in \bN;\\
    [\var{x}^n](f^{\circ(-1)})^m=m[\var{x}^{n-m}](g^{n+1}\partial f) \quad \text{for all}\quad n,m\in \bN.\end{aligned}
    \end{equation}
\end{remark}

\begin{remark}\label{rmk:inverses1}
    The set $\bK\Lhp \fM \times \var{x}^\bN\Rhp^\circ$ consists exactly of the restricted power series that admit a compositional inverse in $\bK \Lhp \fM \times \var{x}^\bN\Rhp$. 
    In fact, if $f=\sum_{n\in \bN^{>0}} f_n \var{x}^n \in \var{x}\bK \Lhp \fM \times \var{x}\Rhp^{\preceq 1}$, then by Remark~\ref{rmk:lagrange} it has a compositional inverse if and only if $f_1 \in \bK \Lhp \fM \Rhp^*$, in which case its coefficients are given by the Lagrange inversion formula which produce a weakly restricted power series because the hypothesis on $f$ entails that $g=\var{x}/f\in \var{x}\bK \Lhp \fM \times \var{x}^\bN \Rhp^{\preceq 1}\subseteq \bK \Lhp \fM \times \var{x}^\bN \Rhp^{\prec 1}$.
    
    On the other hand any $f \in \bK \Lhp \fM \times \var{x}^\bN\Rhp^{\prec 1}$ can be written as $f=c+ \tilde{f}$ where $\tilde{f} \in \var{x}\bK \Lhp \fM \times \var{x}\Rhp^{\preceq 1}$ and $c \in \bK \Lhp \fM \Rhp^{\prec 1}$ and we see that $f$ has a compositional inverse in $\bK\Lhp \fM \times \var{x}^\bN \Rhp$ if and only if $\tilde{f}$ does and the inverses are related by the relation
    \[f^{\circ(-1)}=\tilde{f}^{\circ(-1)} \circ (\var{x}-c) \qquad \tilde{f}^{\circ(-1)}=f^{\circ(-1)} \circ (\var{x}+c).\]
    Thus $f\in \bK \Lhp \fM \times \var{x}^\bN\Rhp$ has a compositional inverse in $\bK\Lhp \fM \times \var{x}^\bN \Rhp$ if and only if $[\var{x}^1]\tilde{f}=[\var{x}^1]f=(\partial_{\var{x}}f)[\var{x}\mapsto 0]$ is a unit of $\bK \Lhp \fM \Rhp$.
\end{remark}

In Lemma~\ref{lem:inversion_tc} below, we will need the following classic fact about how to express the $n$-th coefficient and the $n$-th derivative of the inverse of a series in terms of the coefficients and derivatives of the series itself. We refer to \cite[Sec.~3.3]{comtet1974advanced} for generalities on Bell polynomials.

\begin{fact}\label{fact:inverse_der}
    For $n \in \bN^{>0}$ we have and $f \in \var{x}\bZ^{\neq 0} + \var{x}^2\bZ[\![\var{x}]\!]$ we have
    \[\begin{aligned}
        \HSder^{n}(f^{\circ(-1)})\circ f = \sum_{j <n} \frac{(-1)^j}{n \HSder(f)^{n+j}} \binom{n+j-1}{j} \hat{B}_{n-1,j}(\HSder^2(f), \ldots, \HSder^{n-j+1}(f))=\\
        =\sum_{j \le i < n} \frac{(-1)^j}{\HSder(f)^{n+j+1}} \binom{n+j}{j} (n-i) \hat{B}_{i,j}(\HSder^2(f), \ldots, \HSder^{i-j+2}(f)) \HSder^{n-i}(f).
    \end{aligned}\]
    where $\hat{B}_{i,j}$ denotes the ordinary Bell polynomial of degree $j$ and weight $i$. In particular for any ring $R$, and $f \in \var{x}R^* + \var{x}^2R[\![\var{x}]\!]$, $\HSder^n(f^{\circ(-1)}) \circ f$ is a polynomial with integer coefficients of $\HSder(f)^{-1}$ and $\HSder^{k+2}(f)$ for $k<n-1$.
    \begin{proof}%this needs some improving
        Up to considering $f[\var{x} \mapsto \var{x}+\var{t}] = \sum_{n \in \bN}\HSder^n(f) \var{t}^n \in (\bZ[\![\var{x}]\!])[\![\var{t}]\!]$,
        the statement amounts to prove that given any $f= \sum_{n\ge 1} f_n \var{x}^n$ one has 
        \[\begin{aligned}
            [\var{x}^{n}](f^{\circ(-1)})=\sum_{j < n}\frac{(-1)^{j}}{nf_1^{n+1+j}} \binom{n+j-1}{j} \hat{B}_{n-1,j}(f_2, \ldots f_{n-j})=\\
            =\sum_{j \le i < n} \frac{(-1)^j}{f_1^{n+j+1}} \binom{n+j}{j} (n-i) \hat{B}_{i,j}(f_2, \ldots f_{i-j+1}) f_{n-i}.
        \end{aligned}\]
        This in turn follows from the two forms of Lagrange inversion (\ref{eq:lagrange_inversion})
        \[\begin{aligned}
            [\var{x}^{n}](f^{\circ(-1)}) = \frac{1}{n}[\var{x}^{n-1}]\left(\frac{f}{\var{x}}\right)^{-n}=[\var{x}^{n-1}]\left( \left(\frac{f}{\var{x}}\right)^{-(n+1)}\HSder(f)\right)=\\
            =\sum_{i<n}\;[\var{x}^{i}]\!\!\left(\left( \frac{f}{\var{x}}\right)^{-(n+1)}\right) \cdot [\var{x}^{n-1-i}]\big(\HSder(f)\big).
        \end{aligned}\]
        after using $[\var{x}^{n-1-i}]\HSder(f)=(n-i)f_{n-i}$ and
        \[\begin{aligned}
            [\var{x}^k]\left(\frac{f}{\var{x}}\right)^{-m}=f_1^{-m}[\var{x}^k]\left(\frac{1}{1-(f-\var{x}f_1)/(\var{x}f_1)}\right)^{m}=\\
            =f_1^{-m}[\var{x}^k] \sum_{j} \binom{m+j-1}{j}\left(\frac{f-\var{x}f_1}{f_1\var{x}}\right)^j=\\
            =f_1^{-m} \sum_{j} \binom{m+j-1}{j} \hat{B}_{k,j}\left(\frac{f_2}{-f_1}, \ldots, \frac{f_{k-j+2}}{-f_1}\right)=\\
            =\sum_{j} f_1^{-m-j} (-1)^j \binom{m+j-1}{j} \hat{B}_{k,j}(f_2, \ldots, f_{k-j+2})
        \end{aligned}\]
        respectively with $(k,m)=(n-1,n)$ and $(k,m)=(i, n+1)$.
    \end{proof}
\end{fact}

\begin{lemma}\label{lem:inversion_tc}
    Let $\var{x}$ be a single variable and $\cA(\var{x})\subseteq \bK\Lhp \fM \times \var{x}^\bN\Rhp$ be a subring of $\bK \Lhp \fM \times \var{x}^\bN\Rhp$ closed under Hasse-Schmidt derivatives, truncations and such that $\cA(\cA(\var{x})^{\prec1})\subseteq \cA(\var{x})$.
    Let %maybe let $\cB$ be just a truncation closed subset of the set below
    \[\tilde{\cB}(\var{x})=\{f \in \cA(\var{x}) \cap \var{x}\bK \Lhp \fM \times \var{x}^\bN\Rhp^{\preceq 1}: \HSder(f) \in \big(\bK \Lhp \fM \times \var{x}^\bN\Rhp^{\preceq 1}\big)^*, 1/\HSder(f) \in \cA(\var{x})\},\]
    and suppose that $\cB(\var{x})\subseteq \tilde{\cB}(\var{x})$ is s.t.\ $\cB(\var{x})\cup \{0\}$ is closed for $\|$-truncations, then
    \[\cA\big(\cB(\var{x})^{\circ(-1)}\big)\coloneqq\big\{g \circ f^{\circ(-1)}: g \in \cA(\var{x}), f \in \cB(\var{x})\big\}\]
    is closed under $\|$-truncations.
    \begin{proof}
        Let $f \in \cB(\var{x})$. We need to show that $(a\circ f^{\circ(-1)})\|U \in \cA \big(\cB(\var{x})^{\circ(-1)}\big)$ for all $f \in \cB(\var{x})$, all final segments $U$ of $\fM$, and all $a \in \cA(\var{x})$.

        We argue by induction (cf Remark~\ref{rmk:initial_segments_of_wpo}) on the type of order of $\Supp_{\fM}(f)$, so we can assume that $(a \circ h^{\circ(-1)})\|U\in \cA(\cB(\var{x}))$ for all final segments $U$ of $\fM$, all $a \in \cA(\var{x})$ and all $h\in \cA(\var{x})^{\circ}$ with $\Supp_{\fM}(h)$ order-isomorphic to a proper final segment of $\Supp_{\fM}(f)$ (so a proper initial segment when considering $\le_\op$ for which these are well partial orders).

        By Lemma~\ref{lem:tc_composition}, it suffices to show that $f^{\circ(-1)}\|U\in \cA(\cB(\var{x})^{\circ(-1)})$ for all final segments $U$ of $\fM$: in fact, by said lemma, $(a\circ f^{\circ(-1)})\|U\in \cA (\{f^{\circ(-1)}\|U' : U' \text{ final segment of }\fM\})$, so if each $f^{\circ(-1)}\|U' \in \cA(\cB(\var{x}))$, we have 
        $(a\circ f^{\circ(-1)})\|U \in \cA(\cA(\cB(\var{x})^{\circ(-1)})^{\prec 1})=\cA(\cB(\var{x})^{\circ(-1)})$ becasue $\cA(\cA(\var{x})^{\prec 1})=\cA(\var{x})$.

        When $\Supp_\fM(f)$ is a singleton, then necessarily $f\in \var{x}\bK^*+\var{x}^2\bK[\![\var{x}]\!]$ and thus any proper $\|$-truncation of $f^{\circ(-1)}$ is $0$. So we can also assume that $\Supp_\fM(f)$ is not a singleton. 
        Note that by Taylor and Fact~\ref{fact:inverse_der}, said $d=f-h$ we have
        \[\begin{aligned}
            f^{\circ(-1)}\circ h=f^{\circ(-1)}\circ (f-d)=\sum_{n \ge 0} \big(\HSder^n(f^{\circ(-1)})\circ f\big) \cdot (-d)^n=\\
            =\var{x} + \sum_{n \ge 0} \big(\HSder^{n+1}(f^{\circ(-1)})\circ f\big) \cdot (-1)^{n+1} \cdot d^{n+1}=\\
            =\var{x} + \sum_{n \ge 0} d^{n+1} (-1)^{n+1} p_{n+1}\big((\HSder(f)^{-1}, \HSder^2(f), \ldots, \HSder^{n+2}(f)\big)\\
            %=\var{x} + \sum_{n \ge 0} d^{n+1} \sum_{j \le n} \frac{(-1)^{j+n+1}}{(n+1) \HSder(f)^{n+j+2}} \binom{n+j}{j}\hat{B}_{n,j}(\HSder^{2}(f), \ldots, \HSder^{n-j+2}(f))
            % =\var{x} + \sum_{n \ge 0} \sum_{\susbtack{m \in \bN^n\\\sum_{i<n}(i+1)m_i=n}} \binom{n+|m|}{m_0, \ldots, m_{n-1}, n} \prod_{i<n} (\HSder^{i+2} f)^{m_i} \big) \cdot \frac{d^{n+1}}{(\HSder f)^{2n+1}}
        \end{aligned}\]
        where $p_{n+1}$ is the polynomial with integer coefficients given in Fact~\ref{fact:inverse_der}.
        Thus setting $q_n\coloneqq (-1)^{n+1} p_{n+1}\big(\HSder(f)^{-1}, \HSder^2(f), \ldots, \HSder^{n+2}(f)\big)$ we get
        \[f^{\circ(-1)} = h^{\circ(-1)}+\sum_{n \ge 0} (q_n \circ h^{\circ(-1)}) \cdot (d^{n+1}\circ h^{\circ(-1)}).\]
        Note that $q_n \in \cA(\var{x})$ for all $n$ because $\cA$ is $\HSder$-closed and $\HSder(f)^{-1} \in \cA(\var{x})$. Also note that for all $n \in \bN$, $\Supp(\HSder^n(f))\subseteq \Supp(f)$ and $\Supp(\HSder(f)^{-n})$ is contained in the submonoid of $\fM$ generated by $\Supp_\fM(f)$.
        
        Now let $U$ be an upper segment of $\fM$ and $L=\fM \setminus U$. By Proposition~\ref{prop:word_segmentation}, we can find a finite segmentation $\cS$ of $\Supp_\fM(f)\le 1$, such that $\cS^*$ refines the segmentation $\{\pi^{-1}(U), \pi^{-1}(L)\}$ where $\pi: \Supp_\fM(f)^* \to \fM$ is the natural product map. Let $\cS_U\coloneqq\{S \in \cS: S\cdot  \pi(\Supp_\fM(f)^*)\subseteq U\}$, $\cS_L=\cS \setminus \cS_0$, $S_U=\bigcup \cS_U$ and $S_L=\bigcup\cS_L$. Note that $S_U$ is a final segment, $S_L$ is an initial segment and that $1 \in S_U$, so $1 \notin S_L$. Also note that if $S_U\supseteq \Supp_\fM(f)$, then $f^{\circ(-1)}\|U=f^{\circ(-1)}$, so we can assume without loss of generality that $S_U \cap \Supp_{\fM}(f)$ is a proper final segment of $\Supp_{\fM}(f)$.
        %Also note that $\pi(S_U)^*$ is a final segment of $\pi(\Supp_{\fM}(g)^*)$.
        Set $h=f\|S_U$, $d=g\|S_L$, and observe that $\Supp_\fM(d^{n+1})\subseteq \pi(S_L^{\times(n+1)})$ and $\Supp_{\fM}(q_n)\subseteq \pi(\Supp_{\fM}(f)^*)\le 1$.
        Since for $N\ge 1$ large enough $\pi(S_L^{\times N}) \subseteq L$, we thus have 
        \[f^{\circ(-1)}\|U = h^{\circ(-1)}\|U + \sum_{n=0}^N \big((q_n \circ h^{\circ(-1)}) \cdot (d^{n+1}\circ h^{\circ(-1)})\big)\|U.\]
        By Lemma~\ref{lem:product_segmentation} each $\big((q_n \circ h^{\circ(-1)}) \cdot (d^{n+1} \circ h^{\circ(-1)})\big)\|U$ is in the algebra generated by the $\|$-truncations of the $\HSder^{i+2}(f)\circ h^{\circ(-1)}$ for $i \ge 0$ and of $\HSder(f)^{-1} \circ h^{\circ(-1)}$.
        But using Lemma~\ref{lem:tc_composition}, from the inductive hypothesis, the fact that $\cA$ contains $\HSder^i(f)$ for all $f$ as well as $1/\HSder(f)$ and that $\cA$ is truncation closed, we deduce that $\cA(\cB(\var{x})^{\circ(-1)})$ contains all such $\|$-truncations and thus finally $f^{\circ(-1)}\|U \in \cA(\cB(\var{x})^{\circ(-1)})$.
    \end{proof}
\end{lemma}

\begin{corollary}\label{cor:inverses_tc}
    Let $\var{x}$ be a single variable and $\cA(\var{x})\subseteq \bK\Lhp \fM \times \var{x}^\bN\Rhp$ be a subring of $\bK \Lhp \fM \times \var{x}^\bN\Rhp$ closed under Hasse-Schmidt derivatives and truncations, and such that $\cA(\cA(\var{x})^{\prec1})\subseteq \cA(\var{x})$, $\var{x}, 1/(1-\var{x}) \in \cA(\var{x})$, and $\big(\cA \cap (\var{x}\bK^*)\big)/\var{x}$ is a subgroup of $(\bK^*, \cdot, 1)$. Then $\cA\big((\cA(\var{x})^\circ)^{\circ(-1)}\big)$ is closed under truncations.
    \begin{proof}
        Note that the hypothesis on $\cA(\var{x})$ entails that for all $f \in \cA(\var{x})^\circ \cap \var{x}\bK \Lhp \fM \times \var{x}^\bN\Rhp$, $1/\HSder(f) \in \cA(\var{x})$: in fact letting $f_1=\HSder(f)[\var{x} \mapsto 0]$ we get $f_1=k+ \varepsilon$ where $k \in \big(\var{x}\bK^* \cap \cA(\var{x})\big)/\var{x}$, but then also $k^{-1}\var{x} \in \cA(\var{x})$ and thus $h=1/(k-\var{x})=k^{-1}/(1-k^{-1}\var{x}) \in \cA(\var{x})$ and $1/\HSder(f)=h(\HSder(f)-k)\in \cA(\var{x})$.
        It follows that setting $\cB(\var{x})=\cA(\var{x})^\circ \cap \var{x}\bK \Lhp \fM \times \var{x}^\bN\Rhp$, we have that $\cA(\cB(\var{x}))$ is truncation closed. But 
        \[\cA((\cA(\var{x})^\circ)^{\circ(-1)})=\cA(\cB(\cT(\var{x}))\]
        where $\cT(\var{x})=\{\var{x}-a: a \in \cA(\var{x}) \cap \bK \Lhp \fM \Rhp^{\prec 1}\}$ which is also truncation closed, so we can conclude by Lemma~\ref{lem:tc_composition}.
    \end{proof}
\end{corollary}

\begin{notation}
    Let $\var{x} \in \Var$, $\var{y}\subseteq \Var\setminus \{\var{x}\}$, and $f \in \bK \Lhp \fM\times \var{x}^\bN\var{y}^\bN\Rhp$, we can regard $f$ as an element of $\bK \Lhp \fN\times \var{x}^\bN\Rhp$ with $\fN=\fM \times \var{y}^\bN$ via the ``syntactical'' isomorphism $\tau:\fM\times \var{x}^\bN\var{y}^\bN \to \fN$, $\tau(\fm, \var{x}^n\var{y}^m)=\big((\fm, \var{y}^m), \var{x}^n)$. We will say that $f$ has a (restricted) compositional inverse in $\var{x}$ if $\tau f \in \bK \Lhp \fM \times \var{y}^\bN \times \var{x}^\bN  \Rhp$ has a restricted compositional inverse and write $f^{\circ\var{x}(-1)}$ for $\tau^{-1} \big((\tau f)^{\circ(-1)}\big)$.
    We will denote the set of restricted series with a restricted compositional inverse in $\var{x}$, by
    \[\bK\Lhp \fM \times \var{x}^\bN \var{y}^\bN\Rhp^{\circ \var{x}}
    \coloneqq \big\{f \in \bK\Lhp \fM \times \var{x}^\bN \var{y}^\bN\Rhp^{\prec 1}: (\partial_{\var{x}}f)[\var{x}\mapsto 0] \in \big(\bK \Lhp \fM \times \var{y}^\bN\Rhp^{\preceq 1}\big)^*\big\},\]
    and write $\cF(\var{x}, \var{y})^{\circ\var{x}}$ for $\cF(\var{x}, \var{y}) \cap \bK\Lhp \fM \times \var{x}^\bN \var{y}^\bN\Rhp^{\circ \var{x}}$. 
\end{notation}

\begin{remark}\label{rmk:inverses_vs_implicit}
    Let $\var{x} \in \Var$, $\var{y}\subseteq \Var\setminus \{\var{x}\}$, and $f \in \bK \Lhp \fM\times \var{x}^\bN\var{y}^\bN\Rhp^{\prec 1}$. 
    Note that if $f\in \bK\Lhp \fM \times \var{x}^\bN\var{y}^\bN\Rhp^{\circ\var{x}}$, then $g\coloneqq f^{\circ\var{x}(-1)}[\var{x}\mapsto 0] \in \bK \Lhp \fM \times \var{y}^\bN\Rhp$ is such that $f[\var{x}\mapsto g]=0$.

    Conversely for $\var{z} \in \Var \setminus (\{\var{x}\}\cup \var{y})$, $h=f+\var{z}$, satisfies $(\partial_{\var x}h)[\var{x}\mapsto 0]\in \big(\bK \Lhp \fM \times \var{y}^\bN \Rhp^{\preceq 1}\big)^*\subseteq \big(\bK \Lhp \fM \times \var{y}^\bN\var{z}^\bN \Rhp^{\preceq 1}\big)^*$ if and only if $f\in \bK \Lhp \fM \times \var{x}^\bN \var{y}^\bN\Rhp^{\circ\var{x}}$, in which case $g=f^{\circ \var{x}(-1)}[\var{x} \mapsto \var{z}]$ satisfies $h[\var{x} \mapsto g]=0$. Thus if $\cF$ is a $\bZ$-algebra of weakly restricted power series with $\var{z}\subseteq \cF(\var{z})$ for all $\var{z}\subseteq \Var$, then it has implicit functions (Definition~\ref{defn:closure_ppties}\ref{psFppties:implicit}) if and only if for every $f \in \cF(\var{x}, \var{y})^{\circ\var{x}}$, $f^{\circ\var{x}(-1)}\in \cF(\var{x}, \var{y})$. 
\end{remark}

We are ready to prove the specialization of Theorem~\ref{thm:mainT0} to the case $\Lambda=\bN$.

\begin{proposition}\label{prop:main_for_bN}
    Suppose that $\cL$ is a $\bZ$-algebra of weakly restricted power series, which is $\HSder$-closed and truncation-closed.  
    \begin{enumerate}
        \item The closure $\cL^*$ of $\cL$ under compositions is closed under truncations;
        \item The closure $\cL^{**}$ of $\cL$ under compositions and implicit functions is closed under truncations.
    \end{enumerate}
    Furthermore if $\cL$ contains $\fM$ and is closed under monomial division, and $\fM$ is cone ordered, then $\cL^*$ and $\cL^{**}$ are both closed under monomial division.
    \begin{proof}
        (1) Let $\cL^*$ denote the closure of $\cL$ under compositions. Note that since $\cL$ is a $\HSder$-closed $\bZ$-algebra, so is $\cL^*$ (by Faà di Bruno's formula).
        
        Note that by Remark~\ref{rmk:coefficient-wise_tc_is_tc}, to prove (1) it suffices to show that $(\cL^*)^{\|}=\cL^*$.
        By Lemma~\ref{lem:tc_product}, $(\cL^*)^{\|}$ is a $\HSder$-closed $\bZ$-algebra. On the other hand, by Lemma~\ref{lem:tc_composition}, $(\cL^*)^{\|}$ is closed under compositions. Now since $\cL$ is closed under truncations, $(\cL^*)^{\|}\subseteq \cL^*$ contains $\cL$ and thus it must be $(\cL^*)^{\|}=\cL^*$.

        (2) Let $\cL^{**}$ denote the closure of $\cL^*$ under implicit functions, which by Remark~\ref{rmk:inverses_vs_implicit} is the same as the closure under taking compositional inverses in any given variable. Again by Remark~\ref{rmk:coefficient-wise_tc_is_tc}, to prove (2) it suffices to show that $(\cL^{**})^{\|}=\cL^{**}$ and again $(\cL^{**})^{\|}$ is $\HSder$-closed, closed under compositions, and contains $\cL^*$. Also observe that given any single variable $\var{x}$:
        \begin{enumerate}
            \item $\cL^{**}$ contains $\var{x}/(1-\var{x})$ as solution in $\var{y}$ to $\var{y}-\var{y} \var{x}-\var{x}=0$, so $\cL^{**}$ also contains $1/(1-\var{x})=1+\var{x}/(1+\var{x})=\sum_{n\ge0} \var{x}^n$; thus we have $1/(1-\var{x}) \in (\cL^{**})^{\|}$.
            \item $\big((\cL^{**})^{\|} \cap \var{x}\bK^*\big)/\var{x}$ is a subgroup of $\bK^*$ because if $k\var{x} \in \cL^{**}$ with $k \in \bK^*$, then its compositional inverse $k^{-1}\var{x}$ is must be in $\cL^{**}$.
        \end{enumerate}
        It then follows from Corollary~\ref{cor:inverses_tc}, that $(\cL^{**})^{\|}$ is also closed under taking compositional inverses in any given variable, thus $(\cL^{**})^{\|}=\cL^{**}$ and the proof of (2) is complete.

        Finally in the case $\fM$ is cone ordered and $\cL$ is closed under monomial division and contains $\fM$, the ``furthermore'' can be proven with an analogous argument letting $(\cL^*)^{/\!/}$ be largest family closed under monomial divisions and truncations contained in $\cL^*$ and observing it must be an algebra closed under compositions by Lemmas~\ref{lem:md-closed_gen_subring} and \ref{lem:md_composition}. As for $\cL^{**}$, it suffices to show that if $f\in \cL^*(\var{x},\var{y})$ has a compositional inverse $f^{\circ\var{x}(-1)}$ in $\cL^*(\var{x},\var{y})$ for the variable $\var{x}$, then $1=\max\Supp_{\fM}(f)=\max\Supp_{\fM}(f^{\circ\var{x}(-1)})$, so adjoining compositional inverses does not affect closure under monomial division.
    \end{proof}
\end{proposition}

Deducing Theorem~\ref{thm:mainT0} from Proposition~\ref{prop:main_for_bN} above is now a essentially a matter of syntactical manipualtions.

\begin{definition}\label{def:syntactical_equiv}
    Let $\Var^{(\Lambda)}$ denote the free $\Lambda$-module on $\Lambda$ ordered by $\var{x}^\lambda \le 1$ if and only if $\lambda \ge 0$ component wise.
    We say that a language $\cL$ of weakly restricted $\bN$-power series with coefficients from $\bK \Lhp \fM \times \Var^{(\Lambda)}\Rhp$ is
    \begin{enumerate}
        \item \emph{closed under shifting classical variables} if for all $\var{x},\var{y}, \var{z}$ finite pairwise disjoint sets of variables, whenever $f\in \cL(\var{x})$ is such that $f \in \bK \Lhp \fM \times \var{y}^{\Lambda}\var{z}^\bN \times \var{x}^\bN\Rhp$, we have that $\rho f \in \cL(\var{x}, \var{z})$ where $\rho(\fm,\var{y}^{\gamma}\var{z}^n, \var{x}^m)=(\fm, \var{y}^\gamma, \var{z}^n\var{x}^m)$ and whenever $g \in \cL(\var{x}, \var{y})$, $\rho^{-1}g\in \cL(\var{x})$;
        \item a \emph{$\Lambda$-language} if it is closed under shifting classical variables and each $\cL(\var{x})$ is closed under endomorphisms induced by the monomial transformations of the form
        \[\id_{\fM} \times \sigma \times \id_{\var{x}^\bN}: \fM \times \Var^{(\Lambda)}\times \var{x}^\bN \to \fM \times \Var^{(\Lambda)}\times \var{x}^\bN,\]
        where $\sigma: \Var^{(\Lambda)}\to \Var^{(\Lambda)}$ is induced by a finite-to-one map on $\Var$.
    \end{enumerate}
    Given a language $\cF$ of weakly restricted $\Lambda$-power series with coefficients from $\bK \Lhp \fM \Rhp$, we can form the language $\cL=\cF^{|\bN}$ of weakly restricted $\bN$-power series with coefficients from $\bK \Lhp \fM \times \Var^{(\Lambda)}\Rhp$, by setting
    \[\cL(\var{x})=\cF^{|\bN}(\var{x})=\big\{\sigma_{\var{x}, \var{y}}^{\delta}f: \var{y} \subseteq \Var\setminus \var{x},\; \delta\in \Var^{\var{y}}, f \in \cF(\var{x}, \var{y}) \cap \bK \Lhp \fM \times \var{x}^\bN\var{y}^\Lambda\Rhp\big\},\]
    where $\sigma_{\var{x}, \var{y}}:\fM \times \var{x}^\bN \var{y}^\Lambda \to (\fM \times \var{z}^\bN)\times \var{y}^\Lambda$ is given by $\sigma_{\var{x}, \var{y}}(\fm, \var{x}^n\var{y}^\gamma)=(\fm, \delta(\var{y}^\gamma), \var{x}^n)$.
    Conversely, given a language $\cL$ of weakly restricted $\bN$-power series with coefficients from $\bK \Lhp \fM \times \Var^{(\Lambda)}\Rhp$, we can form a family $\cF=\cL^{\Lambda}$ of weakly restricted $\Lambda$-power series by setting
    \[\cF(\var{x})=\{\tau_{\var{y}, \var{x}} f \in \cL(\var{y}): \var{y}\subseteq \var{x},\, \Supp(f)\subseteq \fM \times (\var{x}\setminus \var{y})^{\Lambda}\times \var{y}^\bN\},\]
    where $\tau_{\var{x},\var{y}}: \fM \times (\var{x}\setminus \var{y})^{\Lambda} \times \var{y}^\bN \to \fM \times \var{x}^\Lambda$ is given by $\tau_{\var{x}, \var{y}}(\fm, (\var{x}\setminus \var{y})^\gamma, \var{y}^n)=(\fm, (\var{x}\setminus \var{y})^\gamma \var{y}^n)$.
\end{definition}

\begin{remark}%TODO:turn this into a Lemma?
    Given a language $\cF$ of weakly restricted $\Lambda$-power series with coefficients from $\bK \Lhp \fM \Rhp$, $\cF^{|\bN}$ is always a $\Lambda$-language and $(\cF^{|\bN})^\Lambda=\cF$. Furthermore if a language of weakly restricted $\bN$-power series with coefficients from $\bK \Lhp \fM \times \Var^{(\Lambda)} \Rhp$ is a $\Lambda$-language, then $\cL^\Lambda$ is a language of weakly-restricted $\Lambda$-power series and $(\cL^\Lambda)^{|\bN}=\cL$.
\end{remark}

\begin{lemma}
    If $\cL$ is a language of weakly restricted $\bN$-power series with coefficients from $\bK\Lhp \fM \times \Var^{(\Lambda)}\Rhp$ which is a $\Lambda$-language and closed under compositions, then $\cL^{\Lambda}$ is closed under classical compositions.
    \begin{proof}
        Let $\tau_{\var{y}, \var{x}}f \in \cL^\Lambda(\var{x})$. Since $\cL^\Lambda$ is closed under shifting classical variables, we can assume that $\var{y}$ are all the classical variables of $\tau_{\var{y},\var{x}}f$, but then the fact that $\cL$ is closed under compositions immediately yields that for any assignment $g \in \big(\cL^\Lambda(\var{z})^{\prec 1}\big)^{\var{y}}$, $f(\var{x}\setminus \var{y}, g) \in \cL^\Lambda$.
    \end{proof}
\end{lemma}

\begin{proof}[Proof of Theorem~\ref{thm:mainT0}]
    By Lemma~\ref{lem:alg_generated_by_tc} we can assume that $\cF$ is a $\bZ$-algebra of restricted $\Lambda$-power series.

    Let $\cL=\cF^{|\bN}$ be as in Definition~\ref{def:syntactical_equiv}. By Proposition~\ref{prop:main_for_bN}, the closure $\cL^*$ (resp.\ $\cL^{**}$) of $\cL$ under compositions (resp.\ compositions and implicit functions) is closed under truncations.
    
    Note that said $\cL_1$ and $\cL_2$ the closures under shifting classical variables of $\cL^*$ and $\cL^{**}$ respectively, we have that $\cL_1$ and $\cL_2$ are still closed under truncations, and are $\Lambda$-languages.

    Now $\cL_1^{\Lambda}$ (resp.\ $\cL_2^{\Lambda}$) is language of weakly restricted $\Lambda$-power series and is the closure of $\cF$ and under classical compositions (resp.\ classical compositions and implicit functions), so we are done.

    Finally under the additionally hypotheses of the ``furthermore'', $\cL=\cF^{|\bN}$ satisfies the additional hypotheses of the ``furthermore'' in Proposition~\ref{prop:main_for_bN}, so $\cL^*$, $\cL^{**}$, $\cL_1$, $\cL_2$, and $\cL_1^{\Lambda}$, $\cL_2^{\Lambda}$ are all closed under monomial division. 
\end{proof}

\section{Serial power-bounded structures}\label{sec:serial}

Recall that a theory $T$ in a language $L$ is called a \emph{Skolem theory} (cf \cite[Sec.~3.1]{hodges1993MT}) if given tuples of variables $\var{x}$ and $\var{y}$ and a formula $\varphi(\var{x}, \var{y})$ in $L$, there is a $L$-term $t(\var{x})$ such that $T \models \forall \var{x},\; \bigg(\big(\exists \var{y}, \varphi(\var{x}, \var{y})\big) \rightarrow \varphi(\var{x}, t(\var{x}))\bigg)$. Such a term $t(\var{x})$ is called a \emph{Skolem} function for $\varphi(\var{x}, \var{y})$. We will need the following variant of the notion of Skolem theory.

\begin{definition}\label{defn:piecewise_Skolem}
    We call a theory $T$ in a language $L$ a \emph{piecewise Skolem theory} if given a tuple of variables $\var{x}$, a variable $\var{y}$ and a formula $\varphi(\var{x}, \var{y})$ in $L$, there are finitely many $L$-terms $\{t_0(\var{x}), \ldots, t_{n-1}(\var{x})\}$ such that 
    \[T \models \forall \var{x},\; \bigg(\big(\exists \var{y}, \varphi(\var{x}, \var{y})\big) \rightarrow \bigvee_{i<n}\varphi(\var{x}, t_i(\var{x}))\bigg).\]
\end{definition}

\begin{remark}
    A theory $T$ is a piecewise Skolem theory if and only if it is model complete and universally axiomatized. In particular such a theory $T$ eliminates quantifiers and moreover every substructure of a model of $T$ is a model of $T$. 
\end{remark}
		
The following definition is partially modelled upon the definition of GQA in \cite{rolin2015quantifier}. 
		
\begin{definition}\label{def:serial}
    Let $\bK_L=(\bK, L)$ be a power-bounded o-minimal field in some functional language expansion $L$, consisting for each $n$ and $r \in (\bK^{>0})^n$ of a $\bK$-algebra $L(r)$ of $n$-ary smooth functions on $\prod_{i<n} (0,r_i)$ and containing the restrictions of the coordinate projections.
	Suppose furthermore that $L$ satisfies:
	\begin{enumerate}[label=(S\arabic*), ref=(S\arabic*)]
		\item\label{Serial:analyticity} for every $f \in L(r)$, $x \in \prod_{i<n} [0, r_i]$ and $\sigma \in \{-1, 0, +1\}^n$, there are $r'$ and $f_{x, \sigma}\in L(r')$ such that $f(x+\sigma z) = f_{x, \sigma} (z)$ for all small enough $z>0$;
        %\item $L(r)=\bigcup\{L(r')|_{\prod_i(0,r_i)}: r'>r\}$; 
		\item\label{Serial:derivatives} the $\bK$ algebra $L_n$ of germs at the origin of functions in $L(r)$ for some $r \in (\bK^{>0})^n$ is closed under renormalized partial derivatives, that is if $f \in L_n$, then $x_i\partial_{i}f\in L_n$ where $x_i$ is the germ of the projection on the $i$-th coordinate;
		\item\label{Serial:Skolem} 
		$\bK_L$ has a piecewise Skolem Theory.
	\end{enumerate}		
	Let $\Lambda$ be the field of exponents of $\bK_L$. Let for every $n$, $\cT=\cT_n$ be an injective algebra embedding $\cT: L_n \to \bK \Lhp \var{x}^{\Lambda} \Rhp$ satisfying $\cT(x_i)= \var{x}_i$ and $\cT(x_i\partial_i f) = \var{x}_i\partial_{i}\cT(f)$ for any coordinate projection $x_i: \bK^n \to \bK$.
			
	Given a multiplicatively written ordered $\Lambda$-vector space $\fM$, we can interpret each $f \in L(r)$ on $\bK\Lhp \fM \Rhp$ by setting for each $x \in \prod_{i<n} [0, r_i]_\bK$ and $\varepsilon \in \bK \Lhp \fM^{<1}\Rhp^n$ such that $0<x+\varepsilon <r$ component-wise,
    \[f(x+\varepsilon):=\cT(f_{x, \sign(\varepsilon)})(|\varepsilon|).\]
    We denote by $\bK \Lhp \fM\Rhp^{\cT}$, the field $\bK\Lhp \fM \Rhp$ expanded with such interpretation of symbols in $L$.
            
    We will say that $(\bK_L, \cT)$ is \emph{$\fM$-serial} if $\bK_L$ satisfies \ref{Serial:analyticity}, \ref{Serial:derivatives} and \ref{Serial:Skolem} and furthermore
	\begin{enumerate}[label=(S\arabic*), ref=(S\arabic*), resume]
	   \item\label{Serial:Seriality} $\bK \preceq \bK\Lhp \fM \Rhp^\cT$.
	\end{enumerate}
			
	We say that $(\bK_L, \cT)$ is \emph{serial} if it is $\fM$-serial for every multiplicatively written ordered $\Lambda$-vector space $\fM$. If the image of $\cT_n$ is truncation-closed for every $n$, then we say that $(\bK_L, \cT)$ is \emph{truncation-closed}. We denote by $\cT(L)$ the (closure under variable-reindexings of the) family of power series in the image of $\cT$.
\end{definition}
		
\begin{remark}
	If each $L_n$ contains (the germ of) $1/x_i$ for each coordinate function $x_i$ (as would usually be the case since $\bK_L$ is assumed to have a Skolem theory), then the condition on $\cT$ can be simplified to that of being a truncation-closed differential algebra embedding.
\end{remark}
		
\begin{remark}
	Note that if $(\bK_L, \cT)$ is truncation closed, then $\cT(L)$ is an almost fine $\bK$-algebra of generalized series.
\end{remark}
		
\begin{lemma}
	Suppose $(\bK_L,\cT)$ is $\fM$-serial and truncation-closed and $\bK \cup \fM \subseteq \bE \subseteq \bK \Lhp \fM\Rhp^{\cT}$ is a truncation-closed subring. Then $\bE$ is $L$-closed if and only if it is $\cT(L)$-closed. %I am assuming L contains the series 1/(1-x)
	\begin{proof}
		Suppose $\bE$ is $\cT(L)$-closed. Let $f \in L(r)$ for some $r \in (\bK^{>0})^n$. If $x \in \bE^n$ is such that $0<x_i<r_i$, we can write $x=\tilde{x}+\sigma \varepsilon$ for some $\tilde{x} \in \prod_{i<n}[0,r_i]_{\bK}$, some $\sigma\in \{-1,0,1\}^n$ and some tuple $\varepsilon$ of positive infinitesimals. Since $\bE$ contains $\bK$, we have $\tilde{x} \in \prod_{i<n} [0, r_i]_{\bK} \subseteq \bE^n$ and since $\bE$ is a $\cT(L)$-closed subring it follows that also $\varepsilon \in \bE^n$. Thus, since $\bE$ is $\cT(L)$-closed we have $f(x)=f_{\tilde{x}, \sigma}(\varepsilon) \in \bE$.
				
		Conversely if $\bE$ is $L$-closed, then it must clearly be $\cT(L)$-closed.
	\end{proof}
\end{lemma}
		
\begin{proposition}\label{prop:tc-ppty}
	Suppose $(\bK_L, \cT)$ is $\fM$-serial and truncation-closed and $\bK \subseteq \bE \subseteq \bK \Lhp \fM\Rhp$ is a truncation-closed elementary substructure. If $x \in \bK \Lhp \fM \Rhp$ is such that for all $\fm \in \fM$, $x |\fm \neq x \rightarrow x|\fm \in \bE$, then $\bE \langle x \rangle_L$ is truncation-closed.
	\begin{proof}
		Notice that $\fN:=\fM \cap\bE$ must be a subgroup of $\fM$ and that $\bE = \langle X \rangle_{\cT(L)}$ for some truncation-closed subset $X$ (e.g.\ $X= \bE$).
		We distinguish two cases. If $\Supp x$ has a minimum $\fm$, then $\bE \langle x\rangle_L= \bE \langle \fm \rangle= \langle X\cup \fN'\rangle_{\cT(L)}$ where $\fN':= \fN \fm^{\Lambda}$ and we can conclude by Theorem~\ref{introthm:A}.
				
		If instead $\Supp x$ has no minimium, then $\bE \langle x \rangle_L = \bE \langle x \rangle_{\cT(L)} = \langle X\cup\{x\}\rangle_{\cT(L)}$ and we are done once again by Theorem~\ref{introthm:A}. 
	\end{proof}
\end{proposition}
		
\begin{lemma}
	Suppose $(\bK_L, \cT)$ is $\fM$-serial for every $\fM$ in some class $\cC$ and $L_0 \subseteq L$. Then there is an expansion by definitions $L_0\subseteq L_* \subseteq L$ such that $(\bK|_{L_*}, \cT|_{L_*})$ is $\fM$-serial for every $\fM\in \cC$.
	\begin{proof}
       Let $L_*$ consist of the function symbols in $L$ that are definable in $\bK_{L_0}$. Since Skolem functions are definable in o-minimal structures, it follows that the reduct $\bK|_{L_*}$ has a piecewise Skolem theory.
	\end{proof} 
\end{lemma}
		
\begin{corollary}\label{cor:anreducts_serial}
	If $(\bK_L, \cT)$ is $\fM$-serial and the image of each $\cT_n$ lies within the Puiseux series $\bigcup_{m \in \bN} \bK \Lhp \var{z}^{1/(m+1)\bZ}\Rhp$, then every reduct of $(\bK_L, \cT)$ has an expansion by definitions that is serial and truncation-closed.
    \begin{proof}
        It suffices to observe that if the image of each $\cT_n$ lies within the $\bK$-algebra $\bigcup_{n \in \bN} \bK \Lhp \var{z}^{1/(m+1)\bZ}\Rhp$, then $(\bK_L, \cT)$ is necessarily truncation-closed.
    \end{proof}
\end{corollary}
		
\begin{example}
	Every real closed field $\bK$ has an expansion by definable functions that is serial. 
	Consider the language $L'$ consisting of $+$, $\cdot$, $(-)^{-1}$, $\{(-)^{1/(n+1)}: n \in \bN\}$ and for every $n \in \bN$, $\sigma \in \{0,\pm 1\}^{n}$, and $r \in (\bK^{>0})^n$ for which it exists, a smooth function $f : \prod_{i}(0, r_i) \to \bK$ satisfying
	\[\sigma_0 \var{x}_0 +  \sigma_1 \var{x}_1 f(\var{x})+ \sigma_2 \var{x}_2 f(\var{x})^2 + \cdots +\sigma_{n-1}\var{x}_{n-1} f(\var{x})^{n-1} = f(\var{x}),\]
	along with all of its partial derivatives. %TODO:for all...
	Then for each $r \in (\bK^{>0})^n$, let $L(r)$ consists of the algebra of $L'$-term-definable functions. Note that $L'$ is closed under partial derivatives and satisfies \ref{Serial:analyticity} and \ref{Serial:derivatives}. To see it has a piecewise Skolem theory \ref{Serial:Skolem}, observe that in the language $L'$ the class of real closed fields is universally axiomatized and eliminates quantifiers, hence all definable functions are piecewise term-definable.
            
	Finally note that the $\cT_n$ are defined naturally on the germs of the $f$s described above as they arise as implicit functions for polynomial equations and that \ref{Serial:Seriality} is satisfied for all $\fM$ because for all divisible ordered Abelian groups $\fM$, $\bK \Lhp \fM \Rhp$ is a real closed field and RCF is model complete.
\end{example}
		
\begin{example}
    Recall that the structure $\bR_{an}$ is given by the ordered field of reals expanded by function symbols $f$ for each function $f:[-1,1]^n \to \bR$ given by the restriction $g|_{[-1,1]^n}$ of some analytic function $g: U \to \bR$ for some open $U\supseteq [-1,1]^n$. The symbol $f$ is interpreted as $g|_{[-1,1]^n}$ on $[-1,1]^n$ and is set to have value $0$ on $\bR^n \setminus [-1,1]^n$.

    As observed in \cite{Dries1986Generalization}, the fact that $\bR_{an}$ is model complete is essentially Gabrielov's theorem of the complement \cite[Thm.~1]{gabrielov1968projections}. The proof also yields o-minimality of $\bR_{an}$. 
    Moreover $\bR_{an}$ eliminates quantifiers when expanded with a function symbol $D$ for restricted division, i.e.\ $D(x,y)=x/y$ for $(x,y) \in [0,1]^2\setminus [0,1]\times \{0\}$ and $D=0$ on $\bR^2 \setminus [0,1] \times (0,1]$ (\cite[(4.3) Main Result]{denef1988p}). 
        
	By \cite[Thm.~2.14]{dries1994elementary} the structure $\big(\bR_{an}, (\sqrt[n]{-})_{n>0}, (-)^{-1}\big)$ is serial with the natural $\cT$. In particular, by Corollary~\ref{cor:anreducts_serial} every reduct of $\bR_{an}$ has an expansion by definitions that is serial.
\end{example}

\begin{example}
	\cite[Prop.~10.4]{dries2000field}, can be restated by saying that the structure $\bR_{\cG}$ is interdefinable with a $\fM$-serial structure for all $\fM$ with finitely many Archimedean classes.
\end{example}
		
\begin{question}\label{quest:GQAserial}
	Let $\cA$ be a generalized quasianalytic algebra as defined in \cite{rolin2015quantifier}. Is $\bR_\cA$ interdefinable with a serial structure?
\end{question}
		
\begin{remark}
	Notice that if $\lambda$ is greater than every ordinal embeddable in the field of exponents $\Lambda$ then $\bK \Lhp \fM \Rhp_\lambda\subseteq \bK \Lhp \fM \Rhp^{\cT}$ is an $L$-substructure. We denote the resulting expansion of $\bK \Lhp \fM \Rhp_\lambda$ to a $L$-structure by $\bK \Lhp \fM \Rhp_\lambda^{\cT}$.
\end{remark}
			
\section{\texorpdfstring{$T$}{T}-convexity, wim-constructible extensions, and \texorpdfstring{$T$-$\lambda$}{T-λ}-spherical completions}\label{sec:Spherical_completion_review}

Recall that a $T$-convex valuation ring $\cO$ on an o-minimal expansion of a field $\bE$, is a convex subring $\cO \subseteq \bE$, closed under total unary $\emptyset$-definable continuous functions (\cite[(2.7)]{dries1995t}).
We will denote by $\co$ the maximal ideal of $\cO$ and write the associated \emph{dominance relation} on $\bE\setminus \{0\}$ as $x \preceq_\cO y \Leftrightarrow x/y\in \cO$. Similarly for $x \asymp y \Leftrightarrow (x \preceq y\; \& \; y\preceq x)$, $x \prec y \Leftrightarrow (x\preceq y\; \& \; y \not \preceq x)$, and $x \sim y \Leftrightarrow x-y \prec x$.
The value group of $(\bE,\cO)$ will be denoted by $(\val(\bE, \cO), +):=(\bE^{\neq 0}/(\cO \setminus \co), \cdot)$ or $\val_\cO(\bE)$ and the valuation by $\val_{\cO} : \bE \to \val(\bE, \cO) \cup \{\infty\}$. As usual the value group will be ordered setting $\val(\bE, \cO)^{\ge0}=\val_\cO(\cO^{\neq 0})$ and $\val_\cO(0)=\infty$. In particular $\val_{\cO}(x)>\val_{\cO}(y)$ if and only if $x \prec_{\cO} y$.
The residue field $\cO/\co$ will be denoted by $\res(\bE, \cO)$ or $\res_{\cO}(\bE)$ and the quotient map by $\res_{\cO}: \cO \to \cO/\co$. The residue-value sort $(\bE^{\neq0}/(1+\co), \cdot)$ will be denoted by $\rv_\cO(\bE)$ and the quotient map by $\rv_{\cO} : \bE^{\neq 0} \to \rv_\cO(\bE)$. The subscript $\cO$ may be omitted if there is no ambiguity.

The common theory of the expansions $(\bE, \cO)$ of models of $T$ by a predicate $\cO$ for a $T$-convex valuation ring is denoted by $T_\convex^-$ and the theory $T_\convex$ is defined as $T_\convex^- \cup \{\exists x \notin \cO\}$. By results of van den Dries and Lewenberg (\cite[(3.10), (3.13), and (3.14)]{dries1995t}), the theory $T_\convex$ is complete and weakly o-minimal and if $T$ is universally axiomatized and eliminates quantifiers then $T_\convex$ eliminates quantifiers.

Also recall that an elementary substructure $\bK\preceq \bE$ is said to be \emph{tame}, denoted $\bK \preceq_\tame \bE$, when for every definable subset $S$ of $\bE$, if $S \cap \bK$ is bounded, then $S\cap \bK$ has a supremum in $\bK$. By a theorem of Marker and Steinhorn \cite{marker1994definable}, $\bK \preceq_\tame \bE$ if and only if for every tuple $c$ in $\bE$, $\tp(c/\bK)$ is a definable type.

$T$-convex subrings are related to tame pairs of o-minimal structures. In fact if $(\bE, \cO) \models T_\convex^-$, then $\bK$ is maximal in $\{\bK' : \bK' \preceq \bE, \; \bK' \subseteq \cO\}$ if and only if $\bK\preceq_\tame \bE$ and $\cO=\CH(\bK)$ (see \cite[(2.11) and (2.12)]{dries1995t}). In particular this entails that there is a unique map $\std_\bK: \cO \to \bK$ such that for all $x \in \cO$, $\std_{\bK}(x)-x \in \co$. Moreover by \cite[Thm.~A]{dries1997t}, $\std_\bK$ induces an isomorphism between the induced structure on the imaginary sort $\res(\bE,\cO)$ and $\bK$.

We will now introduce some  further nomenclature which will be useful in the last section. In particular we will recall and extend some notions and results given in \cite{freni2024t} for models of $T_\convex$, to models of $T_\convex^-$.
	
\begin{definition}
	Let $T$ be the theory of an o-minimal field. For $(\bE, \cO)\models T_\convex$, we will call 
	\begin{itemize}
		\item \emph{residue $T$-section} a $\bK \preceq_\tame \bE$, such that $\cO=\CH(\bK)$;
		\item \emph{$\bK$-monomial group} a subgroup $\fM \subseteq (\bE^{>0}, \cdot)$ stable under the action of $\Exponents(\bE) \cap \bK$ and such that $\val_\cO |_\fM: \fM \to \val_{\cO}(\bE)$ is an isomorphism.
	\end{itemize}
\end{definition}
	
\begin{remark}
	If $T$ is power-bounded, then by Miller's growth dichotomy \cite{miller1993growth}, each exponent is $\emptyset$-definable, so $\Exponents(\bE)=\Exponents(T)$ is the field of exponents of the theory and being a $\bK$-monomial group does not depend on the $T$-residue section $\bK$. In such a case we will therefore just say \emph{monomial group}.
		
	If $T$ is exponential, then $\Exponents(\bE)=\bE$ and a monomial group is required to be closed under $\fm \mapsto \exp(k\log \fm)$ for each $k \in \bK$. 
\end{remark}
		
By an \emph{embedding} of models of $T_\convex^-$, $\iota: (\bE, \cO) \to (\bE_1, \cO_1)$ we will mean an elementary embedding $\iota: \bE \to \bE_1$ such that $\iota^{-1}(\cO_1)=\cO$. Thus if $T$ is universally axiomatized and eliminates quantifiers, this is just an embedding in the language $L_\convex$ given by the language $L$ of $T$, toghether with a unary predicate for $\cO$.

\begin{definition}\label{def:types2}
	Let $(\bU, \cO') \models T_\convex^-$ and $\bE \prec \bU$, $\cO:=\bE \cap \cO'$ (so $(\bE, \cO) \models T_\convex^-$). We say that $x\in \bU\setminus \bE$ is \emph{weakly immediate (wim)} (over $\bE$) if its cut is an intersection of valuation balls and that it is \emph{weakly immediately generated (wimg)} if $\bE \langle x \rangle\setminus \bE$ contains a weakly immediate element.
			
	We say that $x$ is \emph{residual} if $(\bE \langle x \rangle, \cO' \cap \bE \langle x \rangle)$ has a strictly larger residue field than $(\bE, \cO)$.
				  
	We say that $x$ is \emph{purely valuational} if for every $y \in \bE \langle x \rangle\setminus \bE$ there is $c \in \bE$ such that $\val_{\cO'}(x-c) \notin \val_{\cO'}(\bE)$.
			
	A \emph{principal} extension of models of $T_\convex^-$ is an embedding $\iota : (\bE, \cO) \to (\bE_1, \cO_1)$ of models of $T_\convex^-$ such that $\bE_1 =(\iota \bE) \langle x \rangle_T$.
			
	The principal extension $\iota$ will be said to be \emph{wimg}, \emph{residual} or \emph{purely valuational} if $x$ is respectively wimg, residual or purely valuational. 
\end{definition}

\begin{remark}
    Notice that if $\bE \neq \cO$ (and so $(\bE, \cO)\models T_\convex$), then the definitions of wim, wimg, residual and purely valuational given in Definition~\ref{def:types2} above coincide with the definitions given for them in \cite[Def.~3.8]{freni2024t}.
\end{remark}
		
\begin{remark}\label{rmk:disjointkinds}
	By \cite[Thm.~A]{freni2024t} if $x$ is wimg, then it is not residual, hence every $x$ is either wimg, or residual or purely valuational.
\end{remark}
		
\begin{remark}
	If $\bK \models T$, then every wimg principal extension of $(\bK, \bK) \models T_\convex^-$ is trivial.
\end{remark}
		
\begin{remark}
	It follows from the residue-valuation property of power-bounded theories (\cite[Sec.s~9 and 10]{dries2000field}, \cite[Ch.~12 and 13]{tyne2003t}) that if $T$ is power-bounded with field of exponents $\Lambda$, then
	\begin{enumerate}
		\item for every weakly immediate $x$, $(\bE\langle x \rangle, \cO' \cap \bE \langle x \rangle)$ is an immediate extension of $(\bE, \cO)$; 
		\item for every purely valuational $x$, there is $c \in \bE$ such that $\val_{\cO'}(\bE\langle x \rangle) = \val_{\cO'}(\bE) + \Lambda \val_{\cO'}(x-c)$;
        \item for every residual $x$, there are $c, d \in \bE^{\neq 0}$ such that $\res_{\cO'}(d(x-c))\notin \res_{\cO'}(\bE)$.
	\end{enumerate}
	In particular, for power-bounded $T$, principal extensions of models of $T_\convex^-$ are purely valuational if and only if they expand the value group (so in that context sometimes we will just say \emph{valuational} instead of purely valuational). We will also need the following consequence of the residue-valuation property: if  $S\subseteq \bU$ is such that $\val_{\cO'}(S)$ is $\Lambda$-linearly independent over $\val_{\cO'}(\bE)$, then $\res_{\cO'}(\bE) = \res_{\cO'}(\bE \langle S\rangle)$ and $\val_{\cO'}(\bE \langle S \rangle)=\val_{\cO'}(\bE)+ \Span_{\Lambda}(\val_{\cO'}(S))$.
\end{remark}

Let $\lambda$ be an uncountable cardinal and $(\bE, \cO)\preceq (\bU, \cO') \models T_\convex$ with $\bU$ $\lambda$-saturated. For the next section, it is useful to recall the following definitions and results from \cite{freni2024t}.
  
\begin{itemize}
    \item if $x \in \bU\setminus \bE$ is weakly immediate over $(\bE, \cO)$, the \emph{cofinality} of $x$ (over $(\bE, \cO)$) is the cofinality of the set $\bE^{<x}$ or equivalently of $-\bE^{>x}$ ((\cite[Def.~3.13 and Rmk.~3.14]{freni2024t});
    \item a weakly immediate $x$ is \emph{$\lambda$-bounded} if its cofinality is strictly smaller than $\lambda$ (\cite[Def.~3.13]{freni2024t});
    \item a \emph{$\lambda$-bounded wim-construction} is a sequence $(x_i: i< \mu)$ of elements of some $\bU$ indexed by some ordinal $\mu$, such that for all $j< \mu$, $x_j$ is $\lambda$-bounded weakly immediate over $(\bE_j:=\bE \langle x_i:i<j\rangle, \cO' \cap \bE_j)$ (\cite[Def.~3.15]{freni2024t});
    \item an extension $(\bE_*, \cO_*) \succeq (\bE, \cO)$ is said to be \emph{$\lambda$-bounded wim-constructible} if it is generated by a $\lambda$-bounded wim-construction;
    \item $(\bE,\cO)$ has a unique-up-to-non-unique-isomorphism \emph{$T$-$\lambda$-spherical completion}: that is a $\lambda$-spherically complete $\lambda$-bounded wim-constructible extension $(\bE_\lambda, \cO_\lambda)\succ (\bE,\cO)$ (\cite[Thm.~B]{freni2024t});
    \item the $T$-$\lambda$-spherical completion $(\bE_\lambda, \cO_\lambda)$ elementarily embeds over $(\bE, \cO)$ in every $\lambda$-spherically complete extension of $(\bE, \cO)$ (\cite[Thm.~3.25]{freni2024t});
    \item every $\lambda$-bounded wim-constructible extension embeds over $(\bE, \cO)$ in the $T$-$\lambda$-spherical completion $(\bE_\lambda, \cO_\lambda)$ (\cite[Thm.~3.25]{freni2024t}).
\end{itemize}
		
Also recall that if $T$ is power-bounded, the wim-constructible extensions are precisely the immediate extensions (\cite[Cor.~4.17]{freni2024t}).

\section{The Mourgues-Ressayre constructions revisited} \label{sec:MourguesRessayre}
		
\begin{definition}
	An \emph{rv-sected model} of $T_\convex^-$ is a quadruple $\cE:=(\bE, \cO, \bK, \fM)$ where $(\bE, \cO)$ is a model of $T_\convex^-$, $\bK\preceq_\tame \bE$ is a $T$-section for the residue field and $\fM$ is a $\bK$-monomial group.
	The rv-sected model $\cE$ of $T_{\convex}^-$ will be said to be \emph{above} $(\bE, \cO)\models T_\convex$ (or to be a \emph{rv-secting} of $(\bE, \cO)$); $\bK\preceq_\tame \bE$ and $\fM$ will be referred to respectively as the \emph{residue section} and \emph{monomial group} of $\cE$.
			
	An \emph{embedding of rv-sected models} $\iota: \cE_0 \to \cE_1$ is an embedding $\iota:(\bE_0, \cO_0) \to (\bE_1, \cO_1)$ of the underlying models of $T_\convex^-$ such that $\iota(\bK_0)\subseteq \bK_1$ and $\iota(\fM_0) \subseteq \fM_1$ where $\bK_i$ and $\fM_i$ are respectively the residue section and the monomial group $\cE_i$ for $i\in \{0,1\}$. In such a case we will also say that the rv-secting $\cE_0, \cE_1$ are \emph{compatible} with $\iota: (\bE_0, \cO_0) \to (\bE_1, \cO_1)$.
			
	For $T$ a power-bounded theory, an embedding $\iota: (\bE, \cO) \to (\bE_1, \cO_1)$ of models of $T_\convex^-$, will be said to be \emph{$\lambda$-bounded} if it has the following property: for every rv-sected $\cE$ above $(\bE, \cO)$ there is $\cE_1$ above $(\bE_1, \cO_1)$ such that $\iota: \cE \to \cE_1$ is an embedding of rv-sected models and $(\bE_1, \cO_1)$ is $\lambda$-bounded wim-constructible over $(\dcl_{T}(\bE \fM_1), \CH(\bK))$ where $\fM_1$ is the monomial group of $\cE_1$ and $\bK$ is the residue section of $\bE$. Note that in particular this implies that $\bK$ is also the residue section of $\cE_1$.
\end{definition}

\begin{remark}
	The definition of $\lambda$-bounded does not create ambiguity with the previous definition of $\lambda$-bounded wim-constructible, as an extension of models of $T_\convex^-$ is $\lambda$-bounded wim-constructible if and only if it is both $\lambda$-bounded and wim-constructible as shown by the following Lemma.
\end{remark}
		
\begin{lemma}\label{lem:decomp}
	Let $T$ be power-bounded with field of exponents $\Lambda$. The following are equivalent for an extension $\iota: (\bE, \cO) \to (\bE', \cO')$ of models of $T_\convex^-$:
	\begin{enumerate}
		\item there is sequence $(\bE_i, \cO_i)$ such that $(\bE_i, \cO_i) = \bigcup_{j<i} (\bE_{j+1}, \cO_{j+1})$ and every $(\bE_{i+1}, \cO_{i+1})$ is a principal extension of $(\bE_i, \cO_i)$ which is either $\lambda$-bounded weakly immediate or valuational;
		\item $\iota$ is $\lambda$-bounded.
	\end{enumerate}
	\begin{proof}
		$(1) \Rightarrow (2)$, observe that a composition of $\lambda$-bounded extensions is $\lambda$-bounded, therefore it suffices to show that if $(\bE, \cO)\subseteq (\bE\langle x \rangle, \cO_1)$ is a principal extension which is either $\lambda$-bounded weakly immediate or valuational, then the extension is $\lambda$-bounded. If it is $\lambda$-bounded weakly immediate there is nothing to prove. In the other case there is $c \in \bE$ such that $\val_{\cO_1}(x-c)\notin \val_{\cO_1}\bE$. Notice that if $\fM$ is a monomial group for $(\bE, \cO)$, and $\val_\cO(y) \in \val_{\cO_1}(\bE_1) \setminus \val_{\cO}(\bE)$, then by the rv-property $\fM_1:=\fM y^{\Lambda}$ is a monomial group for $(\bE_1, \cO_1)$ and one easily sees that $\bE\langle x \rangle = \dcl_{T}(\fM_1\bE)$.
				
		$(2) \Rightarrow (1)$ let $\cE$ and $\cE'$ be rv-sectings of $\bE$ and $\cE'$ with residue field $\bK$ and monomial groups $\fM$ and $\fM'$. Suppose that $\cE$ is $\lambda$-bounded wim-constructible above $(\dcl_T(\bE \fM'), \CH(\bK))$. Then a sequence of principal extension as in (1) can be obtained by adjoining first a $\Lambda$-basis of $\fM'$ (thus getting only valuational extensions) and then considering a $\lambda$-bounded wim-construction of $\bE'$ above $(\bE, \CH(\bK))$.
        %TODO:explain better 
	\end{proof}
\end{lemma}
		
\begin{lemma}
	If $T'$ is the theory of an o-minimal field, $T$ is a power-bounded reduct and $(\bE, \cO) \prec (\bE_*, \cO_*)\models T'_\convex$ is $\lambda$-wim constructible for some $\lambda \ge |T'|^+$, then the underlying extension of reducts to $T_\convex$ is $\lambda$-bounded.
	\begin{proof}
		Since $\lambda$-bounded extensions of models of $T$ are closed under composition, it suffices to show that every $(\bE_*, \cO_*):=(\bE \langle x \rangle_{T'}, \cO_x) \succ (\bE, \cO)\models T'_\convex$ with $x$ weakly immediate over $(\bE, \cO)$, is $\lambda$-bounded qua extension of models of $T_\convex^-$.
				
		By Lemma~\ref{lem:decomp}, it suffices to show that $(\bE_*, \cO_*) \succ (\bE, \cO)\models T_\convex$ is a composition of principal extensions that are either $\lambda$-bounded immediate or valuational when regarded as extensions of models $T_\convex^-$.
				
		Since by Remark~\ref{rmk:disjointkinds} $\res(\bE:_*,\cO_*)=\res(\bE,\cO)$, and $T$ is power-bounded, given a $\dcl_T$-basis $(x_i: i< \mu)$ of $\bE_*$ over $\bE$, for every $i$, $x_i$ is either valuational or weakly immediate over $\bE_i:=\bE\langle x_j: j< i\rangle$. 
		So it suffices to show that whenever $\tp(x_i/\bE_i)$ is weakly immediate, it must have cofinality $<\lambda$.
				
		Note that $\bE^{<x_i}$ has cofinality $< \lambda$ because by hypothesis $\bE^{<x}$ has cofinality $<\lambda$, and that we must have $\mu \le |T'|$, therefore $\bE_i^{< x_i}\setminus \bE^{<x_i}$ if non-empty has cofinality at most $|T|+|i|\le |T'|$. It follows that $\bE_i^{<x_i}$ has cofinality $< \lambda$.
	\end{proof}
\end{lemma}
		
\begin{definition}
	Suppose $\bK_L\models T$ is power-bounded with field of exponents $\Lambda$ and that $(\bK_L, \cT)$ is serial. For every multiplicatively written ordered $\Lambda$-vector space $\fM$, we will denote by $[\bK\Lhp \fM \Rhp_\lambda^\cT]$ the rv-sected model of $T_\convex$ $\big(\bK\Lhp \fM\Rhp_\lambda^\cT, \CH(\bK), \bK , \fM \big)$.
			
	Given a rv-sected model $\cE$ of $T_\convex$ with residue section ($T$-isomorphic to) $\bK$, a \emph{t.c.\ embedding} is an embedding of rv-sected models of $T_\convex^-$, $\iota : \cE \to [\bK \Lhp \fN \Rhp_\lambda^\cT]$ for some $\lambda$ such that the image of $\iota$ is a truncation-closed subfield of $[\bK \Lhp \fN \Rhp_\lambda^\cT]$.
			
	Let $(\bE, \cO) \preceq (\bE_1, \cO_1)$ and $\cE$ be an rv-sected model above $(\bE, \cO)$. A t.c.\ embedding $\iota: \cE \to [\bK\Lhp \fM \Rhp_\lambda]_{\cT}$ will be said to be \emph{$\val$-maximal within $(\bE_1, \cO_1)$} if every proper extension of $\iota$ along some $T_\convex$-extension $j: (\bE, \cO) \preceq (\bE_2, \cO_2)$ factoring the inclusion $(\bE, \cO) \preceq (\bE_1, \cO_1)$ is such that $\val_{\cO_2}(\bE) \neq \val_{\cO_2}(\bE_2)$.
\end{definition}
		
\begin{context}\label{cont:K_LT_serial}
    Throughout the rest of the section $(\bK_L, \cT)$ will be a serial power-bounded structure with field of exponents $\Lambda$, $T$ will be the theory of $(\bK_L, \cT)$ and $T'$ will be the theory of an o-minimal expansion $\bK_{L'}$ of $\bK_L$.
\end{context}
 
\begin{lemma}\label{lem:extension}
	Assume $\eta : \cE_0 \to \cE_1$ is a $\lambda$-bounded extension of rv-sected models of $T_\convex^-$ with residue section $\bK$ and value groups $\fM$ and $\fM_1$ respectively.
			
	If $\iota : \cE \to [\bK \Lhp \fN \Rhp_\lambda^\cT]$ is a $\lambda$-bounded truncation-closed embedding, and $j:\fM_1 \to \fN$ is an (ordered $\Lambda$-linear) extension of $\iota|_{\fM}$, then $\iota$ extends along $\eta$ to a $\lambda$-bounded truncation-closed embedding $\iota_1 : \cE_1 \to [\bK \Lhp \fN \Rhp_\lambda^\cT]$ such that $\iota_1|_{\fM_1}=j$.
	\begin{proof}%TODO: change using the definition o lambda bounded and embed first the definable closure of the value group.
        Let $\cE:=(\bE, \cO, \bK,\fM)$ and $\cE_1:=(\bE_1, \cO_1, \bK, \fM_1)$.
		It suffices to prove the statement in the case $\bE_1 = \bE \langle x \rangle$ is a principal extension. By Lemma~\ref{lem:decomp} it is either valuational or $\lambda$-bounded immediate.
				
		If $x$ is valuational, then for some $c \in \bE$, $\val(x-c) \notin \val \bE$. Pick $\fn \in \fN$ with $\tp(\fn/\bE)=\tp(x-c/\bE)$. By Proposition~\ref{prop:tc-ppty}, $(\iota \bE)\langle \fn \rangle_T$ is truncation-closed and we are done.
				
		If $x$ is $\lambda$-bounded immediate and $(x_i)_{i < \mu}\in \bE^\mu$ is a p.c.-sequence for $x$, then we can set $\iota(x)$ to be the only element of $(\bK \Lhp \fN \Rhp_\lambda, \CH(\bK))$ which is a pseudolimit for $(x_i)_{i<\mu}$ and is such that $\iota(x)|\fm \neq \iota(x)\Rightarrow \iota(x)|\fm \in \bE$. Again by Proposition~\ref{prop:tc-ppty} $(\iota \bE)\langle \fn \rangle_T$ is truncation-closed and we are done.
	\end{proof}
\end{lemma}

We are now ready to prove the first part of Theorem~\ref{introthm:C} of the introduction.

\begin{theorem}\label{thm:MRanalogue}
    As in Context~\ref{cont:K_LT_serial}, let $\bK\models T'$, $\bK_L$ be a power bounded reduct of $\bK$ such that $(\bK_L, \cT)$ is serial, and $T=\Th(\bK_L)$. Let $\bK \preceq_\tame \bE$, and $(\bE_\lambda, \cO_\lambda)$ be the $T'$-$\lambda$-spherical completion of $(\bE, \CH(\bK))$ for $\lambda$ large enough.
            
    Then there is an $L(\bK)$-isomorphism $\eta: \bE_\lambda \to \bK\Lhp \fN \Rhp^\cT_\lambda$ such that $\eta(\bE)$ is truncation-closed.
	\begin{proof}
		Let $\varepsilon_\lambda: (\bE, \cO) \to (\bE_\lambda, \cO_\lambda)$ be an elementary embedding of $(\bE, \cO)$ into its $\lambda$-bounded spherical completion. We can assume $\varepsilon_\lambda$ is an inclusion and $\bK\preceq_\tame \bE_\lambda$.
				
		Let $\cE:=(\bE, \cO, \bK, \fM)$ and $\cE_\lambda:=(\bE,\cO, \bK, \fN)$ be rv-sections above the $T_\convex^-$ reducts of $(\bE,\cO)$ and $(\bE_\lambda, \cO_\lambda)$, compatible with $\varepsilon_\lambda$.
				
		By Lemma~\ref{lem:extension}, if $\lambda$ is large enough there is a $\lambda$-bounded truncation-closed embedding $\iota':\cE \to [\bK\Lhp \fM \Rhp_\lambda]_{\cT}$. 
				 
		Let $\iota''$ be the composition of $\iota'$ with the natural inclusion
		\[[\bK\Lhp \fM \Rhp_\lambda^\cT] \hookrightarrow [\bK\Lhp \fN \Rhp_\lambda^\cT].\]
		Observe that $\iota''$ also has a truncation-closed image.
		Again by Lemma~\ref{lem:extension}, $\iota''$ can be extended along $\varepsilon_\lambda$ to a $\lambda$-bounded truncation-closed embedding $\eta : (\bE_\lambda, \cO_\lambda) \to [\bK\Lhp \fN \Rhp_\lambda^\cT]$. Now observe that since $(\bE_\lambda, \cO_\lambda)$ is $\lambda$-spherically complete, $\eta$ must be an isomorphism. The image of $\eta\circ \varepsilon_\lambda$ is the image of $\iota''$, so it is truncation-closed.
	\end{proof}
\end{theorem}

The remainder of this section is dedicated to the proof of the remainder of Theorem~\ref{introthm:C}. For this we need to introduce some new definitions concerning compatibility of t.c.\ embeddings with an exponential function.

\begin{definition}
	Let $(\bE, \cO)\models T_\convex^-$. We say that an ordered exponential $\exp: \bE \to \bE^{>0}$ is \emph{weakly $\cO$-compatible} if:
    \begin{enumerate}
        \item $\exp(\cO) \subseteq \cO$;
        \item for all $t\succ 1$, $1 \prec \log|t| \prec t$, where $\log:=\exp^{-1}$ is the compositional inverse of $\exp$;
        \item there is some interval $I$ containing $\co$ such that $\exp|_I$ is $T$-definable.
    \end{enumerate}
			
	In the presence of a weakly $\cO$-compatible exponential, we will say that a t.c.\ embedding $\iota: \cE \to [\bK\Lhp \fN \Rhp_\lambda^{\cT}]$, with $\cE=(\bE, \cO, \bK, \fM)$ is
	\begin{itemize}
		\item \emph{dyadic} (after Ressayre \cite{Ressayre1993Integer}) if $\iota : \cE \to [\bK \Lhp \fN \Rhp_\lambda^\cT]$ is such that $\iota \log \fM \subseteq \bK\Lhp \fN^{>1} \Rhp_{\lambda}$;
		\item \emph{T4} (after Schmeling \cite{schmeling2001corps}) if for every sequence of monomials $(\fm_n)_{n \in \bN}$ with $\iota \fm_{n+1} \in \Supp(\iota \log \fm_n)$ for all $n \in \bN$, there is $N\in \bN$ such that for all $n\ge N$, $\iota \log \fm_n = c_n \pm \iota \fm_{n+1}$ for some $c_n$ with $\Supp c_n>\iota \fm_{n+1}$.
	\end{itemize}
    We will call \emph{R.S.\ embedding} (for Ressayre and Schmeling) a t.c.\ embedding that is both dyadic and T4.
\end{definition}

\begin{remark}\label{rmk:D&T4reminder}
    Recall the conditions \ref{Intro:D} and \ref{Intro:T4} of the introduction.
    Notice that a t.c.\ embedding $\iota$ whose image has the form $\bK \Lhp \fM \Rhp_B \subseteq \bK \Lhp \fM \Rhp$ is dyadic if and only if its image with the induced exponential satisfies \ref{Intro:D}. Similarly $\iota$ is T4 if and only if its image with the induced exponential satisfies \ref{Intro:T4}.
\end{remark}

\begin{lemma}\label{lem:exp_extension}
	Let $j: (\bE, \cO) \preceq (\bE_*, \cO_*)$ be a $\lambda$-bounded extension of models of $T_\convex^-$ expanded with a weakly $\cO$-compatible (resp.\ $\cO_*$-compatible) exponential.
	Let $\fN$ be a multiplicative copy of the value group of $(\bE_*, \cO_*)$. Assume that $\cE$ is an rv-secting of $(\bE, \cO)$ and $\iota : \cE \to [\bK \Lhp \fN \Rhp_\lambda^\cT]$ is a R.S.\ embedding maximal within $(\bE_*,\cO_*)$. Then there is an rv-secting $\cE_*$ of $(\bE_*, \cO_*)$ extending $\cE$, such that $\iota$ extends to a R.S.\ embedding $\iota_*: \cE_* \to  [\bK \Lhp \fN \Rhp_\lambda^\cT]$.
	\begin{proof}
		An adaptation of the argument in Ressayre's \cite{Ressayre1993Integer} (see also \cite{DAquino2012Real}). Note that it suffices to show the statement with $\fN$ a sufficiently saturated multiplicatively written ordered $\Lambda$-vector space.
				
		Let $\fM$ be the monomial group of the fixed rv-secting $\cE$ on $(\bE, \cO)$. By hypothesis $\iota (\fM) \subseteq \fN$ and $\log \iota (\fM) \subseteq \bK\Lhp \fN^{>1} \Rhp$. %One can assume that $\iota$ is already $\val$-maximal within $(\bE_1, \cO_1)$. 
		Pick an $x \in \bE_1 \setminus \bE$. It suffices to show that we can extend $\iota$ to a R.S.\ embedding of some extension $\cE_{\omega}$ of $\cE$ above some intermediate $(\bE_{\omega}, \cO_{\omega})$, $(\bE, \cO) \preceq (\bE_{\omega}, \cO_{\omega}) \preceq (\bE_*, \cO_*)$ which is closed under exponentials and logarithms.
				 
		Consider $\bE \langle x \rangle_T$. Since $\iota$ is $\val$-maximal within $(\bE_*, \cO_*)$, without loss of generality $\val(\bE \langle x \rangle_T)\neq \val\bE$, for otherwise by Lemma~\ref{lem:extension}, this would contradict maximality. Thus without loss of generality we can assume that $\bE \langle x \rangle_T= \bE \langle y \rangle_T$ where $\val(y) \notin \val(\bE)$ and $y \notin \cO_*$. 

		We now inductively build a sequence $(y_i)_{i<\omega} \in \bE_*$ such that $y_0=y$ and for every $i$, $y_i \notin \cO_*$, $\val_{\cO_*}(y_i) \notin \val_{\cO_*}\bE$, $\log|y_i| - y_{i+1} \in \bE$ and $\iota(\log|y_i|-y_{i+1}) \in \bK\Lhp\fN^{>1}\Rhp$.
				
		Given $(y_{j})_{j<i+1}$, observe that $(\log|y_i| +\cO_*) \cap \bE = \emptyset$ because $y_i \notin \cO_*$ and $\val_{\cO_*}(y_i) \notin \val_{\cO_*}(\bE)$. Note that since $\log|y_i|$ is either weakly immediate $\lambda$-bounded or valuational over $(\bE, \cO)$, we have that either $\val_{\cO_*}(\log|y_i|-c_i) \notin \val_{\cO_*}\bE$ for some $c \in \bE$ or $\log|y_i|$ is a pseudo-limit of a p.c. sequence. However since $\iota$ was $\val$-maximal within $(\bE_*, \cO_*)$, this second option would imply $\log|y_i| \in \bE$, but this would contradict that $\bE$ is $\exp$-closed. Thus there is $c_i\in \bE$ such that $\val_{\cO_*}(\log|y_i|-c_i) \notin \val_{\cO_*}\bE$.
				
		Since $(\log|y_i| +\cO_*) \cap \bE = \emptyset$, it follows that $\log|y_i|-c_i \notin \cO_*$. We can thus change the choice of $c_i$ so that $c_i \in \bK\Lhp \fN^{>1} \Rhp$, maintaining the property $\val_{\cO_*}(\log|y_i|-c_i) \notin \val_{\cO_*}\bE$.
				
		Setting $y_{i+1}:=\log|y_i|-c_i$ we have an extension of the sequence with the required properties. We claim this implies that for all $i$ we also have $\val_{\cO_*}(y_i) \notin \val_{\cO_*}\bE + \sum_{j<i} \Lambda \val_{\cO_*}(y_j)$. In fact if not, then we would have for some $(\beta_j)_{j<i} \in \Lambda^{i}$
		\[\log |y_i| - \sum_{j<i} \beta_j \log |y_j|  = y_{i+1} + c_i - \sum_{j<i} \beta_j c_{j} + \sum_{j<i} y_{j+1} \beta_j \in \bE+\cO\]
		whence $y_{i+1} - \sum_{j<i} y_{j+1} \beta_j \in \bE + \cO$ which is absurd because the $\val(y_j)$ are negative and pairwise distinct.
		
		Let $\fM_1:=\fM \cdot \bigcup_{i<\omega}y_0^\Lambda \cdot \ldots \cdot y_i^\Lambda$, $\bE_1'=\bE \langle y_i: i<\omega \rangle_T$ and $\cE_1'=(\bE_1', \cO_1', \bK, \fM_1)$. Observe $\cE_1'$ extends $\cE$ and that $\bE_1'$ is $\log$-closed by \cite[Lem.~5.18]{freni2024t}. 

		Let $h: \fM_1 \to \fN$ be an extension of $\iota|_\fM$ to $\fM_1$. By Lemma~\ref{lem:extension} $\iota: \cE \to (\bK \Lhp \fN \Rhp_\lambda,\CH(\bK))$ extends to a $\iota_1': \cE_1' \to (\bK \Lhp \fN \Rhp_\lambda, \CH(\bK))$ such that $\iota_1'|_{\fM_1'}=h$.

		Again by Lemma~\ref{lem:extension} $\iota_1'$ extends to a $\iota_1: \cE_1 \to (\bK \Lhp \fN\Rhp_\lambda, \CH(\bK))$ that is $\val$-maximal within $(\bE_*, \cO_*)$. The structure $(\bE_1, \cO_1)$ underlying $\cE_1$ is still closed under logarithms by \cite[Lem.~5.16]{freni2024t}. %In fact for all $y \in \bE_1$, there are $n\in \omega$,  $(\beta_i)_{i<n}$, $c \in \bE$ and $\varepsilon \in \bE_1 \cap \co$, such that $y = c\prod_{i<n} |y_i|^{\beta_i}(1+\varepsilon)$, whence $\log |y| = \sum_{i<n} \beta_{i}(c_i + y_{i+1})+ \log |c| + \log(1+\varepsilon)\in \bE_1$ and $\log(1+\varepsilon) \in \bE_1$ because $\log$ is $T$-definable aorund $1+\co$.
		Also observe that by construction:
		\begin{itemize}
			\item $\iota_1 \log \fM_1 \subseteq \bK \Lhp \fN^{>1} \Rhp$;
			\item $\Supp (\iota_1 \log|y_i|) \subseteq \iota\fM \cup \{\iota_1 y_{i+1}\}$.
		\end{itemize}
				
		Now inductively define a sequence $\iota_n: \cE_n \to \bK \Lhp \fN \Rhp$ such that $\iota_n$ is $\val$-maximal within $(\bE_*, \cO_*)$, $(\bE_n, \cO_n)$ is closed under logarithms, $\iota_n (\log \fM_n) \subseteq \bK \Lhp \fN ^{>1}\Rhp$ and
		$\fM_{n+1}=\exp(\iota_{n}^{-1}(\bK\Lhp\fN^{>1} \Rhp))$.
		The base case is given by the $(\bE_1, \cO_1)$ constructed above and the inductive step is possible by Lemma~\ref{lem:extension} because $\fM_{n} \subseteq \fM_{n+1}$. 
		Setting $\iota_\omega= \bigcup_n \iota_n$ we get the desired extension.
	\end{proof}
\end{lemma}
		
\begin{corollary}\label{cor:exp_extension}
	Suppose $(\bE, \cO)$ is a $\lambda$-bounded extension of $(\bK_{L'},\bK_{L'}) \models (T')_\convex^-$ with $T'$ exponential and $\fM$ is a multiplicative copy of the value group of $(\bE, \cO)$. Then there is some $\cE$ above $(\bE, \cO)$ and a R.S. embedding $\cE \to [\bK\Lhp \fM \Rhp_\lambda^\cT]$.
\end{corollary}
		
We have said enough to deduce the remainder of Theorem~\ref{introthm:C} of the introduction.

\begin{theorem}\label{thm:final}
    As in Context~\ref{cont:K_LT_serial}, let $\bK\models T'$, $\bK_L$ be a power bounded reduct of $\bK$ such that $(\bK_L, \cT)$ is serial, and $T=\Th(\bK_L)$.
    Let $\bK \preceq_\tame \bE$, and $(\bE_\lambda, \cO_\lambda)$ be the $T'$-$\lambda$-spherical completion of $(\bE, \CH(\bK))$ for some large enough $\lambda$. Suppose $T'$ defines an exponential. Then there is a $L(\bK)$-isomorphism $\eta: \bE_\lambda \to \bK\Lhp \fN \Rhp_\lambda^\cT$ for some $\fN$, such that $\eta(\bE)$ is truncation-closed and the expansion $\eta \circ \exp \circ \eta^{-1}$ satisfies \ref{Intro:D} and \ref{Intro:T4} in the introduction.
	\begin{proof}
		Let $\varepsilon_\lambda: (\bE, \cO) \to (\bE_\lambda, \cO_\lambda)$ be an elementary embedding of $(\bE, \cO)$ into its $T$-$\lambda$-spherical completion. Let $\fM$ be a section of the value group of $(\bE, \cO)$ and let $\cE=(\bE, \cO, \bK, \fM)$. By Lemma~\ref{lem:exp_extension}, if $\lambda$ is large enough there is a R.S.\ embedding $\iota': \cE \to [\bK\Lhp \fM \Rhp_\lambda^\cT]$.
		By Lemma~\ref{lem:exp_extension}, $\iota'$ can be extended along $\varepsilon_\lambda$ to a R.S.\ embedding $\eta : \cE_\lambda \to [\bK\Lhp \fN \Rhp_\lambda^\cT]$, for some $\cE_\lambda$ above $(\bE_\lambda, \cO_\lambda)$ extending $\cE$. Now observe that since $(\bE_\lambda, \cO_\lambda)$ is $\lambda$-spherically complete $\eta$ must be an isomorphism. The map $\eta \circ \varepsilon_\lambda$ is then truncation-closed and $\eta \circ \exp \circ \eta^{-1}$ is an exponential on $\big(\bK\Lhp \fN \Rhp_\lambda^{\cT}, \CH(\bK)\big)$ satisfying \ref{Intro:D} and \ref{Intro:T4} in the introduction (cf Remark~\ref{rmk:D&T4reminder}).
	\end{proof}
\end{theorem}

\subsection{Application to intial embeddings in the surreal numbers}

Finally as suggested by Mantova, Theorem~\ref{thm:final} can be combined with \cite[Thm.~8.1]{ehrlich2021surreal} to answer the open question in \cite{ehrlich2021surreal}, showing that for all complete $T_{\RCF}\subseteq T\subseteq T_{an,\exp}$ (and thus in particular for $T_{\exp}$), every model of $T$ has an elementary initial embedding into $\No$.

\begin{corollary}\label{cor:initial}
	Let $\bR_L$ be a reduct of $\bR_{an}$ defining restricted exponentiation. Then every model of $\Th(\bR_{L,\exp})$ has an elementary truncation-closed embedding in the field of surreal numbers $\No$.
	\begin{proof}
		Let $(\bR_{An},\cT)$ be the standard serial structure on the expansion by definition $\bR_{An}:=\big(\bR_{an}, (\sqrt[n]{-})_{n>0}, (-)^{-1}\big)$ of $\bR_{an}$. Without loss of generality we can assume that $(\bR_L, \cT|):=(\bR_L, \cT|_L)$ is serial.
        
		Let $T=Th(\bR_L)$ and $\bE\models T$. Then $(\bE, \CH(\bZ))\models T_\convex^-$ and we can consider an elementary residue section $\bK_L\models T$ of $(\bE, \CH(\bZ))$. Note that $\bK_L$ is serial, thus in particular the $T_\convex^-$-reduct of $(\bE, \CH(\bR))$ is a $\lambda$-bounded extension of $(\bK_L, \bK_L) \models T_\convex^-$ for every large enough $\lambda$. Let $(\bE_\lambda, \cO_\lambda)\succeq (\bE, \CH(\bZ))$ and 
		\[\eta:  (\bE_\lambda, \cO_\lambda) \to \big(\bK\Lhp \fN \Rhp_\lambda^{\cT|}, \CH(\bZ)\big) \]
		be as in Theorem~\ref{thm:final}.

        Observe that since $(\bK_L, \cT|)$ is serial $\big(\bK\Lhp \fN \Rhp_\lambda^{\cT|}, \CH(\bZ)\big)$ is an $L$-elementary substructure of $(\bK\Lhp \fN \Rhp^{\cT|}, \CH(\bZ))$. We can extend $\eta \log \eta^{-1}$ to a (not necessarily surjective) logarithm on the $L$-elementary extension $\big(\bR\Lhp \fN \Rhp^{\cT|}, \CH(\bZ)\big)\succ (\bK\Lhp \fN \Rhp^{\cT|}, \CH(\bZ))$ making such extension into a transserial Hahn field in the sense of \cite[Def.~6.1 and 6.2]{ehrlich2021surreal}: given $r\fn (1+\varepsilon) \in \bK \Lhp \fN \Rhp^{>0}$ we set its logarithm to be $\eta\log \eta^{-1} (\fn) +\log r+ \sum_{n\in \bN} \varepsilon^{n+1}(-1)^n/(n+1)$ and a routine check shows it extends $\eta \log \eta^{-1}$.

        Now $\eta (\bE)$ is truncation closed in $\bR\Lhp \fN \Rhp$ and $\bK \subseteq \eta(\bE)\subseteq \bK \Lhp \fN \Rhp$, thus $\fM\coloneqq \eta(\bE) \cap \fN$ is a monomial group for $\eta(\bE)$; furthermore since $\eta(\bE)$ is $\log$-closed $\bR \Lhp \fM \Rhp$ is also $\log$-closed and $\bR \Lhp \fM \Rhp$ is a transserial Hahn field in which $\eta(\bE)$ sits as truncation closed logarithmic subfield which is cross-sectional in the sense of \cite[Sec.~4.1]{ehrlich2021surreal}. Now by \cite[Thm~8.1]{ehrlich2021surreal} there is an initial transserial embedding $\iota : \bR \Lhp \fM \Rhp \to \No$ and by \cite[Prop.~5.1 and subsequent sentence]{ehrlich2021surreal}, the image of $\eta(\bE)$ by it is initial as well.
        Finally, since $(\bR_L, \cT|)$ is serial, $\iota$ is $L$-elementary and $\eta(\bE)$ is an $L$-elementary substructure of $\bR \Lhp \fM \Rhp$. Since then $\iota|_{\eta(\bE)}$ is a $L_{\log}$-embedding and by \cite[Thm.~3.2]{dries2002minimal}, $T=Th(\bR_{L, \exp})$ is still model complete, $\iota|_{\eta(\bE)}$ is also $(L, \exp)$-elementary.
	\end{proof}
\end{corollary}

\bibliographystyle{abbrvnat}
\bibliography{Res}

@article {denef1988p,
    AUTHOR = {Denef, J. and van den Dries, L.},
     TITLE = {{$p$}-adic and real subanalytic sets},
   JOURNAL = {Ann. of Math. (2)},
  FJOURNAL = {Annals of Mathematics. Second Series},
    VOLUME = {128},
      YEAR = {1988},
    NUMBER = {1},
     PAGES = {79--138},
      ISSN = {0003-486X,1939-8980},
   MRCLASS = {03C10 (03C60 14G20 14G30 32B20)},
  MRNUMBER = {951508},
MRREVIEWER = {Max\ A.\ Dickmann},
       DOI = {10.2307/1971463},
       URL = {https://doi.org/10.2307/1971463},
}

@article{gabrielov1968projections,
  title={Projections of semi-analytic sets},
  author={Gabrielov, Andrei M},
  journal={Functional Analysis and its applications},
  volume={2},
  number={4},
  pages={282--291},
  year={1968},
  publisher={Springer}
}

@article{dries1994elementary,
 AUTHOR = {{\noopsort{dries}van den Dries}, Lou and Macintyre, Angus and Marker, David},
     TITLE = {The elementary theory of restricted analytic fields with
              exponentiation},
   JOURNAL = {Ann. of Math. (2)},
  FJOURNAL = {Annals of Mathematics. Second Series},
    VOLUME = {140},
      YEAR = {1994},
    NUMBER = {1},
     PAGES = {183--205},
      ISSN = {0003-486X,1939-8980},
   MRCLASS = {12L12 (03C10 03C62)},
  MRNUMBER = {1289495},
MRREVIEWER = {Thanases\ Pheidas},
       DOI = {10.2307/2118545},
       URL = {https://doi.org/10.2307/2118545},
}

@article{neumann1949ordered,
    AUTHOR = {Neumann, B. H.},
     TITLE = {On ordered division rings},
   JOURNAL = {Trans. Amer. Math. Soc.},
  FJOURNAL = {Transactions of the American Mathematical Society},
    VOLUME = {66},
      YEAR = {1949},
     PAGES = {202--252},
      ISSN = {0002-9947,1088-6850},
   MRCLASS = {09.1X},
  MRNUMBER = {32593},
MRREVIEWER = {R.\ Moufang},
       DOI = {10.2307/1990552},
       URL = {https://doi.org/10.2307/1990552},
}

@article {kuhlmann1997exponentiation,
    AUTHOR = {Kuhlmann, Franz-Viktor and Kuhlmann, Salma and Shelah,
              Saharon},
     TITLE = {Exponentiation in power series fields},
   JOURNAL = {Proc. Amer. Math. Soc.},
  FJOURNAL = {Proceedings of the American Mathematical Society},
    VOLUME = {125},
      YEAR = {1997},
    NUMBER = {11},
     PAGES = {3177--3183},
      ISSN = {0002-9939,1088-6826},
   MRCLASS = {12J15 (06F20 12J25 13J05)},
  MRNUMBER = {1402868},
MRREVIEWER = {Joachim\ Gr\"ater},
       DOI = {10.1090/S0002-9939-97-03964-6},
       URL = {https://doi.org/10.1090/S0002-9939-97-03964-6},
}

@book{hodges1993MT,
    place={Cambridge},
    series={Encyclopedia of Mathematics and its Applications},
    title={Model Theory},
    DOI={10.1017/CBO9780511551574},
    publisher={Cambridge University Press},
    author={Hodges, Wilfrid},
    year={1993},
    collection={Encyclopedia of Mathematics and its Applications}
}

@article{mourgues1993every,
    AUTHOR = {Mourgues, M.-H. and Ressayre, J. P.},
     TITLE = {Every real closed field has an integer part},
   JOURNAL = {J. Symbolic Logic},
  FJOURNAL = {The Journal of Symbolic Logic},
    VOLUME = {58},
      YEAR = {1993},
    NUMBER = {2},
     PAGES = {641--647},
      ISSN = {0022-4812,1943-5886},
   MRCLASS = {03C60 (12J15 16W80)},
  MRNUMBER = {1233929},
MRREVIEWER = {S.\ R.\ Kogalovski\u{\i}},
       DOI = {10.2307/2275224},
       URL = {https://doi.org/10.2307/2275224},
}

@article {daquino2012real,
    AUTHOR = {D'Aquino, Paola and Knight, Julia F. and Kuhlmann, Salma and
              Lange, Karen},
     TITLE = {Real closed exponential fields},
   JOURNAL = {Fund. Math.},
  FJOURNAL = {Fundamenta Mathematicae},
    VOLUME = {219},
      YEAR = {2012},
    NUMBER = {2},
     PAGES = {163--190},
      ISSN = {0016-2736,1730-6329},
   MRCLASS = {03C57 (03C60 03C70)},
  MRNUMBER = {2993471},
MRREVIEWER = {Giuseppina\ Terzo},
       DOI = {10.4064/fm219-2-6},
       URL = {https://doi.org/10.4064/fm219-2-6},
}

@incollection {ressayre1993integer,
    AUTHOR = {Ressayre, J.-P.},
     TITLE = {Integer parts of real closed exponential fields (extended
              abstract)},
 BOOKTITLE = {Arithmetic, proof theory, and computational complexity
              ({P}rague, 1991)},
    SERIES = {Oxford Logic Guides},
    VOLUME = {23},
     PAGES = {278--288},
 PUBLISHER = {Oxford Univ. Press, New York},
      YEAR = {1993},
      ISBN = {0-19-853690-9},
   MRCLASS = {03C60 (12J15)},
  MRNUMBER = {1236467},
}

@article {kuhlmann2005kappa,
    AUTHOR = {Kuhlmann, Salma and Shelah, Saharon},
     TITLE = {{$\kappa$}-bounded exponential-logarithmic power series
              fields},
   JOURNAL = {Ann. Pure Appl. Logic},
  FJOURNAL = {Annals of Pure and Applied Logic},
    VOLUME = {136},
      YEAR = {2005},
    NUMBER = {3},
     PAGES = {284--296},
      ISSN = {0168-0072,1873-2461},
   MRCLASS = {03C60 (06A05)},
  MRNUMBER = {2169687},
MRREVIEWER = {Martin\ Weese},
       DOI = {10.1016/j.apal.2005.04.001},
       URL = {https://doi.org/10.1016/j.apal.2005.04.001},
}

@article{hoeven2001operators,
AUTHOR = {{\noopsort{hoeven}van der Hoeven}, Joris},
     TITLE = {Operators on generalized power series},
   JOURNAL = {Illinois J. Math.},
  FJOURNAL = {Illinois Journal of Mathematics},
    VOLUME = {45},
      YEAR = {2001},
    NUMBER = {4},
     PAGES = {1161--1190},
      ISSN = {0019-2082,1945-6581},
   MRCLASS = {16W60},
  MRNUMBER = {1894891},
MRREVIEWER = {N.\ Sankaran},
       URL = {http://projecteuclid.org/euclid.ijm/1258138061},
}

@article{dries2001logarithmic,
 AUTHOR = {{\noopsort{dries}van den Dries}, Lou and Macintyre, Angus and Marker, David},
     TITLE = {Logarithmic-exponential series},
 BOOKTITLE = {Proceedings of the {I}nternational {C}onference ``{A}nalyse \&
              {L}ogique'' ({M}ons, 1997)},
   JOURNAL = {Ann. Pure Appl. Logic},
  FJOURNAL = {Annals of Pure and Applied Logic},
    VOLUME = {111},
      YEAR = {2001},
    NUMBER = {1-2},
     PAGES = {61--113},
      ISSN = {0168-0072,1873-2461},
   MRCLASS = {12J15 (03H05 12H05 26E05)},
  MRNUMBER = {1848569},
MRREVIEWER = {Niels\ Schwartz},
       DOI = {10.1016/S0168-0072(01)00035-5},
       URL = {https://doi.org/10.1016/S0168-0072(01)00035-5},
}

@article {higman1952ordering,
    AUTHOR = {Higman, Graham},
     TITLE = {Ordering by divisibility in abstract algebras},
   JOURNAL = {Proc. London Math. Soc. (3)},
  FJOURNAL = {Proceedings of the London Mathematical Society. Third Series},
    VOLUME = {2},
      YEAR = {1952},
     PAGES = {326--336},
      ISSN = {0024-6115,1460-244X},
   MRCLASS = {09.1X},
  MRNUMBER = {49867},
MRREVIEWER = {D.\ Zelinsky},
       DOI = {10.1112/plms/s3-2.1.326},
       URL = {https://doi.org/10.1112/plms/s3-2.1.326},
}

@article{berarducci2021value,
    AUTHOR = {Berarducci, Alessandro and Freni, Pietro},
     TITLE = {On the value group of the transseries},
   JOURNAL = {Pacific J. Math.},
  FJOURNAL = {Pacific Journal of Mathematics},
    VOLUME = {312},
      YEAR = {2021},
    NUMBER = {2},
     PAGES = {335--354},
      ISSN = {0030-8730,1945-5844},
   MRCLASS = {16W60 (03C64 12L12)},
  MRNUMBER = {4305776},
MRREVIEWER = {Lothar\ Sebastian\ Krapp},
       DOI = {10.2140/pjm.2021.312.335},
       URL = {https://doi.org/10.2140/pjm.2021.312.335},
}

@article{krapp2022rayner,
 AUTHOR = {Krapp, Lothar Sebastian and Kuhlmann, Salma and Serra,
              Michele},
     TITLE = {On {R}ayner structures},
   JOURNAL = {Comm. Algebra},
  FJOURNAL = {Communications in Algebra},
    VOLUME = {50},
      YEAR = {2022},
    NUMBER = {3},
     PAGES = {940--948},
      ISSN = {0092-7872,1532-4125},
   MRCLASS = {13J05 (06F20 12J20 16W60 20K01)},
  MRNUMBER = {4379648},
MRREVIEWER = {G\'{e}rard\ Leloup},
       DOI = {10.1080/00927872.2021.1976789},
       URL = {https://doi.org/10.1080/00927872.2021.1976789},
}

@article{dries1997t,
AUTHOR = {{\noopsort{dries}van den Dries}, Lou},
     TITLE = {{$T$}-convexity and tame extensions. {II}},
   JOURNAL = {J. Symbolic Logic},
  FJOURNAL = {The Journal of Symbolic Logic},
    VOLUME = {62},
      YEAR = {1997},
    NUMBER = {1},
     PAGES = {14--34},
      ISSN = {0022-4812,1943-5886},
   MRCLASS = {03C60 (03C10 03C35 12L12)},
  MRNUMBER = {1450511},
MRREVIEWER = {O.\ V.\ Belegradek},
       DOI = {10.2307/2275729},
       URL = {https://doi.org/10.2307/2275729},
}

@article{dries1995t,
 AUTHOR = {{\noopsort{dries}van den Dries}, Lou and Lewenberg, Adam H.},
     TITLE = {{$T$}-convexity and tame extensions},
   JOURNAL = {J. Symbolic Logic},
  FJOURNAL = {The Journal of Symbolic Logic},
    VOLUME = {60},
      YEAR = {1995},
    NUMBER = {1},
     PAGES = {74--102},
      ISSN = {0022-4812,1943-5886},
   MRCLASS = {03C60 (03C10 03C35 12L12)},
  MRNUMBER = {1324502},
MRREVIEWER = {O.\ V.\ Belegradek},
       DOI = {10.2307/2275510},
       URL = {https://doi.org/10.2307/2275510},
}

@incollection {miller1993growth,
    AUTHOR = {Miller, Chris},
     TITLE = {A growth dichotomy for o-minimal expansions of ordered fields},
 BOOKTITLE = {Logic: from foundations to applications ({S}taffordshire,
              1993)},
    SERIES = {Oxford Sci. Publ.},
     PAGES = {385--399},
 PUBLISHER = {Oxford Univ. Press, New York},
      YEAR = {1996},
      ISBN = {0-19-853862-6},
   MRCLASS = {03C60 (03C50)},
  MRNUMBER = {1428013},
MRREVIEWER = {G.\ Cherlin},
}

@book{tyne2003t,
 AUTHOR = {Tyne, James Michael},
     TITLE = {T-levels and {T}-convexity},
      NOTE = {Thesis (Ph.D.)--University of Illinois at Urbana-Champaign},
 PUBLISHER = {ProQuest LLC, Ann Arbor, MI},
      YEAR = {2003},
     PAGES = {106},
      ISBN = {978-0496-34011-8},
   MRCLASS = {99-05},
  MRNUMBER = {2704495},
       IgnoreURL = {http://gateway.proquest.com/openurl?url_ver=Z39.88-2004&rft_val_fmt=info:ofi/fmt:kev:mtx:dissertation&res_dat=xri:pqdiss&rft_dat=xri:pqdiss:3086205},
}

@article{dries2000field,
 AUTHOR = {{\noopsort{dries}van den Dries}, Lou and Speissegger, Patrick},
     TITLE = {The field of reals with multisummable series and the
              exponential function},
   JOURNAL = {Proc. London Math. Soc. (3)},
  FJOURNAL = {Proceedings of the London Mathematical Society. Third Series},
    VOLUME = {81},
      YEAR = {2000},
    NUMBER = {3},
     PAGES = {513--565},
      ISSN = {0024-6115,1460-244X},
   MRCLASS = {03C64 (03C10 12L12 26E05)},
  MRNUMBER = {1781147},
MRREVIEWER = {Chris\ Miller},
       DOI = {10.1112/S0024611500012648},
       URL = {https://doi.org/10.1112/S0024611500012648},
}

@article{dries2002minimal,
  AUTHOR = {{\noopsort{dries}van den Dries}, L. and Speissegger, P.},
     TITLE = {O-minimal preparation theorems},
 BOOKTITLE = {Model theory and applications},
    SERIES = {Quad. Mat.},
    VOLUME = {11},
     PAGES = {87--116},
 PUBLISHER = {Aracne, Rome},
      YEAR = {2002},
      ISBN = {88-7999-411-5},
   MRCLASS = {03C64 (03C10 32B05)},
  MRNUMBER = {2159715},
MRREVIEWER = {Marcus\ Tressl},
}

@article {freni2023vector,
    AUTHOR = {Freni, Pietro},
     TITLE = {On {V}ector {S}paces with {F}ormal {I}nfinite {S}ums},
   JOURNAL = {Appl. Categ. Structures},
  FJOURNAL = {Applied Categorical Structures. A Journal Devoted to
              Applications of Categorical Methods in Algebra, Analysis,
              Computer Science, Logic, Order and Topology},
    VOLUME = {34},
      YEAR = {2026},
    NUMBER = {2},
     PAGES = {Paper No. 15},
      ISSN = {0927-2852,1572-9095},
   MRCLASS = {13F25 (08A65 13J05 18B15 18F60)},
  MRNUMBER = {5039691},
       DOI = {10.1007/s10485-025-09844-w},
       URL = {https://doi.org/10.1007/s10485-025-09844-w},
}

@article{marker1994definable,
AUTHOR = {Marker, David and Steinhorn, Charles I.},
     TITLE = {Definable types in {O}-minimal theories},
   JOURNAL = {J. Symbolic Logic},
  FJOURNAL = {The Journal of Symbolic Logic},
    VOLUME = {59},
      YEAR = {1994},
    NUMBER = {1},
     PAGES = {185--198},
      ISSN = {0022-4812,1943-5886},
   MRCLASS = {03C45},
  MRNUMBER = {1264974},
MRREVIEWER = {Alexandre\ Ivanov},
       DOI = {10.2307/2275260},
       URL = {https://doi.org/10.2307/2275260},
}

@article{kaplan2023t,
    AUTHOR = {Kaplan, Elliot},
     TITLE = {{$T$}-convex {$T$}-differential fields and their immediate
              extensions},
   JOURNAL = {Pacific J. Math.},
  FJOURNAL = {Pacific Journal of Mathematics},
    VOLUME = {320},
      YEAR = {2022},
    NUMBER = {2},
     PAGES = {261--298},
      ISSN = {0030-8730,1945-5844},
   MRCLASS = {03C64 (12H05 12J10)},
  MRNUMBER = {4548312},
MRREVIEWER = {Athipat\ Thamrongthanyalak},
       DOI = {10.2140/pjm.2022.320.261},
       URL = {https://doi.org/10.2140/pjm.2022.320.261},
}

@misc{freni2024t,
      title={{$T$}-convexity, Weakly Immediate Types and {$T$-$\lambda$}-Spherical Completions of o-minimal Structures}, 
      author={Pietro Freni},
      year={2024},
      eprint={2404.07646},
      archivePrefix={arXiv},
      primaryClass={math.LO},
      url={https://arxiv.org/abs/2404.07646}, 
}

@misc{rolin2024transasymptotic,
      title={Transasymptotic expansions of o-minimal germs}, 
      author={Jean-Philippe Rolin and Tamara Servi and Patrick Speissegger},
      year={2024},
      eprint={2402.12073},
      archivePrefix={arXiv},
      primaryClass={math.LO},
      url={https://arxiv.org/abs/2402.12073}, 
}

@article {ehrlich2021surreal,
    AUTHOR = {Ehrlich, Philip and Kaplan, Elliot},
     TITLE = {Surreal ordered exponential fields},
   JOURNAL = {J. Symb. Log.},
  FJOURNAL = {The Journal of Symbolic Logic},
    VOLUME = {86},
      YEAR = {2021},
    NUMBER = {3},
     PAGES = {1066--1115},
      ISSN = {0022-4812,1943-5886},
   MRCLASS = {06A05 (03C64 06F20 06F25 12J15)},
  MRNUMBER = {4347569},
MRREVIEWER = {Ricardo\ Bianconi},
       DOI = {10.1017/jsl.2021.59},
       URL = {https://doi.org/10.1017/jsl.2021.59},
}

@misc{fornasiero2013initial,
  title={Initial Embeddings in the Surreal Numbers of Models of {$T_{an}(exp)$}},
  author={Fornasiero, Antongiulio},
  year={2013},
  publisher={Citeseer},
  url={https://citeseerx.ist.psu.edu/document?repid=rep1&type=pdf&doi=6fdc74d61b4acd866cfc96fda9ddd4efc634ac26}
}

@article {rolin2015quantifier,
    AUTHOR = {Rolin, J.-P. and Servi, T.},
     TITLE = {Quantifier elimination and rectilinearization theorem for
              generalized quasianalytic algebras},
   JOURNAL = {Proc. Lond. Math. Soc. (3)},
  FJOURNAL = {Proceedings of the London Mathematical Society. Third Series},
    VOLUME = {110},
      YEAR = {2015},
    NUMBER = {5},
     PAGES = {1207--1247},
      ISSN = {0024-6115,1460-244X},
   MRCLASS = {32B20 (03C64 14P15 32C05)},
  MRNUMBER = {3349791},
MRREVIEWER = {Chris\ Miller},
       DOI = {10.1112/plms/pdv010},
       URL = {https://doi.org/10.1112/plms/pdv010},
}

@phdthesis{schmeling2001corps,
  title={Corps de transs{\'e}ries},
  author={Schmeling, Michael Ch},
  year={2001},
  school={Paris 7}
}

@article {dries1986generalization,
    AUTHOR = {{\noopsort{dries}van den Dries}, Lou},
     TITLE = {A generalization of the {T}arski-{S}eidenberg theorem, and
              some nondefinability results},
   JOURNAL = {Bull. Amer. Math. Soc. (N.S.)},
  FJOURNAL = {American Mathematical Society. Bulletin. New Series},
    VOLUME = {15},
      YEAR = {1986},
    NUMBER = {2},
     PAGES = {189--193},
      ISSN = {0273-0979,1088-9485},
   MRCLASS = {03C40 (03E47)},
  MRNUMBER = {854552},
MRREVIEWER = {Bienvenido\ F.\ Nebres},
       DOI = {10.1090/S0273-0979-1986-15468-6},
       URL = {https://doi.org/10.1090/S0273-0979-1986-15468-6},
}

@unpublished{wreo36168,
           month = {September},
            note = {Unpublished},
          school = {University of Leeds},
           title = {Structural Investigations in some Classes of o-minimal Fields},
            year = {2024},
             url = {https://etheses.whiterose.ac.uk/id/eprint/36168/},
          author = {Freni, Pietro}
}

@incollection {dries2014truncation,
    AUTHOR = {{\noopsort{dries}van den Dries}, Lou},
     TITLE = {Truncation in {H}ahn fields},
 BOOKTITLE = {Valuation theory in interaction},
    SERIES = {EMS Ser. Congr. Rep.},
     PAGES = {579--595},
 PUBLISHER = {Eur. Math. Soc., Z\"urich},
      YEAR = {2014},
      ISBN = {978-3-03719-149-1},
   MRCLASS = {12J99 (06F20)},
  MRNUMBER = {3329048},
MRREVIEWER = {Dikran\ Dikranjan},
}

@article {fornasiero2006embedding,
    AUTHOR = {Fornasiero, Antongiulio},
     TITLE = {Embedding {H}enselian fields into power series},
   JOURNAL = {J. Algebra},
  FJOURNAL = {Journal of Algebra},
    VOLUME = {304},
      YEAR = {2006},
    NUMBER = {1},
     PAGES = {112--156},
      ISSN = {0021-8693,1090-266X},
   MRCLASS = {12J10 (12J15 12J20)},
  MRNUMBER = {2255823},
MRREVIEWER = {Guang\ Xing\ Zeng},
       DOI = {10.1016/j.jalgebra.2006.06.037},
       URL = {https://doi.org/10.1016/j.jalgebra.2006.06.037},
}

@book {fuchs1963partially,
    AUTHOR = {Fuchs, L.},
     TITLE = {Partially ordered algebraic systems},
 PUBLISHER = {Pergamon Press, Oxford; Addison-Wesley Publishing Co., Inc.,
              Reading, Mass.-Palo Alto, Calif.-London},
      YEAR = {1963},
     PAGES = {ix+229},
   MRCLASS = {06.00 (20.00)},
  MRNUMBER = {171864},
MRREVIEWER = {P.\ F.\ Conrad},
}

@article {kruskal1960well,
    AUTHOR = {Kruskal, J. B.},
     TITLE = {Well-quasi-ordering, the {T}ree {T}heorem, and {V}azsonyi's
              conjecture},
   JOURNAL = {Trans. Amer. Math. Soc.},
  FJOURNAL = {Transactions of the American Mathematical Society},
    VOLUME = {95},
      YEAR = {1960},
     PAGES = {210--225},
      ISSN = {0002-9947,1088-6850},
   MRCLASS = {06.00},
  MRNUMBER = {111704},
MRREVIEWER = {Graham\ Higman},
       DOI = {10.2307/1993287},
       URL = {https://doi.org/10.2307/1993287},
}

@article {rado1954partial,
    AUTHOR = {Rado, R.},
     TITLE = {Partial well-ordering of sets of vectors},
   JOURNAL = {Mathematika},
  FJOURNAL = {Mathematika. A Journal of Pure and Applied Mathematics},
    VOLUME = {1},
      YEAR = {1954},
     PAGES = {89--95},
      ISSN = {0025-5793},
   MRCLASS = {27.2X},
  MRNUMBER = {66441},
MRREVIEWER = {Djuro\ Kurepa},
       DOI = {10.1112/S0025579300000565},
       URL = {https://doi.org/10.1112/S0025579300000565},
}

@book {heibsch1998semirings,
    AUTHOR = {Heibsch, U. and Weinert, H. J.},
     TITLE = {Semirings Algebraic theory and applications in computer science},
 PUBLISHER = {World Scientific},
      YEAR = {1998},
     PAGES = {ix+229},
}

@book {comtet1974advanced,
    AUTHOR = {Comtet, Louis},
     TITLE = {Advanced combinatorics},
   EDITION = {enlarged},
      NOTE = {The art of finite and infinite expansions},
 PUBLISHER = {D. Reidel Publishing Co., Dordrecht},
      YEAR = {1974},
     PAGES = {xi+343},
      ISBN = {90-277-0441-4},
   MRCLASS = {05-02},
  MRNUMBER = {460128},
}

@article {gessel2016lagrange,
    AUTHOR = {Gessel, Ira M.},
     TITLE = {Lagrange inversion},
   JOURNAL = {J. Combin. Theory Ser. A},
  FJOURNAL = {Journal of Combinatorial Theory. Series A},
    VOLUME = {144},
      YEAR = {2016},
     PAGES = {212--249},
      ISSN = {0097-3165,1096-0899},
   MRCLASS = {05A15 (05A19)},
  MRNUMBER = {3534068},
MRREVIEWER = {L\'aszl\'o\ T\'oth},
       DOI = {10.1016/j.jcta.2016.06.018},
       URL = {https://doi.org/10.1016/j.jcta.2016.06.018},
}

@incollection {mourgues1993transfinite,
    AUTHOR = {Mourgues, M.-H. and Ressayre, J.-P.},
     TITLE = {A transfinite version of {P}uiseux's theorem, with
              applications to real closed fields},
 BOOKTITLE = {Logic {C}olloquium '90 ({H}elsinki, 1990)},
    SERIES = {Lecture Notes Logic},
    VOLUME = {2},
     PAGES = {250--258},
 PUBLISHER = {Springer, Berlin},
      YEAR = {1993},
      ISBN = {3-540-57094-2},
   MRCLASS = {12E05 (03C60 03F65 12D15)},
  MRNUMBER = {1279844},
}

@article {delon1991indecidibilite,
    AUTHOR = {Delon, Fran\c coise},
     TITLE = {Ind\'ecidabilit\'e{} de la th\'eorie des paires imm\'ediates
              de corps valu\'es henseliens},
   JOURNAL = {J. Symbolic Logic},
  FJOURNAL = {The Journal of Symbolic Logic},
    VOLUME = {56},
      YEAR = {1991},
    NUMBER = {4},
     PAGES = {1236--1242},
      ISSN = {0022-4812,1943-5886},
   MRCLASS = {03C60 (03C65 03D35 12J10 12L12)},
  MRNUMBER = {1136453},
MRREVIEWER = {Fran\c coise\ Point},
       DOI = {10.2307/2275471},
       URL = {https://doi.org/10.2307/2275471},
}

@misc{dries2025truncation,
      title={Truncation Structures}, 
      author={{\noopsort{dries}van den Dries}, Lou},
      year={2025},
      eprint={2512.22640},
      archivePrefix={arXiv},
      primaryClass={math.LO},
      url={https://arxiv.org/abs/2512.22640}, 
}

\end{document}